\documentclass[11pt]{amsart}

\usepackage{epsfig,amsmath,amsfonts,amssymb,latexsym}
\usepackage[abs]{overpic}
\usepackage{color}

\usepackage{palatino} 
\usepackage{xcolor}
\definecolor{indigo}{rgb}{0.29, 0.0, 0.51}
\usepackage[colorlinks, urlcolor=indigo, linkcolor=indigo, citecolor=indigo]{hyperref}

\usepackage[nocompress]{cite}
\usepackage{nicefrac,xfrac}
\usepackage{bm}

\newtheorem{thm}{Theorem}[section]
\newtheorem{prop}[thm]{Proposition}
\newtheorem{lem}[thm]{Lemma}
\newtheorem{cor}[thm]{Corollary}

\newtheorem{question}[thm]{Question}

\theoremstyle{definition}

\theoremstyle{remark}
\newtheorem{remark}[thm]{Remark}
\numberwithin{equation}{section}

\def\dfn#1{{\em #1}} 

\newcommand{\Z}{\mathbb{Z}}  
\newcommand{\Q}{\mathbb{Q}}  
\newcommand{\N}{\mathbb{N}}  
\newcommand{\bd}{\partial}  
\DeclareMathOperator{\interior}{int}  

\newcommand{\vects}[2]{\left(\begin{smallmatrix} #1 \\ #2 \end{smallmatrix}\right)} 
\newcommand{\matrixb}[4]{{\left(\begin{matrix}#1 & #2 \\ #3 & #4\end{matrix} \right)}} 

\newcommand{\HFhat}{\widehat{{HF}}} 

\newcommand{\modp}[1]{\,(\!\!\!\!\mod #1)}
\newcommand{\case}[2]{\begin{cases}#1 \\ #2 \end{cases}}

\begin{document}

\begin{abstract}
Two of the basic questions in contact topology are which manifolds admit tight contact structures, and on those that do, can we classify such structures.  We present the first such classification on an infinite family of (mostly) hyperbolic $3$--manifolds: surgeries on the figure-eight knot.  We also determine which of the tight contact structures are symplectically fillable and which are universally tight.
\end{abstract}

\title[Tight Contact Structures on Surgeries on the Figure-Eight Knot]{Classification of Tight Contact Structures on \\ Surgeries on the Figure-Eight Knot}
\author{James Conway}
\author{Hyunki Min}
\address{Department of Mathematics \\ University of California, Berkeley \\ Berkeley \\ California}
\email{conway@berkeley.edu}
\address{School of Mathematics \\ Georgia Institute
of Technology \\  Atlanta  \\ Georgia}
\email{hmin38@gatech.edu}
\maketitle

\section{Introduction}

Ever since Eliashberg distinguished overtwisted from tight contact structures in dimension 3 \cite{Eliashberg:OT}, there has been an ongoing project to determine which closed, oriented $3$--manifolds support a tight contact structure, and on those that do, whether we can classify them.  Much work has been done: reducing the problem to prime manifolds \cite{Colin:connectsum}; determining that there are infinitely many tight contact structures if and only if the manifold is toroidal \cite{CGH, HKM:convex}; showing the existence of a tight contact structure for manifolds that support a taut foliation \cite{ET} and for those with $b_2 > 0$ \cite{Gabai, ET}; and the existence and classification on many small Seifert fibered spaces \cite{Honda:classification1, Tosun:SFS, LS:seifert, Ghiggini:seifert, GLS:e0positive, GLS:someseifert, Wu, EH:non-existence, Matkovic} and some fiber bundles \cite{Giroux:classification, Giroux:bundles, Honda:classification2}.

Although there have been isolated constructions --- mostly via contact surgery --- of tight contact structures on hyperbolic manifolds, there has not appeared any full classification result. The only partial result is a classification of {\em extremal} tight contact structures on hyperbolic surface bundles of genus $g > 1$ by Honda, Kazez, and Mati\'c \cite{HKM:extremal}: contact structures whose Euler class evaluates on the fiber to $\pm (2g-2)$. However, there might be other tight contact structures (and in some cases we know that indeed there are others). In work to appear, \cite{Min}, the second author will present a full classification of tight contact structures on the Weeks manifold.

In this paper, we present a full classification of tight contact structures for an infinite family of hyperbolic $3$--manifolds: those coming from Dehn surgery on the figure-eight knot in $S^3$.  Let $r$ be a rational number, let $K$ be the figure-eight knot in $S^3$, and let $M(r)$ be the manifold resulting from smooth $r$--surgery on $K$. Since $K$ is amphicheiral, $M(-r)$ is diffeomorphic to $-M(r)$. A classic result of Thurston \cite{Thurston} is that $M(r)$ is hyperbolic except when $r \in \{0, \pm1, \pm2, \pm3, \pm4\}$.  When $r = 0$ or $r = \pm 4$, there is an incompressible torus in $M(r)$, and hence by \cite{CGH}, there exist infinitely many tight contact structures on $M(r)$, and $M(0)$ is a torus-bundle over $S^1$, whose tight contact structures are classified by Honda \cite{Honda:classification2}.  When $r = \pm1, \pm2, \pm3$, the manifold $M(r)$ is a small Seifert fibered space. Of these cases, the classifications have already been done in two cases: $M(1)$ is diffeomorphic to the Brieskorn sphere $\Sigma(2,3,7)$, which was found to have exactly two tight contact structures by Mark and Tosun \cite[Section~2]{MT:pseudoconvex}, and $M(-1)$ is diffeomorphic to the small Seifert fibered space $M(-2;\frac12,\frac23,\frac67)$, which has a unique tight contact structure, as shown by Tosun \cite[Theorem~1.1(3)]{Tosun:SFS}.

In this paper, we classify tight contact structures on $M(r)$ up to isotopy for surgery coefficients
\[ r \in \mathcal R = \left((-\infty, -4) \cup [-3, 0)\cup [1, 4)\cup [5, \infty)\right)\cap \Q. \]  Before stating the results, we define two functions $\Phi(r)$ and $\Psi(r)$.  For $r = \frac pq \in (0, 1] \cap \Q$, write the negative continued fraction of $-\frac qp$, that is:
\[ -\frac qp = [r_0, \ldots, r_n] = r_0 - \frac1{r_1 - \frac1{\ddots \, - \frac1{r_n}}}, \]
where $r_0 \leq -1$, and $r_i \leq -2$ for $i = 1, \ldots, n$. Then define
\[ \Phi(r) = \left|r_0(r_1+1)\cdots(r_n+1)\right|.\]
We define $\Phi$ on all of $\Q$ by setting $\Phi(r+1) = \Phi(r)$.  Then we define $\Psi(r)$ for $r \in \Q$ by
\[ \Psi(r) = \left\{
\begin{array}{ll}
0, & r \geq -3 \\
\Phi(-\frac1{r+3}), & r < -3.
\end{array}
\right. \]
With these functions in hand, we can state the main theorem of this paper.

\begin{thm} \label{main classification}
Let $r \in \mathcal R$ be a rational number.  Then $M(r)$ supports
\[ \left\{
\begin{array}{ll}
2\Phi(r), & r \in [1, 4) \cup [5, \infty) \\
\Phi(r) + \Psi(r), & r \in (-\infty,-4)\cup[-3,0)
\end{array}
\right. \]
tight contact structures up to isotopy, distinguished by their Heegaard Floer contact classes.
\end{thm}

\begin{remark}
For $r \not\in \mathcal R$, not equal to $0$ or $\pm 4$, we can construct $2\Phi(r)$ (respectively, $\Phi(r) + \Psi(r)$) tight contact structures, when $r > 0$ (respectively, $r < 0$), but we believe that there are also others. We hope to return to this in a future paper.
\end{remark}

The tight contact structures counted by $\Psi$ arise via negative contact surgery on Legendrian figure-eight knots in $(S^3, \xi_{\rm{std}})$. Those counted by $\Phi$ arise via negative contact surgery on Legendrian figure-eight knots in $(S^3, \xi_1^{OT})$, where $\xi_1^{OT}$ is the unique overtwisted contact structure on $S^3$ that has (normalized) $3$--dimensional homotopy invariant equal to 1 (see \cite{Gompf} for details on this invariant).

Combining Theorem~\ref{main classification} with the fact that $M(0)$ and $M(\pm 4)$ are toroidal, we can give the following simple counts for integer surgeries.

\begin{cor} \label{integer classification}
Let $n$ be an integer. Then, $M(n)$ supports
\[ \left\{
\begin{array}{ll}
\infty, & n = 0, \pm 4 \\
2, & n > 0, n \neq 4 \\
1, & n = -1, -2, -3 \\
|n|-2, & n \leq -5
\end{array}
\right. \]
tight contact structure(s) up to isotopy, distinguished by their Heegaard Floer contact classes.
\end{cor}

We complete the classification in Theorem~\ref{main classification} by estimating an upper bound using convex surface decompositions (see Section~\ref{sec:convex}), and realize this upper bound by constructing tight contact structures via contact surgery. We use the contact class in Heegaard Floer homology to distinguish the contact structures that we construct.  Unlike many of the classification results cited above, our surgery constructions start both from tight contact structures and from an overtwisted contact structure on $S^3$.

Our analysis of the upper bound is possible due to the following facts: (1) the figure-eight knot is a fibered knot; (2) the monodromy of the fibration is pseudo-Anosov; (3) the genus of the fiber surface is one.
In particular, our approach will work for other genus--$1$ fibered knots with pseudo-Anosov monodromy (necessarily, knots in manifolds other than $S^3$). Since the complicated combinatorics precludes a general solution at this juncture, we have focused on a single knot, the figure-eight knot, but hope to return to more general questions in future work ({\em cf.\@} Question~\ref{general-high-surgery}).

In addition to classifying the tight contact structures, we also wish to understand what properties these contact structures have. Namely, we are interested in whether they are symplectically fillable (in any sense) and whether they are universally tight or virtually overtwisted.

\begin{thm} \label{fillable}
For $r \in \mathcal R$, all tight contact structures on $M(r)$ are strongly symplectically fillable. In addition, the following tight contact structures are Stein fillable:
\begin{itemize}
\item All tight contact structures counted by $\Psi$.
\item All tight contact structures, when $r \geq -9$ and $r \in \mathcal R$.
\end{itemize}
\end{thm}

\begin{remark}
We leave it as an open question whether the tight contact structures counted by $\Phi$ on $M(r <  -9)$ are Stein fillable.
\end{remark}

\begin{thm} \label{universally tight}
Let $r \in \mathcal R$. Then the number of universally tight contact structures that $M(r)$ supports is exactly
\[
\begin{cases}
	1, & \mbox{when } r < 0\mbox{ and } r \in \Z, \\
	2, & \mbox{when } r < 0\mbox{ and } r\not\in \Z\mbox{, or }r > 0\mbox{ and }r\in\Z.
\end{cases}
\] 
When $r > 0$ and $r \not\in \Z$, then $M(r)$ supports either $2$ or $4$ universally tight contact structures.
\end{thm}

When $n$ is an integer, the universally tight contact structure on $M(n \leq -5)$ is the one counted by $\Phi$ (recall that $\Phi(n) = 1$ for any integer $n$).  In Section~\ref{sec:properties}, we describe for general $r \in \mathcal R$ which contact structures are the universally tight ones.

We have the following existence result for contact structures on hyperbolic homology spheres. This family is given by $M(-\frac1n)$, for $n > 1$.

\begin{cor} \label{reciprocal classification}
There is an infinite family of hyperbolic integer homology spheres that admit exactly two contact structures up to isotopy, both of which are Stein fillable and universally tight (in fact, they are contactomorphic).
\end{cor}

As part of our proofs, we deal with a large number of \dfn{non-loose} (or \dfn{exceptional}) figure-eight knots in overtwisted contact structures in $S^3$.  These are Legendrian knots in overtwisted contact manifolds, but where the complement of a standard neighborhood is tight.  (If the complement were instead overtwisted, we would call the knot \dfn{loose}.) Although our work does not lend itself immediately to a classification of non-loose figure-eight knots, several properties of such knots can be extracted from our results.  For example: any non-loose figure-eight knot $L$ with $tb(L) \not\in \{-3, 0, 1, 5\}$ destabilizes. This follows from the convex surface theory analysis in Section~\ref{sec:upperbound}.

Finally, it is interesting to compare our results to existing classification results for surgeries on knots in $S^3$, namely, to surgeries on torus knots. Let $K$ be a positive torus knot or the negative trefoil, let $n$ be a positive integer, and let $r \in [0, 1)$ be rational. It follows from the classification results that Tosun proved in \cite{Tosun:SFS} that the number of isotopy classes of tight contact structures on $S^3_{n+r}(K)$ is independent of $n$, for $n$ sufficiently large (how large depends on the torus knot in question). This is a property shared by the figure-eight knot, as we see in Theorem~\ref{main classification}, which leads us to ask:

\begin{question} \label{general-high-surgery}
If $K \subset S^3$ is a fibered knot, and $r \in [0, 1)$ is rational, is the number of isotopy classes of tight contact structures on $S^3_{n+r}(K)$ independent of $n \in \N$ for sufficiently large $n$?
\end{question}

\subsection{Organization of Paper}

After going over the requisite background material in Section~\ref{sec:background}, we turn to the calculation of the upper bound for the number of tight contact structures in Section~\ref{sec:upperbound}. Then in Section~\ref{sec:lowerbound}, we construct as many distinct tight contact structures as we calculated might appear, and in Section~\ref{sec:properties}, we determine which are fillable and which are universally tight.

\subsection{Acknowledgements}

The authors thank John Etnyre and B\"ulent Tosun for many helpful conversations. The first author was partially supported by NSF grant DMS-1344991; the second author was partially supported by NSF grant DMS-1608684.

\section{Contact Geometry Background}
\label{sec:background}

We assume that the reader has an understanding of the basic definitions and results in $3$--dimensional contact geometry \cite{Geiges:book}, including Legendrian knots and their invariants \cite{Etnyre:knots}, Legendrian surgery \cite{Eliashberg:stein}, open book decompositions \cite{Etnyre:OBlectures}, convex surface theory \cite{Etnyre:convex}, and Heegaard Floer homology \cite{OS:hf1, OS:hf2, OS:contact}.  We suffice ourselves with citing more recent results, along with classical results that we will make use of often.

\subsection{Convex Surface Theory}
\label{sec:convex}

As mentioned above, we assume the reader is familiar with convex surfaces, and bring only results that we will repeatedly cite.  We extend a word of warning to the reader who is familiar with the conventions from convex surface theory: our convention (and the common one today) is that slopes are given by $\frac{\rm{meridian}}{\rm{longitude}}$, and not the inverse. This has led to some differences between the way we cite results and the way they were originally presented.

We will often make use of several properties of convex surfaces without explicitly mentioning them, to make it easier on the reader to follow the argument: perturbing a surface to be convex, realizing a particular foliation that is divided by the given dividing set, and in particular, using the Legendrian realization principle.  We assume throughout that the boundary of any convex surface $\Sigma$ is Legendrian, if non-empty, and we denote the dividing set of $\Sigma$ by $\Gamma_\Sigma$.  When $\bd \Sigma$ is a single knot, then $\bd \Sigma$ is null-homologous, and we can measure $tb(\bd \Sigma)$. Kanda \cite{Kanda:tb} proved that
\[
	tb(\bd \Sigma) = -\frac{1}{2}\left|\bd \Sigma \cap \Gamma_\Sigma\right|.
\]
If $\bd \Sigma$ is oriented as the boundary of $\Sigma$, then $rot(\bd \Sigma) = \chi(\Sigma_+) - \chi(\Sigma_-)$, where $\Sigma_\pm$ are the positive/negative regions of the convex surface.  If $\Sigma$ is a convex surface with Legendrian boundary properly embedded in a contact manifold with non-empty convex boundary, then the  {\em relative Euler class} of the contact structure evaluates to $\chi(\Sigma_+) - \chi(\Sigma_-)$ on $\Sigma$.

Giroux \cite{Giroux:convex} proved that if $\Sigma$ is a convex surface with a tight neighborhood, then if $\Sigma \not\cong S^2$, then no component of $\Gamma_\Sigma$ is contractible, and if $\Sigma \cong S^2$, then $\Gamma_\Sigma$ is a single closed curve.

Let $\Sigma_0$, $\Sigma_1$ be convex surfaces with Legendrian boundary, and assume they have a common boundary component $L$ along which they intersect transversely.
In this case, Honda showed in \cite{Honda:classification1} that $\Gamma_{\Sigma_i}$ look like the left-hand picture in Figure~\ref{edge-rounding}. We can consider $\Sigma_0 \cup \Sigma_1$ as a surface with a corner, and ``round the corner'' of $\Sigma_0 \cup \Sigma_1$ to obtain a smooth convex surface, see the right-hand picture in Figure~\ref{edge-rounding}.

\begin{figure}[htbp]
	\begin{center}
	\begin{overpic}[scale=1,tics=20]{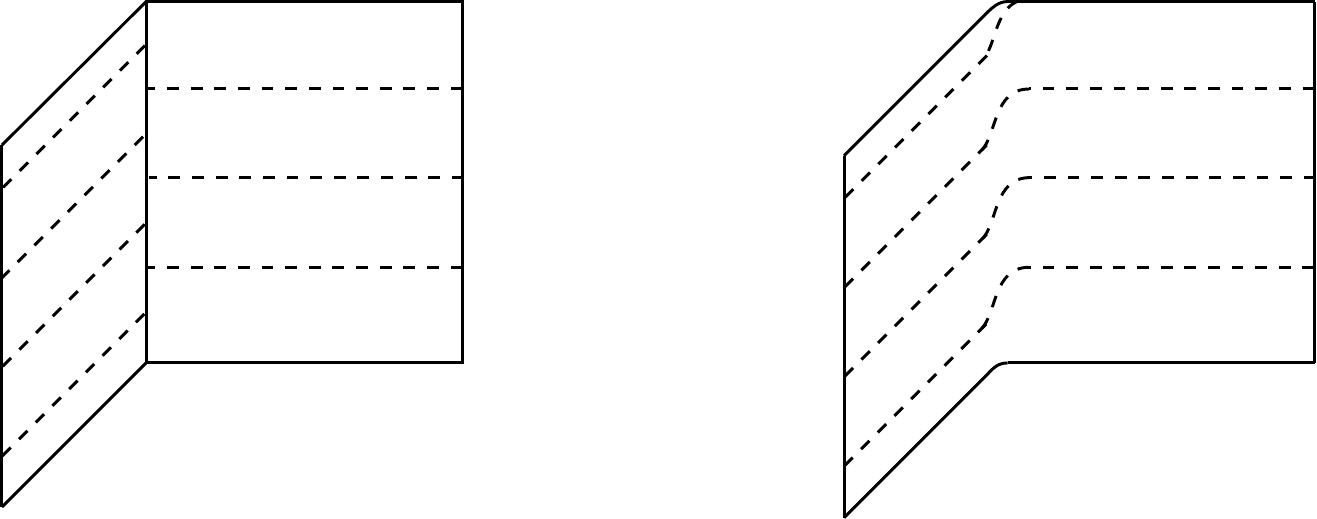}
	\end{overpic}
	\caption{The left-hand picture is prior to edge-rounding. The right-hand picture is after edge-rounding.}
	\label{edge-rounding}
	\end{center}
 \end{figure}
 
There is another way to modify a convex surface, due to Honda \cite{Honda:classification1}: attaching a \dfn{bypass}. A bypass is a particular type of a convex surface; in fact, it is half of an overtwisted disc. Consider a convex overtwisted disc whose dividing set consists of one closed curve. Take an arc in the disc whose endpoints lie on the boundary and that intersects the dividing curve in two points. Apply the Legendrian realization principle to realize the arc as a Legendrian arc $\gamma$, and cut the disc along $\gamma$; each half-disc is called a bypass. Now, suppose a bypass $D$ intersects a convex surface $\Sigma$ such that $D \cap \Sigma = \gamma$. Since the dividing sets interleave, $\gamma \cap \Gamma_\Sigma$ is three points. We call such a Legendrian arc $\gamma$ on a surface an \dfn{attaching arc}. After edge-rounding, the convex boundary of a neighborhood of $D \cup \Sigma$ is a surface isotopic to $\Sigma$ but with its dividing set changed in a neighborhood of the attaching arc as in Figure~\ref{bypass-attachment}. We call this process a \dfn{bypass attachment along $\gamma$}. Note that Figure~\ref{bypass-attachment} is drawn for the case that the bypass $D$ is attached ``from the front'', that is, sitting above the page; if we attach a bypass ``from the back'' of $\Sigma$, the result will be Figure~\ref{bypass-attachment} reflected over a vertical line.

\begin{figure}[htbp]
	\begin{center}
	\begin{overpic}[scale=1,tics=20]{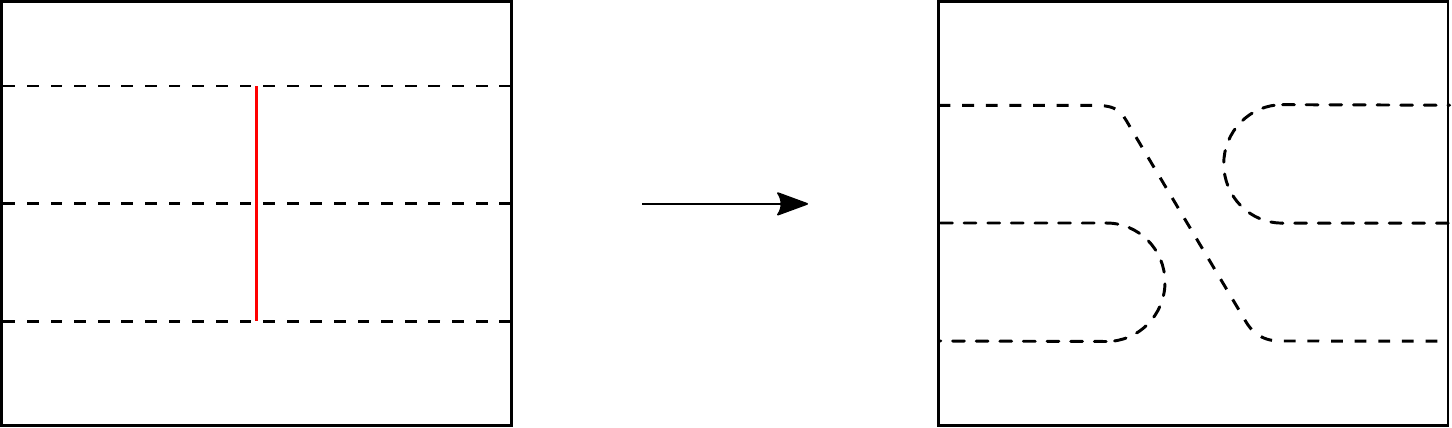}
	\put(78,78){$\gamma$}
	\end{overpic}
	\caption{How the dividing set (dashed lines) changes in a neighborhood of the attaching arc (solid red line) for a bypass attachment.}
	\label{bypass-attachment}
	\end{center}
 \end{figure}
 
There are several results about bypasses which we will make frequent use of:

\begin{thm}[Bypass Rotation, Honda--Kazez--Mati\'c \cite{HKM:pinwheels}] \label{bypass-rotation} Suppose that there is a bypass for $\Sigma$ from the front along an attaching arc $\gamma$. If $\gamma'$ is a Legendrian arc as in Figure~\ref{fig:bypass-rotation}, then there exists a bypass for $\Sigma$ from the front along $\gamma'$.
\end{thm}

\begin{thm}[Bypass Sliding, Honda \cite{Honda:classification1}] \label{bypass-sliding}
	Suppose that there is a bypass for $\Sigma$ from the front along an attaching arc $\gamma$. If $\gamma'$ is a Legendrian arc that is isotopic to $\gamma$ relative to $\Gamma_\Sigma$, then there is a bypass for $\Sigma$ from the front along $\gamma'$.
\end{thm}

\begin{figure}[htbp]
\begin{center}
\begin{overpic}[scale=1,tics=20]{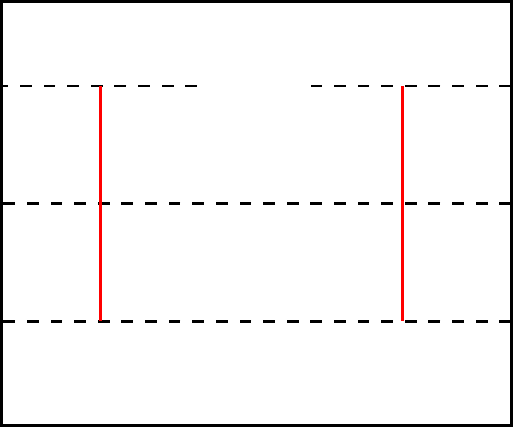}
\put(120,78){$\gamma$}
\put(32,78){$\gamma'$}
\end{overpic}
\caption{The attaching arc for a bypass rotation.}
\label{fig:bypass-rotation}
\end{center}
\end{figure}
 
Consider now a convex torus $T$ with dividing set $\Gamma$. Choose coordinates for $H_1(T)$, and let $\vects{p}{q}$ be written as $\frac pq$.  Let $\gamma$ be an attaching arc for $T$ from the front.  If $\gamma$ intersects three different dividing curves, then it is easy to see that the bypass attachment will reduce $\left|\Gamma\right|$. If a bypass intersects exactly two dividing curves, then Honda completely studied what happens. In order to describe the important case, we need to briefly give our conventions for the Farey graph. Recall, the Farey graph is the 1-skeleton of the triangulation of the disc model of the hyperbolic plane by ideal triangles, whose vertices are labeled by rational numbers; $\frac pq$ and $\frac rs$ are connected via an edge if $\vects{p}{q}$ and $\vects{r}{s}$ is a basis for $\Z^2$, and $0$, $-1$, and $1$ appear in counter-clockwise order around the circle.

\begin{thm}[Honda \cite{Honda:classification1}] \label{bypass on torus}
	Suppose a convex torus $T$ has two dividing curves of slope $s$, and let $\gamma$ be an attaching arc of slope $r$ of a bypass for $T$ from the front (along a ruling curve for $T$). Let $T'$ be the convex torus obtained from $T$ by attaching a bypass along $\gamma$. Then $\Gamma_{T'}$ consists of two dividing curves of slope $s'$, where $s'$ is the label on the Farey graph that is the furthest clockwise of $s$ and counter-clockwise of $r$ that is connected to $s$ by an edge (and if $s$ and $r$ are connected by an edge, then $s' = r$).
\end{thm}

In particular, adding a bypass from the front changes the slope clockwise, whereas adding a bypass from the back changes the slope counter-clockwise.  All of our discussions above are only useful when we can find a bypass. In general, we can construct a bypass by finding a boundary-parallel dividing arc along a transverse surface. Honda described situation in which we can apply this method. Note that when we say a dividing arc, it means a properly embedded arc.

\begin{thm}[Honda \cite{Honda:classification1}] \label{imba} Let $\Sigma$ be a convex surface.
	\begin{itemize}
		\item Suppose $A = S^1 \times [0,1]$ is a convex surface with Legendrian boundary and $\Sigma \cap A = S^1 \times \{0\}$. If $\left|\Gamma_A \cap \Sigma\right| > \left|\Gamma_A \cap S^1 \times \{1\}\right|$, then there exists a bypass for $\Sigma$ along an attaching arc in $S^1 \times \{0\}$.
		\item Suppose $D$ is a convex disc with Legendrian boundary and $\Sigma \cap D = \bd D$. If $tb(\bd D) < -1$, there exists a bypass for $\Sigma$ along an attaching arc in $\bd D$.
		\item Suppose $S$ is a convex surface with Legendrian boundary and $\Sigma \cap S = \bd S$. Assume further that $S$ is not a disc. If there exists a boundary-parallel dividing arc, then there exists a bypass for $\Sigma$ along an attaching arc in $\bd S$.
	\end{itemize}
\end{thm}

\subsection{Classification results} In Section~\ref{sec:upperbound}, we will break up our manifold $M(r)$ into simpler pieces, and use some of the following classification results to get an upper bound on the number of tight contact structures $M(r)$ supports. Let $M = T^2 \times I$, and fix a singular foliation $\mathcal F$ on $\bd M$ that is divided by two dividing curves of slope $s_i$ on $T \times \{i\}$, where $s_0$ and $s_1$ are connected by an edge in the Farey graph.  Say that a contact structure on $M$ is \dfn{minimally twisting} if for any convex boundary-parallel torus $T \subset M$, the slope of $\Gamma_T$ is clockwise of $s_0$ and counter-clockwise of $s_1$ on the Farey graph.

\begin{thm}[Honda \cite{Honda:classification1}] \label{basic-slice}
If $M = T^2 \times I$ and $\mathcal F$ are as above, then there are exactly two minimally twisting tight contact structures on $M$ that induce $\mathcal F$ on the boundary, up to isotopy fixing $\mathcal F$.
\end{thm}

The two different contact structures given by Theorem~\ref{basic-slice} can be distinguished by their relative Euler class, and after picking an orientation, we call them \dfn{positive} and \dfn{negative basic slices}.

Let $M = S^1 \times D^2$ and choose coordinates for $H_1(\bd M)$ such that $0$ describes a positively oriented longitude and $\infty$ is a positive meridian. Fix a singular foliation $\mathcal F$ on $\bd M$ that is divided by two dividing curves of slope $r = \frac pq$.  Let $k$ be the unique integer such that $\frac{p+kq}{q} \in [-1,0)$, and let $\frac q{p+kq}$ have the negative continued fraction decomposition $[r_0, \ldots, r_n]$, where $r_n \leq -1$, and $r_i \leq -2$ for $i=0,\ldots,n-1$.

\begin{thm}[Honda \cite{Honda:classification1}] \label{solid torus}
If $M = S^1 \times D^2$ and $\mathcal F$ are as above, then there are exactly \[ \left|(r_0+1)\cdots(r_{n-1}+1)r_n\right| \] tight contact structures on $M$ that induce $\mathcal F$ on the boundary, up to isotopy fixing $\mathcal F$.
\end{thm}

As is explained in \cite{Honda:classification1}, each tight contact structure on $M$ is constructed by starting with a basic slice with dividing curve slopes $s_0 < s_1 = r$.  Then, we glue another basic slice onto the boundary component that has dividing curves of slope $s_0$ to lower $s_0$, and we repeat this until we arrive at a $T^2 \times I$ with dividing curves on $T^2 \times \{0\}$ of slope $n \in \Z$.  Finally, we glue on a solid torus to $T^2 \times \{0\}$. As we are looking for a tight contact structure on $M$, we must take care never to glue two basic slices of different signs to arrive at dividing curve slopes that are adjacent on the Farey graph, since this would give us an overtwisted contact structure, by Theorem~\ref{basic-slice}.

\subsection{Gluing contact structures and Giroux torsion} In Section~\ref{sec:upperbound}, we will construct contact structures by gluing two tight contact structures along torus boundaries. To show that the resulting contact structure is tight, Colin's theorem \cite{Colin:gluing} is useful, which states that under suitable conditions, gluing two universally tight contact structures results in a tight contact structure. However, it requires that the torus boundary must be pre-Lagrangian. We will make use of a version of this theorem for convex tori.

\begin{thm}[Honda--Kazez--Mati\'c \cite{HKM:convex}] \label{gluing} Let $(M,\xi)$ be an oriented, irreducible contact $3$--manifold and $T \subset M$ an incompressible convex torus. If $\xi|_{M \setminus T}$ is universally tight and there exists $T^2\times[-1,1]\subset M$ where $T=T^2\times\{0\}$, $\xi|_{T^2\times[-1,1]}$ is universally tight, and $T^2\times[-1,0]$ and $T^2\times[0,1]$ are both rotative, then $\xi$ is universally tight.
\end{thm}

A contact structure on $T^2 \times [0,1]$ is called {\em rotative} if there is a convex boundary-parallel torus with dividing slope different from the slope on the boundary. Recall that \dfn{half Giroux torsion} is an embedding of $T^2 \times I$ that is contactomorphic to $(T^2 \times [0, 2], \xi)$, where the contact structure defines a negative basic slice on $[0, 1]$ with slopes $s_0 = 0$ and $s_1 = 1$, and a negative basic slice on $[1, 2]$ with slopes $s_1 = 1$ and $s_2 = 0$. This contact structure is tight by Theorem~\ref{gluing}. For any $s \in \Q \cup \{\infty\}$, we can find a convex torus with dividing curves and Legendrian divides of slope $s$.

\subsection{Legendrian Knots and Surgery}
\label{sec:legendrianknots}

\subsubsection{Surgery on Legendrian knots} \label{sec:surgery} Given a Legendrian knot $L$ in a contact manifold $(M, \xi)$, we describe an algorithm --- due to Ding and Geiges \cite{DG:surgery} --- to turn a contact $(r)$--surgery on $L$ into a sequence of Legendrian surgeries on the components of a Legendrian link, for $r < 0$. First, write $r = [r_0, \ldots, r_n]$ as a negative continued fraction, as in the introduction.  Then, take a push-off of $L$, and stabilize it $|r_0+1|$ times: this is $L_0$. Then, for $i = 1, \ldots, n$, to get $L_i$, we take a push-off of $L_{i-1}$, and stabilize it $|r_i+2|$ times.  The result of Legendrian surgery on all of the $L_i$ is the same as contact $(r)$--surgery on $L$.  There are in general many choices of stabilization, each leading to potentially different contact structures.

\subsubsection{Stein cobordisms} A given topological cobordism $W$ from $M$ to $M'$ will in general support many different Stein structures, even assuming that it restricts to a given contact structure on $M$.  However, the following result allows us to understand the contact structure on $M'$ that arises from different Stein structures on $W$.

\begin{thm}[{Lisca--Mati\'c \cite[Theorem~1.2]{LM}, Plamenevskaya \cite[Theorem~2]{Plamenevskaya2004}, Simone \cite[Theorem~1.1]{Simone}}] \label{steincobordism}
Suppose $(M, \xi)$ has non-vanishing Heegaard Floer contact invariant, and $(W, J_i)$ is a Stein cobordism from $(M, \xi)$ to contact manifolds $(M', \xi'_i)$, for $i = 1, 2$. If $\xi'_1$ and $\xi'_2$ are isotopic, then the Spin$^c$ structures induced by $J_1$ and $J_2$ are isomorphic (and in particular, have the same $c_1$). Indeed, if $J_1$ and $J_2$ are non-isomorphic, then the Heegaard Floer contact invariants of $\xi'_i$ are distinct.
\end{thm}

\begin{remark} In fact, \cite[Theorem~1.1]{Simone} only requires the non-vanishing of a particular twisted Heegaard Floer contact invariant, but what we have stated is sufficient for our purposes.\end{remark}

\subsection{Open Book Decompositions}
\label{sec:obd}

In this section, we describe results about Legendrian approximations of bindings of open books and surgery on these Legendrian approximations.  All of the results admit more general statements to general manifolds, but we only need their statements for $S^3$.

\subsubsection{Legendrian approximations} Given a null-homologous fibered knot $K \subset S^3$ that defines an open book decomposition, denote by $\phi_K$ the monodromy of the open book, and by $\xi_K$ the contact structure supported by the open book.  The knot $K$ gets a natural orientation, and becomes a positive transverse knot (also denoted $K$) in $(S^3, \xi_K)$, with self-linking equal to $sl(K) = 2g-1$, where $g$ is the genus of the pages of the open book (see, for example, \cite{Etnyre:OBlectures}).  We are interested in what Legendrian approximations of $K$ we can find in $(S^3, \xi_K)$.  In general, we can always find approximations $L$ with $tb(L) < 0$, but in certain cases we can say more.

\begin{thm}[\!{\cite[Theorem~1.8 and Lemma~4.2]{Conway:admissible}}] \label{legendrianapproximations oneside}
If $\phi_K$ is not right-veering, then there exist Legendrian approximations $L_n$ of $K$ for each $n \in \Z$, where
\[ tb(L_n) = n\mbox{\hspace{.4cm}and\hspace{.4cm}}rot(L_n) = 1 - 2g + n. \]
\end{thm}

When $\phi_K$ is not right-veering then $\xi_K$ is overtwisted, by \cite[Theorem~1.1]{HKM:right1}, and so both the transverse knot $K$ and each Legendrian approximation $L_n$ are non-loose knots, that is the complement of any standard neighborhood of the knot is tight.

Given a Legendrian knot $L$, we can reverse the orientation to get a new Legendrian knot $-L$ with $tb(-L) = tb(L)$ and $rot(-L) = -rot(L)$. If $L$ is non-loose, then $-L$ is also non-loose. In the case that the orientation-reversal $-K$ of $K$ is smoothly isotopic to $K$, then we can upgrade Theorem~\ref{legendrianapproximations oneside}. This is indeed the case for the figure-eight knot considered in this paper.

\begin{thm} \label{legendrianapproximations}
If $\phi_K$ is not right-veering and $-K$ is smoothly isotopic to $K$, then there exist non-loose Legendrian knots $L^\pm_n$ in the smooth knot type of $K$ for each $n \in \Z$, and:
\begin{itemize}
\item $tb(L^\pm_n) = n$ and $rot(L^\pm_n) = \mp(1 - 2g + n)$.
\item $S^\pm(L^\pm_n)$ is isotopic to $L^\pm_{n-1}$.
\item $S^\pm(L^\mp_n)$ is loose.
\item The complement of $L_n$ has no Giroux torsion, no boundary-parallel half Giroux torsion, and is universally tight.
\end{itemize}
Here, $S^\pm(L)$ denotes a positive/negative stabilization of $L$.
\end{thm}
\begin{proof}
The first two bullet-points follow from Theorem~\ref{legendrianapproximations oneside} and the fact that the $L^-_n$ are Legendrian approximation of $K$ defined by thickening a standard neighborhood of a Legendrian approximation of $K$ to arbitrarily high slope.  The proof of Lemma~4.2 in \cite{Conway:admissible} shows that given a standard neighborhood $N$ of $L^-_n$, there is a negative basic slice in the complement of $\interior(N)$ with one boundary equal to $\partial N$ and the other having meridional dividing curves. Since the complement of a standard neighborhood of $S^+(L^-_n)$ is equal to the complement of $\interior(N)$ with a positive basic slice glued on, we can glue this positive basic slice to the negative basic slice to get a contact structure on $T^2 \times I$ which must be overtwisted, by Theorem~\ref{basic-slice}.
The statement for $L^+_n$ follows from that of $L^-_n$.
The statement about Giroux torsion follows from the fact that the complement of $K$ has no Giroux torsion, by Etnyre and Vela-Vick \cite[Theorem~1.2]{EVV:torsion}. No boundary-parallel half Giroux torsion can be present if Legendrian surgery on $L^\pm_n$ is tight, which it is by Theorem~\ref{surgeries on legendrianapproximations}. The complements are universally tight as they are Legendrian approximations of bindings of open books, which by construction have universally tight complements.
\end{proof}

Theorem~\ref{legendrianapproximations} gives two Legendrian knots $L^\pm_{2g-1}$ with the same invariants. However, since $S^+(L^+_{2g-1})$ is non-loose while $S^+(L^-_{2g-1})$ is loose, they must be non-isotopic.

\subsubsection{Surgery on the Legendrian approximations} Let $\overline K \subset S^3$ denote the mirror of $K$, and let $\xi_{\overline K}$ be the contact structure on $S^3$ with binding $\overline K$. Under certain circumstances, we can use information about $\xi_{\overline K}$ to conclude information about surgeries on $L^\pm_n$.

\begin{thm}[\!{\cite[Theorem~1.8]{Conway:admissible}}] \label{surgeries on legendrianapproximations}
Under the hypotheses from Theorem~\ref{legendrianapproximations}, and assuming that the Heegaard Floer contact invariant $c(\xi_{\overline K}) = 0$, then the result of Legendrian surgery on $L^\pm_n$ for any $n$ is tight, and has non-vanishing reduced Heegaard Floer contact invariant $c^+_{\rm{red}}$.
\end{thm}

\begin{remark} This is stated in \cite{Conway:admissible} in terms of admissible transverse surgery, but can be put into this form via the translation between admissible transverse surgery on a transverse knot and Legendrian surgery on its Legendrian approximations, from \cite[Theorem~3.1]{BE:transverse}. \end{remark}

\subsubsection{Open books for surgery on the Legendrian approximations} When $n \leq 0$, we can construct abstract open books that describe the result of Legendrian surgery on $L^-_n$, following \cite[Proposition~3.9]{Conway:transverse} and \cite[Example~5.2]{BEVHM}.

\begin{thm} \label{constructing open books}
Given an abstract open book $(\Sigma, \phi)$ corresponding to $K \subset S^3$, and let $n \leq 0$.  Let $\Sigma_n$ be a genus--$0$ surface with $2-n$ boundary components $B_0$, $\ldots$, $B_{1-n}$, and let $\Sigma'$ be the result of gluing $\partial \Sigma$ to $B_0$.  Then $(\Sigma', \phi\circ \Delta)$ is an abstract open book for the result of Legendrian surgery on $L^-_n$, where we extend $\phi$ over $\Sigma_n$ by the identity, and $\Delta$ is the composition of a positive Dehn twist around each boundary component of $\Sigma'$.
\end{thm}

\section{Upper Bound}
\label{sec:upperbound}

In this section, we will apply convex decomposition theory to compute the upper bound (Theorem~\ref{upper bound}) for isotopy classes of tight contact structures supported by $M(r)$ for $r\in \mathcal{R}$.  The strategy is to first decompose $M(r)$ into simpler manifolds with convex boundary, then normalize the dividing curves on the boundaries, and finally find a small upper bound on the number of tight contact structures on each piece.  Ultimately, our upper bound on the total number of tight contact structures for $M(r)$ will be the product of the upper bounds found for each piece.  When piecing together a contact structure on $M(r)$ from contact structures on the pieces, we only concern ourselves with the dividing curves on the convex boundaries, and not the characteristic foliations, even though the latter are what actually determine the germ of the contact structure at the boundary. This is because any desired characteristic foliation that is divided by a given dividing set can be achieved by a small isotopy of the surface, see \cite{Giroux:convex}.

\subsection{Convex decomposition} \label{sec:decomp} Our first step is to decompose $M(r)$ into simpler pieces. Fortunately, $M(r)$ has a nice structure, since the figure-eight knot is a genus--$1$ fibered knot. Let $C$ be the closure of the figure-eight knot complement in $S^3$; this is a once-punctured torus bundle over $S^1$. If we denote a fiber surface by $\Sigma$, then there is a symplectic basis of $H_1(\Sigma)$ such that the action of the monodromy is given by

\[
\phi = \matrixb{2}{1}{1}{1}.
\]

We assume $\phi$ is the identity near the boundary and that there is no boundary twisting. Putting this all together, we have described $C$ by
\[ C = \Sigma \times [0,1]\big/(x,1) \sim (\phi(x),0). \]

Let $N$ be the solid torus that is the closure of the complement of $C$ in $M(r)$.  We will describe elements of $H_1(\bd N)$ by rational numbers: we let $\mu$ be the slope such that Dehn filling $C$ along $\mu$ gives $S^3$, and we let $\lambda$ be isotopic to the boundary of a fiber surface $\Sigma$; we orient $\mu$ to agree with the orientation of $[0, 1]$ in the fibration, and we orient $\lambda$ to agree with the orientation of $\bd \Sigma$.  Then, the element $p \mu + q\lambda \in H_1(\bd N)$ can be described as $\frac{p}{q}$.

Note that we build $M(r)$ from $S^3$ by removing $N \subset S^3$ and regluing it such that the slope $r$ curve on $-\bd C$ bounds a disc in $N$.

Our strategy to normalize the dividing curves is as follows: we assume that $M(r)$ has been given a tight contact structure, and we isotope the core $L$ of $N$ to be Legendrian.  We assume that $N$ is a standard neighborhood of $L$, whose convex boundary then has two dividing curves of slope $s$ (where we will always measure slopes on $\bd N = -\bd C$ assuming that $C$ is sitting inside $S^3$). Since the contact planes give a framing of $L$, we know that $s$ and $r$ are connected by an edge in the Farey graph.  Let $\mathcal S(r)$ be the set of possible slopes $s$ for a given surgery coefficient $r$ that are clockwise of $r$ and counter-clockwise of $\infty$ on the Farey graph ($\mathcal S(r)$ may include $\infty$). When $C$ or $N$ have contact structures with convex boundaries with two dividing curves of slope $s$, we will find it helpful to denote them by $C(s)$ and $N(s)$.

Our first normalization result is to \dfn{thicken} $C(s)$ to a standard form. This means finding a $C(s') \subset C(s)$ such that the complement is $T^2 \times I$ with one boundary component equal to $\bd C(s)$ and the other boundary component equal to $\bd C(s')$. We say that $C(s)$ \dfn{thickens to $C(s')$}, or that $C(s)$ \dfn{thickens to slope $s'$}.  According to our slope conventions, if $C(s)$ thickens to $C(s')$, then $s'$ is clockwise of $s$ on the Farey graph, and for small thickening (such that $C(s)\setminus C(s')$ has no half Giroux torsion), we have $s' \geq s$.  If $C(s)$ thickens to slope $s'$ implies that $s' = s$, then we say that $C(s)$ \dfn{does not thicken}. 
\footnote{We say that $C$ {\em thickens}, even though $C$ is actually shrinking; we have still found this terminology sensible, since the slope is indeed increasing. Indeed, we are thickening $N$, but for most of the section, the statements are independent of $N$.}
We will thicken $C(s)$ by attaching bypasses to its boundary from the back. However, since we measure the slope $s$ according to the opposite orientation (\textit{ie.@} when it is considered as the boundary of $N$), the slope increases after attaching a bypass.

\begin{lem} \label{thickening}
Let $r \in \mathcal R$ and $s \in \mathcal S(r)$, and assume that there is no boundary-parallel half Giroux torsion in $C(s)$.  Then if $s \neq 0$, then $C(s)$ thickens to slope $-3$ or $\infty$, and does not thicken further. The same is true for $s = 0$ if $C(s)$ embeds in a tight contact structure on $M(r)$.
\end{lem}

One of the necessary ingredients --- both for thickening $C(s)$ and counting the possible number of tight contact structures on $C(s)$ --- will be to understand how to normalize the possible dividing curves on a fiber surface $\Sigma$ in $C(s)$. The possibilities will depend heavily on $s$, since the number of arcs in $\Gamma_\Sigma$ will depend on how many times $\bd \Sigma$ intersects the dividing set on $\bd C(s)$.

If there were boundary-parallel half Giroux torsion in $C$, then we could find a boundary-parallel convex torus with Legendrian divides of slope $r$, which then would bound overtwisted discs in $M(r)$. Thus, if we are looking for tight contact structures on $M(r)$, we know that we cannot have any boundary-parallel half Giroux torsion.

After Lemma~\ref{thickening}, we need to understand the number of contact structures on $C(-3)$ and $C(\infty)$ that cannot thicken.

\begin{lem}[Etnyre--Honda \cite{EH:knots}] \label{complement count -3}
There are at most two tight contact structures on $C(-3)$ that do not thicken, up to isotopy fixing a given singular foliation on the boundary. They are complements of a standard neighborhood of isotopic Legendrian figure-eight knots in $(S^3, \xi_{\rm std})$ with $tb=-3$ and $rot=0$.
\end{lem}

\begin{lem} \label{complement count infinity}
There are at most four tight contact structures on $C(\infty)$ that do not thicken, up to isotopy fixing a given singular foliation on the boundary. Two of them are complements of isotopic Legendrian knots in $(M(0),\xi^-)$, and the other two are complements of isotopic Legendrian knots in $(M(0), \xi^+)$.
\end{lem}

The contact structures $(M(0), \xi^\pm)$ are the only two tight contact structures on $M(0)$, as classified by Honda \cite{Honda:classification2}.

We will see that as $r$ switches from positive to negative, the count of tight contact structures is halved. In order to prove that, we will need the following lemma, showing that the varying contact structures on $C(\infty)$ are irrelevant, once $r < 0$.

\begin{lem} \label{complement count -1n}
There are at most two tight contact structures on $C(-1)$, and at most four tight contact structures on $C(-\frac1n)$ when $n > 1$ is an integer, that have no boundary-parallel half Giroux torsion, up to isotopy fixing a given singular foliation on the boundary.

In addition, if we write $C(-\frac1n) = T \cup C(\infty)$, where $T = T^2 \times [0,1]$ with convex boundary and dividing curves of slope $s_0 = -\frac1n$ and $s_1 = \infty$, then the (possibly) tight contact structures on $C(-\frac1n)$ are determined by their restrictions to $T$.
\end{lem}

Finally, we use Theorem~\ref{solid torus} to count the number of tight contact structures on the solid tori $N(-3)$ and $N(\infty)$ when the meridional slope is $r$ and the convex boundary has dividing curves of slopes $-3$ or $\infty$.

\begin{lem} \label{tight N}
For any $r \in \Q$, there are exactly $\Phi(r)$ tight contact structures on $N(\infty)$, up to isotopy fixing a given singular foliation on the boundary.

If $r < -3$, then there are exactly $\Psi(r)$ tight contact structures on $N(-3)$, up to isotopy fixing a given singular foliation on the boundary.
\end{lem}

Combining these lemmas, we get our upper bound.

\begin{thm} \label{upper bound}
Let $r \in \mathcal R$.
\begin{itemize}
\item If $r > 0$, there are at most $2\Phi(r)$ tight contact structures on $M(r)$ up to isotopy
\item If $r < 0$, there are at most $\Phi(r) + \Psi(r)$ tight contact structures on $M(r)$ up to isotopy.
\end{itemize}
\end{thm}
\begin{proof}
For $r > 0$, any $s \in \mathcal S(r)$ satisfies $s > -3$, and so $C(s)$ will thicken to slope $\infty$, by Lemma~\ref{thickening}.  Thus, any tight contact structure on $M(r)$ can be described by a tight contact structure on $C(\infty)$ glued to a tight contact structure on $N(\infty)$. By Lemma~\ref{complement count infinity}, we can see the tight contact structure on $C(\infty)$ as the complement of a knot in $(M(0), \xi^\pm)$, and gluing on a solid torus with some tight contact structure to its complement corresponds to some negative contact surgery. Since we know that isotopic knots give rise to isotopic contact manifolds after surgery, we arrive at the upper bound of $2\Phi(r)$, where $2$ is the choices of knots (up to isotopy) on which to do surgery, and $\Phi(r)$ is the number of tight contact structures on $N$, by Lemma~\ref{tight N}.

For $r < 0$, by Lemma~\ref{thickening}, $C(s)$ for any $s \in \mathcal S(r)$ thickens to slope $-3$ or $\infty$, and so any tight contact structure on $M(r)$ is built out of a pair of tight contact structures on $N(-3)$ and $C(-3)$ or on $N(\infty)$ and $C(\infty)$.  The first possibility leads to the $\Psi(r)$ that appears in the theorem statement, by Lemma~\ref{complement count -3}, Lemma~\ref{tight N} and the fact that contact surgeries on isotopic Legendrian knots result in isotopic contact structures, as in the preceding paragraph.

In the case that $C(s)$ thickens to $\infty$, after a stabilization if necessary, we can assume that $s \in \mathcal S(r)$ is negative, and hence there exists $n \in \N$ such that $s < -\frac1n$. By the discussion after Theorem~\ref{solid torus},  we can realize a boundary-parallel convex torus in $N(\infty)$ with two dividing curves of slope $-\frac 1n$, and use this to break $N(\infty)$ into $N(-\frac 1n)$ and $T$, where $T \cong T^2 \times I$ has dividing curve slopes $s_0 = -\frac 1n$ and $s_1 = \infty$.

By Lemma~\ref{tight N}, there are exactly $\Phi(r)$ tight contact structures on $N(-\frac 1n) \cup T$.  By Lemma~\ref{complement count -1n}, the contact structure on $T \cup C(\infty)$ depends only on the contact structure on $T$, and is independent of which of the four possible tight contact structures we choose on $C(\infty)$. Thus, we arrive at an upper bound of $\Phi(r)$ for tight contact structures of this type on $M(r)$.
\end{proof}

\subsection{Normalizing the dividing set on \texorpdfstring{$\Sigma$}{Sigma}} \label{sec:normalize} Our next goal is to normalize the dividing curves on the boundary of $C(s)$. To do this, we find a bypass for $-\bd C$ which lies inside of $C(s)$. Then according to Theorem~\ref{bypass on torus}, we can increase $s$ by attaching the bypass. Our strategy to find a bypass is to find a boundary-parallel dividing arc on $\Sigma$: we then obtain a bypass by Theorem~\ref{imba}. As a first step, we roughly normalize the dividing set on $\Sigma$ for general $s\in\mathbb{Q}\cup\{\infty\}$. Fortunately, this was already done by Etnyre and Honda \cite[Propositions~5.5,~5.6,~5.9]{EH:knots} when they classified Legendrian figure-eight knots in $(S^3,\xi_{\rm std})$, and we summarize the results here. By arc, we mean a properly embedded arc.

\begin{prop}[Etnyre--Honda \cite{EH:knots}] \label{normalizing-dividing-set}
	Let $m,n$ be co-prime integers and suppose $\left|m\right|>1$. If $s=\frac mn$, then there exists an isotopic copy of $\Sigma$ in $C(s)$ such that the dividing set is one of the following:
	
	\begin{enumerate}
	\item $m$ is odd, and
    	\begin{itemize}
    		\item there are $\left|m\right|$ arcs and one closed curve, parallel to $\vects{1}{0}$, or
    		\item there are $\left|m\right|-2$ arcs parallel to $\vects{0}{1}$, one arc parallel to $\vects{1}{1}$ and one arc parallel to $\vects{1}{0}$, or
    		\item there is a boundary-parallel arc (possibly with other dividing curves).
    	\end{itemize}
    \item $m$ is even, and there is a boundary-parallel arc (possibly with other dividing curves).
	\end{enumerate}
\end{prop}

\begin{remark} \label{remark:any-PA}
Note that the propositions appear slightly differently in \cite{EH:knots}, since Etnyre and Honda used the inverse gluing convention $C=\Sigma\times[0,1]/((x,0)\sim(\phi(x),1))$. Additionally, although they only considered integral values of $s$, their proof applies to any rational $s$. Moreover, their proof also applies to any pseudo-Anosov monodromy (although with a different monodromy, the slopes of the arcs/curves might be different than above.)
\end{remark}

When $\left|m\right|=1$, we have the following.

\begin{prop}[\!{\cite[Proposition~4.8]{Min}}] \label{normalizing-dividing-set-2}
	Suppose $s=\frac 1n$, for $n\in\mathbb{Z}$. Moreover, assume that there exists no boundary-parallel half Giroux torsion in $C(s)$. Then there exists an isotopic copy of $\Sigma$ in $C(s)$ such that the dividing set is one of the following:
	\begin{itemize}
		\item one arc and one closed curve, parallel to $\vects{1}{0}$, or
		\item one boundary-parallel arc (without any other dividing curves). 
	\end{itemize}
\end{prop}

\begin{remark}
	Although the second author used different monodromy for $C(s)$ in \cite{Min}, the proof applies to any pseudo-Anosov monodromy (although with a different monodromy, the closed curve might not be parallel to $\vects{1}{0}$).
\end{remark}

\subsection{Thickening \texorpdfstring{$C(s)$}{C(s)}} \label{sec:thickening} In this section, we will prove Lemma~\ref{thickening} and Lemma~\ref{complement count -1n}. As discussed above, we first find a bypass for $-\bd C(s)$ along $\bd\Sigma$ so that by attaching this bypass we can thicken $C(s)$. Since all the bypasses will be found in a copy of $\Sigma$, the bypasses have slope $0$ on $\bd C(s)$. Recall also that thickening $C(s)$ changes $s$ clockwise on the Farey tessellation.

We do so via the following propositions, which generalize \cite[Propositions~5.8,~5.9]{EH:knots}, which finds bypasses for integral $s < -3$.  This first proposition deals with negative slopes $s$ that might occur in $\mathcal S(r)$, for $r \in \mathcal R$.

\begin{prop} \label{bypass-negative-slope}
	Suppose $s<0$ and $s \notin \{-\frac{4n-1}{n} \mid n \in \mathbb{N}\}$. Then there exists an isotopic copy of $\Sigma$ in $C(s)$ that contains a boundary-parallel dividing arc.
\end{prop}

\begin{proof}
	Let $s=-\frac{m}{n}$ for a pair of co-prime positive integers $m,n$. Note that there are $m$ properly embedded dividing arcs on $\Sigma$. Decompose $C(s)$ along $\Sigma$ into a genus--$2$ handlebody $\Sigma \times [0,1]$ and round the edges to obtain a smooth convex boundary. The dividing curves on $\Sigma \times \{0\}$ are the image under the monodromy of the dividing curves on $\Sigma \times \{1\}$. The dividing curves divide $\bd\Sigma$ into $2m$ intervals; label these intervals $1$ through $2m$ following the orientation on the boundary of the surface induced by $\bd \left(\Sigma \times\{1\}\right)$. The dividing curves on $\bd\Sigma\times [0,1]$ connect the intervals on $\Sigma \times \{0\}$ and $\Sigma \times \{1\}$, but introduce a twist. More precisely, the $i$-th interval on $\Sigma \times \{1\}$ is connected to the $(i+2n-1)$-th interval (mod $2m$) on $\Sigma \times \{0\}$. See Figure~\ref{negative-edge-rounding}, for example. 

	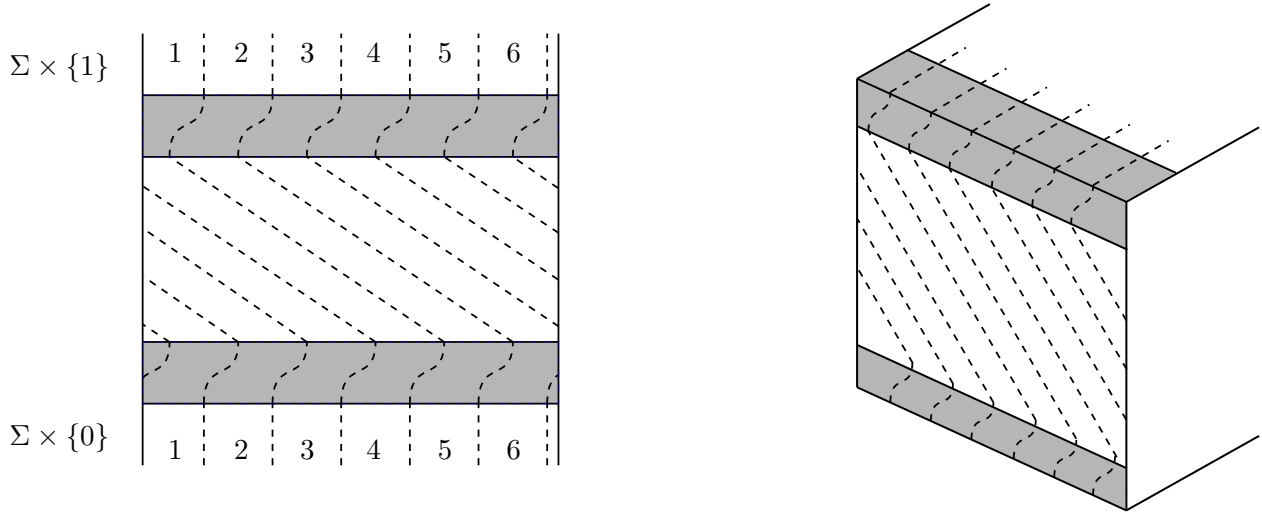
\begin{figure}[htbp]
		\begin{center}
		\begin{overpic}[scale=0.9,tics=20]{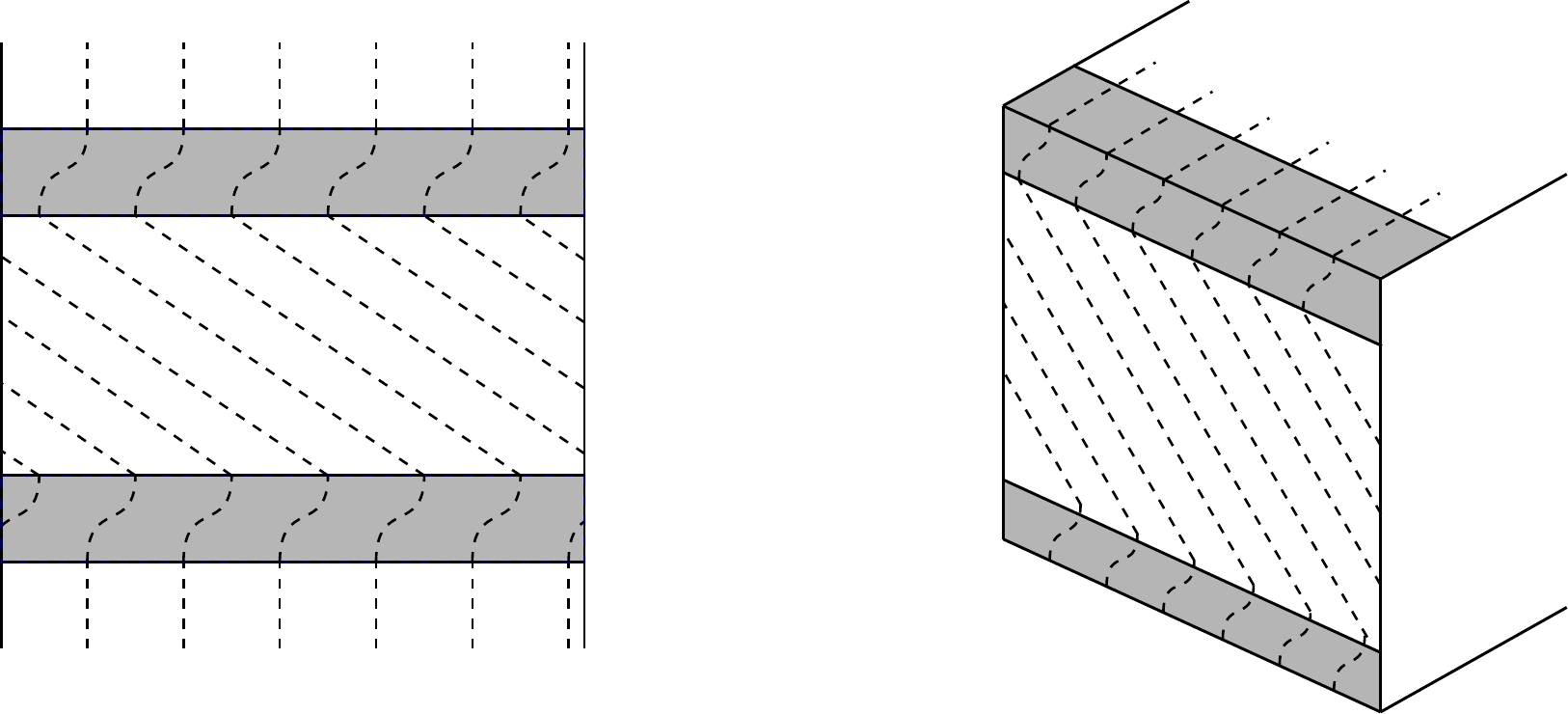}
		\put(-50,165){$\Sigma\times\{1\}$}
		\put(-50,25){$\Sigma\times\{0\}$}
		\put(10,170){$1$}
		\put(35,170){$2$}
		\put(60,170){$3$}
		\put(85,170){$4$}
		\put(112,170){$5$}
		\put(138,170){$6$}
		\put(10,20){$1$}
		\put(35,20){$2$}
		\put(60,20){$3$}
		\put(85,20){$4$}
		\put(112,20){$5$}
		\put(138,20){$6$}
		\end{overpic}
		\caption{Edge rounding when $s=-\frac 32$. In each drawing, the left side is identified with the right side.}
		\label{negative-edge-rounding}
		\end{center}
	 \end{figure}
	 
	 Let $\Gamma_i$ be a set of dividing curves on $\Sigma\times\{i\}$ for $i=0,1$. Although after edge-rounding, the entire dividing set consists of closed curves, we will continue to call dividing curves that pass through $\bd\Sigma$ dividing arcs, for convenience. By Proposition~\ref{normalizing-dividing-set} and Proposition~\ref{normalizing-dividing-set-2}, we only need to consider the following two cases.

	 \begin{figure}[htbp]
		\begin{center}
		\begin{overpic}[scale=1.2,tics=10]{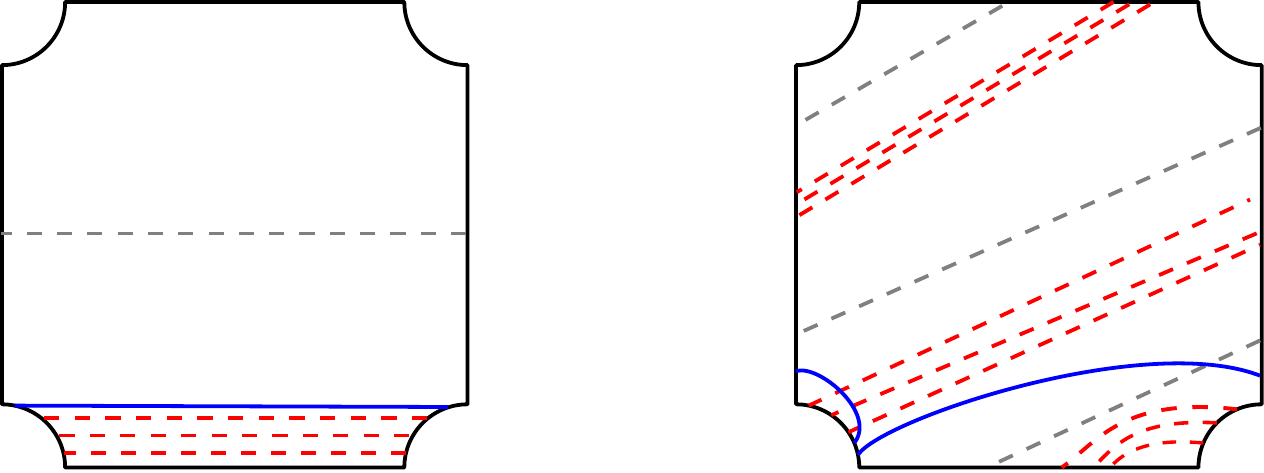}
		\put(60,-20){$\Sigma\times\{1\}$}
		\put(340,-20){$\Sigma\times\{0\}$}
		\put(5,16){\tiny $1$}
		\put(12.5,11.5){\tiny $2$}
		\put(16,6){\tiny $3$}
		\put(17.5,0){\tiny $4$}
		\put(10,148){\tiny $4$}	
		\put(150,148){\tiny $4$}
		\put(141.5,0){\tiny $4$}
		\put(142.7,6){\tiny $5$}
		\put(146,11.5){\tiny $6$}
		\put(154,16){\tiny $1$}
		\put(275.5,17){\tiny $1$}
		\put(281,15.5){\tiny $2$}
		\put(288,12){\tiny $3$}
		\put(291,5){\tiny $4$}
		\put(285,148){\tiny $4$}	
		\put(425,148){\tiny $4$}
		\put(417.5,2){\tiny $4$}
		\put(421,9.5){\tiny $5$}
		\put(425,14){\tiny $6$}
		\put(432,17){\tiny $1$}
		\end{overpic}
		\vspace{0.7cm}
		\caption{$\Sigma \times \{0, 1\}$ in $C(-\frac 32)$. The dotted lines are dividing curves. The blue lines are the intersections of the disc $D_0$ and $\Sigma \times \{0, 1\}$. In each drawing, the top and bottom are identified, and so are the left and right sides. In this figure, as well as in the rest of the paper, we adopt the convention of drawing $\Sigma \times \{0\}$ with the orientation induced from the fibration $C$, which disagrees with that induced by $\Sigma \times [0,1]$. Also in this picture (and this picture only), we have enumerated the regions of $\bd \Sigma$ to better see the twisting of the dividing set as it runs over $\bd\Sigma \times [0,1]$.}
		\label{negative-punctured-torus-1}
		\end{center}
	 \end{figure}
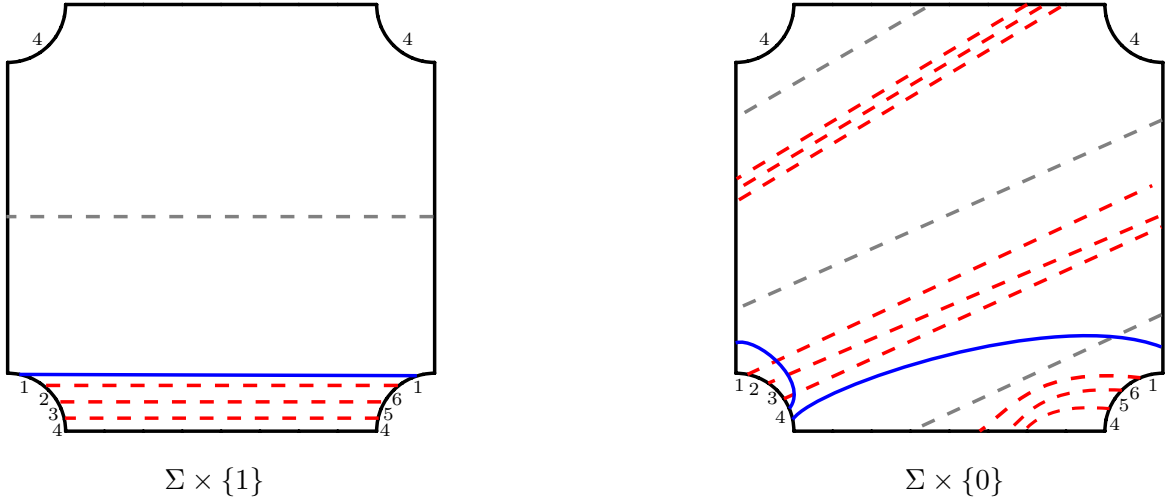
	
	(1) {\it $\Gamma_1$ contains $m$ arcs and one closed curve parallel to $\vects{1}{0}$}: Take an arc $\alpha$ on $\Sigma\times\{1\}$ which is parallel to $\vects{1}{0}$ and does not intersect any dividing curves, as shown in the left drawing of Figure~\ref{negative-punctured-torus-1}. Then $D_0=\alpha\times[0,1]$ is a compressing disc in $\Sigma\times[0,1]$ (the subscript $0$ here indicates that the slope of $\alpha$ is $0$). Perturb $D_0$ so that its boundary does not intersect any dividing curves on $\bd\Sigma\times[0,1]$. On $\Sigma\times\{0\}$, this will manifest as sliding the basepoints of $\alpha$ $2n+1$ spots along the positive orientation of $\bd \Sigma$, as described in the preceding paragraphs; see the right drawing of Figure~\ref{negative-punctured-torus-1}. After perturbing $D_0$ to be convex with Legendrian boundary, $\bd D_0$ will only intersect the dividing curves in $\Gamma_0$. Observe that $D_0$ intersects a closed dividing curve exactly once, and this intersection point separates $\alpha\times\{0\}$ into two sides: on one side, it intersects $|2n-1|$ dividing arcs, and on the other side, it intersects $|m-(2n-1)|$ dividing arcs. Since the total number of intersections is larger than $2$ (except when $m=n=1$, which we address below), we can find a bypass in the disc by Theorem~\ref{imba}. The difference between the numbers of intersections is 

	 \[ d = \left| |m-(2n-1)|-|2n-1| \right| = \case{|m-(4n-2)|, & \text{for } m\geq2n-1,}{|m|, & \text{for } m<2n-1.} \] 

	 If $d\neq1$, the dividing curves in $D_0$ cannot be nested like they are in the left drawing of Figure~\ref{dot-balancing}, so we can always find a bypass which does not straddle the closed dividing curve. After attaching this bypass, we obtain an isotopic copy of $\Sigma$ containing a boundary-parallel dividing arc.

	 \begin{figure}[htbp]
		\begin{center}
		\begin{overpic}[scale=1.2,tics=20]{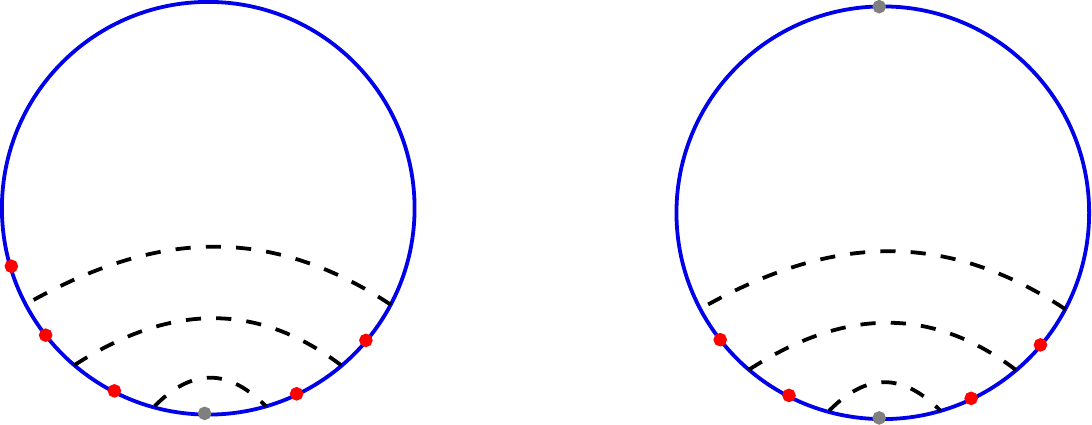}
		\end{overpic}
		\caption{Nested dividing curves in $D_0$.}
		\label{dot-balancing}
		\end{center}
	 \end{figure}
	 
	 \begin{figure}[htbp]
		\begin{center}
		\begin{overpic}[scale=1.2,tics=20]{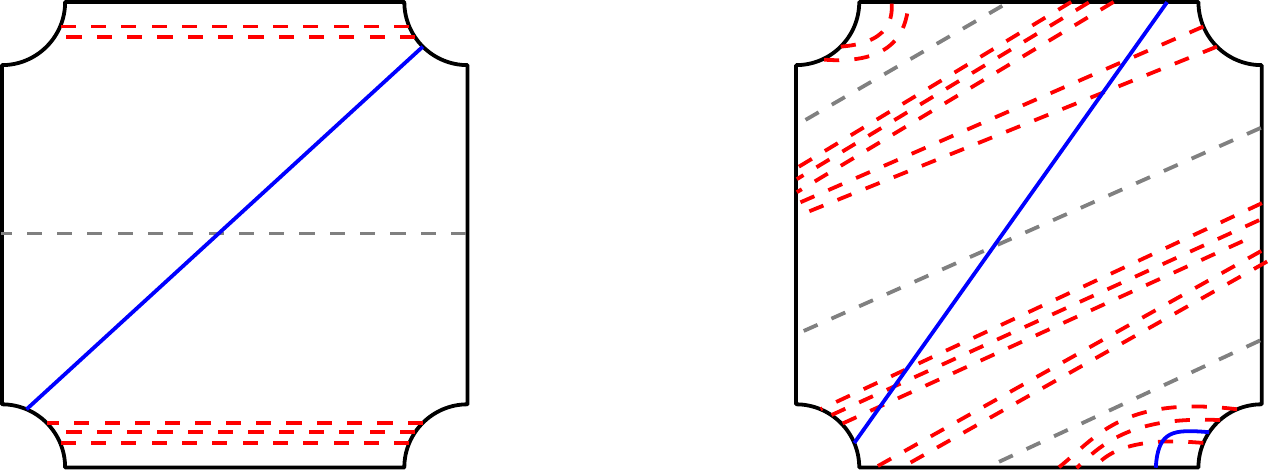}
		\put(60,-20){$\Sigma\times\{1\}$}
		\put(340,-20){$\Sigma\times\{0\}$}	
		\end{overpic}
		\vspace{0.7cm}
		\caption{$\Sigma \times \{0, 1\}$ in $C(-\frac 52)$. The blue lines are the intersections between $D_1$ and $\Sigma \times \{0, 1\}$.}
		\label{negative-punctured-torus-2}
		\end{center}
	 \end{figure}
	 
	 	 \begin{figure}[htbp]
		\begin{center}
		\begin{overpic}[scale=1.2,tics=20]{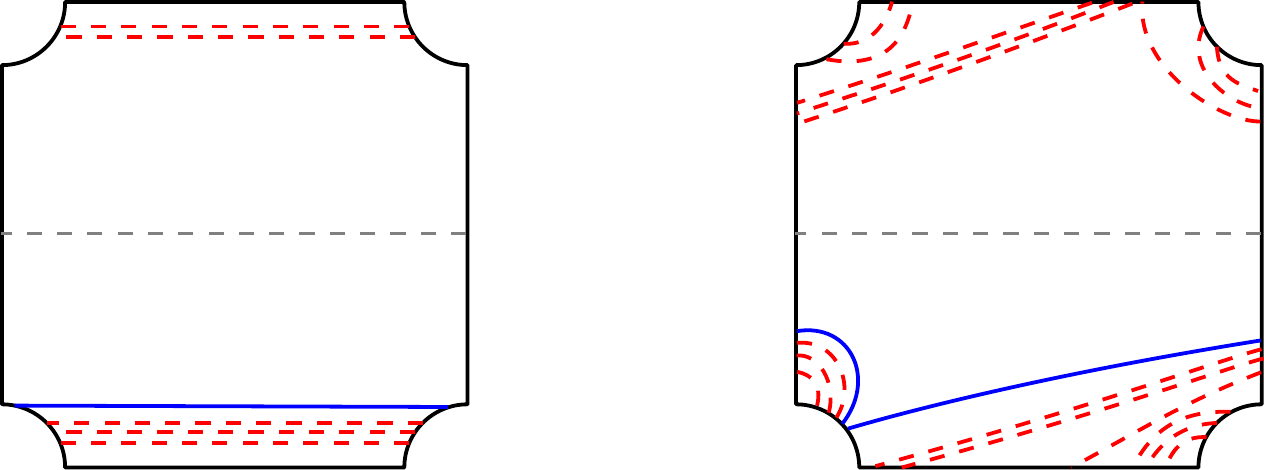}
		\put(60,-20){$\Sigma \times \{1\}$}
		\put(340,-20){$\Sigma \times \{\frac12\}$}	
		\end{overpic}
		\vspace{0.7cm}
		\caption{$\Sigma \times \{\frac12, 1\}$ in $C(-\frac 52)$ after attaching nested bypasses to $\Sigma \times \{0\}$. The blue lines are the intersections between an overtwisted disc and $\Sigma \times \{\frac12, 1\}$.}
		\label{negative-punctured-torus-3}
		\end{center}
	 \end{figure}
	 
	 If $d=1$, the dividing curves in $D_0$ might be nested. Moreover, we must have $m=4n-1$, $m=4n-3$, or $m=1$. First suppose $m=4n-3$ and $n>1$. Take an arc $\beta$ parallel to $\vects{1}{1}$ on $\Sigma \times \{1\}$ that only intersects the closed dividing curve once, see Figure~\ref{negative-punctured-torus-2}. Take a compressing disc $D_1=\beta \times [0,1]$ (with subscript $1$ due to the slope of $\beta$) in $\Sigma \times [0,1]$ and perturb it so that it does not intersect the dividing curves on $\bd \Sigma \times[0,1]$; the disc then intersects $4n-1$ dividing curves in $\Gamma_0$. The total number of intersections is clearly larger than $4$, so we can find a bypass in the disc.
	 
	 If there is a bypass on $\Sigma \times \{0\}$ which does not straddle the closed dividing curve, we obtain a boundary-parallel dividing arc after attaching the bypass, as before, so assume that there does not exist a such bypass. This implies that the dividing curves in the disc are nested as shown in the right drawing of Figure~\ref{dot-balancing}. Attach all bypasses successively to $\Sigma \times \{0\}$ to obtain $\Sigma \times \{\frac12\}$ with dividing curves parallel to $\vects{1}{0}$ but with some boundary twisting, see Figure~\ref{negative-punctured-torus-3}. (To be precise, we create $\Sigma\times\{1/2\}$ by first pushing $\Sigma\times\{0\}$ over the bypass in the neighborhood of the bypass, and then isotoping the convex surface using the contact vector field such that it is disjoint from $\Sigma\times\{0\}$.) Then as we demonstrate in Figure~\ref{negative-punctured-torus-3}, we can find an overtwisted disc by perturbing $\alpha \times [1/2,1]$, so this cannot happen.
	 
	 Now consider the case $m=1$ and $n>1$. Take $D_1$ to be as before, see Figure~\ref{negative-punctured-torus-4}. If there exists a bypass for $\Sigma\times\{0\}$ which does not straddle the closed dividing curve, we obtain a boundary-parallel dividing arc after attaching the bypass. If not, attach the bypass to $\Sigma\times\{0\}$ which straddles the closed dividing curve to obtain $\Sigma\times\{\frac13\}$, as shown in the first drawing of Figure~\ref{negative-punctured-torus-5}.
	 
	 Observe that the number of intersection points between the disc and dividing curves is $4n-4$. Since $n>1$, we can find a bypass for $\Sigma \times \{\frac13\}$ along a disc found by perturbing $\alpha \times [1/3,1]$ and we obtain a boundary-parallel dividing arc after attaching the bypass, see the last drawing of Figure~\ref{negative-punctured-torus-5}.
	 
	 \begin{figure}[htbp]
		\begin{center}
		\begin{overpic}[scale=1.2,tics=20]{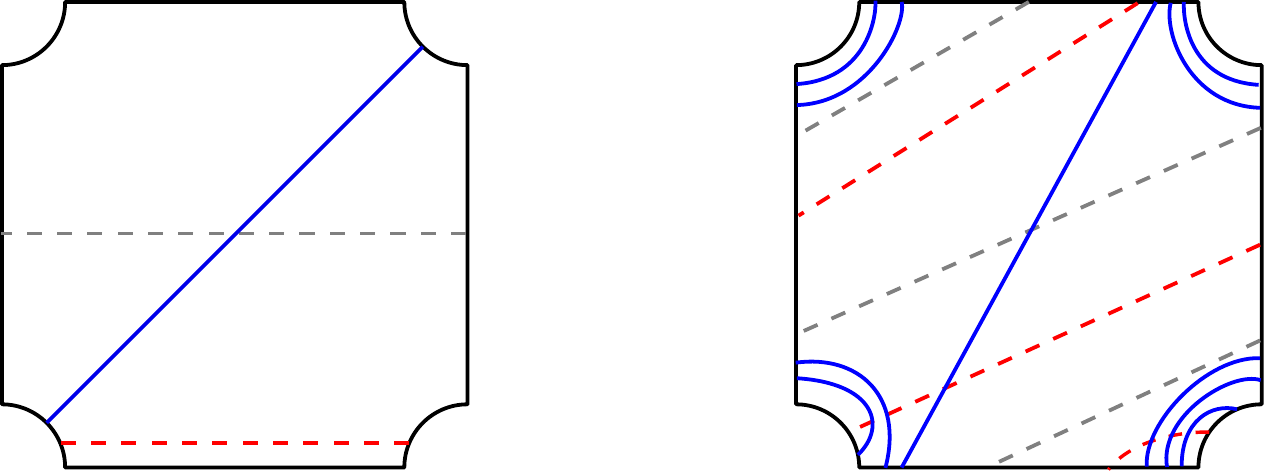}
		\put(60,-20){$\Sigma\times\{1\}$}
		\put(340,-20){$\Sigma\times\{0\}$}	
		\end{overpic}
		\vspace{0.7cm}
		\caption{$\Sigma \times \{0, 1\}$ in $C(-\frac 12)$. The blue lines are the intersections between $D_1$ and $\Sigma \times \{0, 1\}$.}
		\label{negative-punctured-torus-4}
		\end{center}
	 \end{figure}

	 \begin{figure}[htbp]
		\begin{center}
		\begin{overpic}[scale=1.2,tics=20]{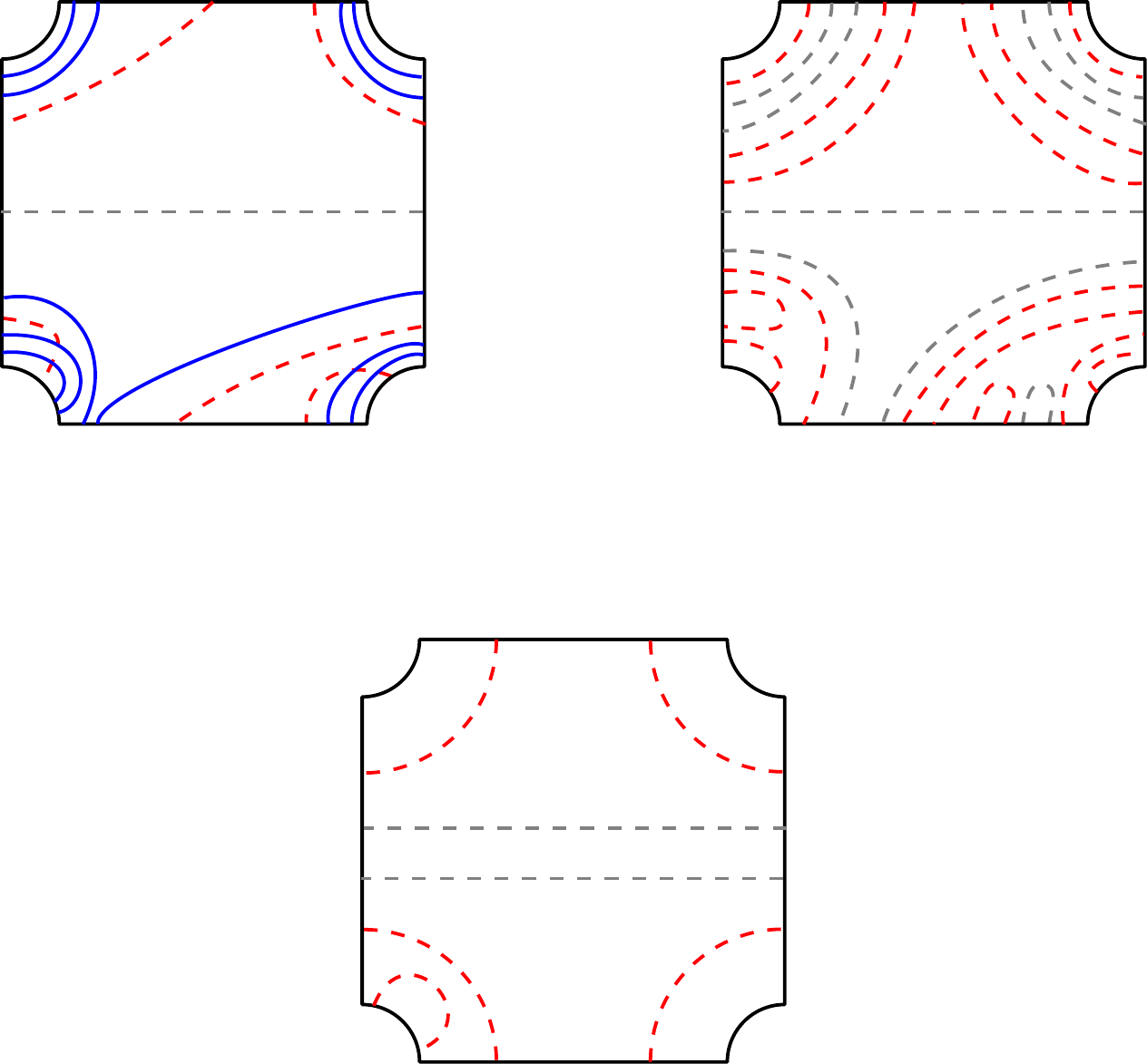}
		\put(60,220){$\Sigma\times\{\frac13\}$}
		\put(340,220){$\Sigma\times\{\frac23\}$}
		\put(200,-20){$\Sigma\times\{\frac23\}$}
		\end{overpic}
		\vspace{0.7cm}
		\caption{Top left: $\Sigma \times \{\frac13\}$ in $C(-\frac 12)$ after attaching a bypass to $\Sigma \times \{0\}$ which straddles the gray curve from Figure~\ref{negative-punctured-torus-4}. The blue lines are the intersections between the disc and $\Sigma \times \{\frac 13\}$. Top right: $\Sigma \times \{\frac23\}$ in $C(-\frac 12)$ after attaching a bypass to $\Sigma \times \{\frac13\}$ in the figure on the left. Bottom: the same as the top right, but after simplifying the dividing set.}
		\label{negative-punctured-torus-5}
		\end{center}
	 \end{figure}
	 
	 Now suppose $m=1$ and $n=1$. Take a compressing disc $D_1$ as before. This disc intersects dividing curves at four points, so we can find a bypass. If a bypass straddles the closed dividing curve on $\Sigma\times\{0\}$, attach this bypass and take the disc $D_0$. Since the number of intersections between the disc and dividing curves is $4n-4=0$, it is an overtwisted disc. Suppose a bypass straddles a dividing arc. Note that the attaching arc for this bypass is not contained entirely on $\Sigma\times\{0\}$, as shown in Figure~\ref{negative-punctured-torus-6}, but by bypass rotation and sliding (Theorem~\ref{bypass-rotation} and Theorem~\ref{bypass-sliding}), we can find a new bypass that is attached to $\Sigma \times \{0\}$, see Figure~\ref{negative-punctured-torus-7}. (Note that we rotate the bypass from left to right while on $\Sigma \times \{1\}$ since the bypass is attached from the back of $\Sigma\times\{1\}$.) Thus after attaching this new bypass, as in Figure~\ref{negative-punctured-torus-7}, we obtain an isotopic copy of $\Sigma$ containing one boundary-parallel arc and one boundary-parallel closed curve, as desired.

	 \begin{figure}[htbp]
		\begin{center}
		\begin{overpic}[scale=1.2,tics=20]{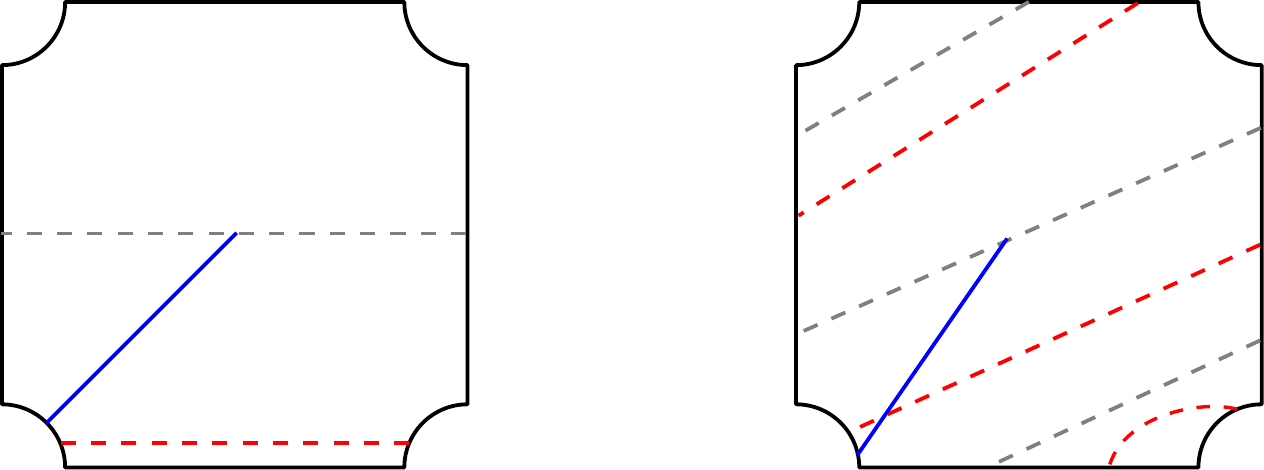}
		\put(60,-20){$\Sigma\times\{1\}$}
		\put(340,-20){$\Sigma\times\{0\}$}	
		\end{overpic}
		\vspace{0.7cm}
		\caption{$\Sigma \times \{0, 1\}$ in $C(-1)$. The blue lines are attaching arcs of a bypass.}
		\label{negative-punctured-torus-6}
		\end{center}
	 \end{figure}

	 \begin{figure}[htbp]
		\begin{center}
		\begin{overpic}[scale=1.2,tics=20]{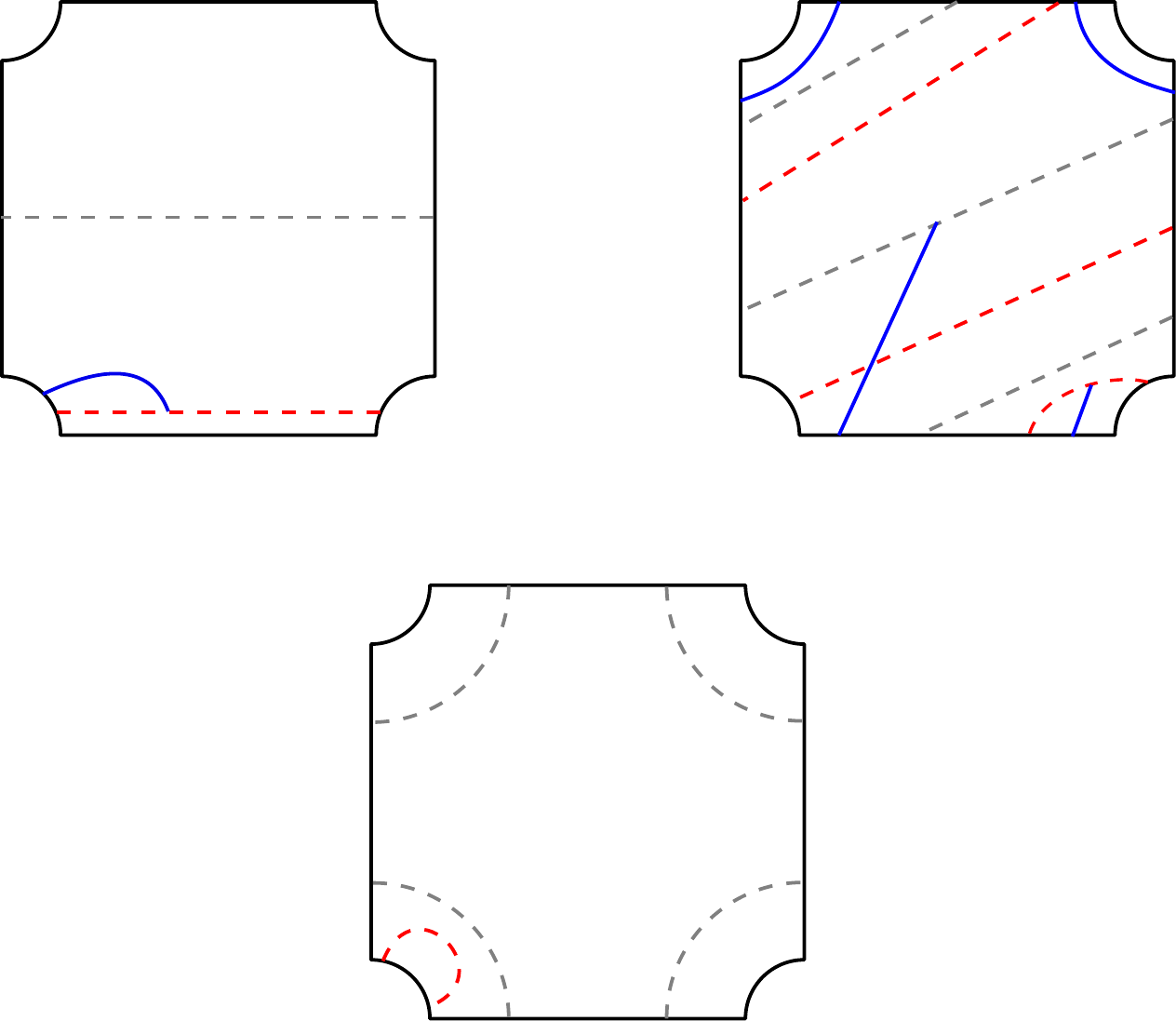}
		\put(60,200){$\Sigma\times\{1\}$}
		\put(340,200){$\Sigma\times\{0\}$}
		\put(200,-20){$\Sigma\times\{\frac12\}$}	
		\end{overpic}
		\vspace{0.7cm}
		\caption{The blue line in the first drawing is the attaching arc of the bypass in Figure~\ref{negative-punctured-torus-6} after the bypass rotation. The blue lines in the second drawing is the attaching arc of the bypass after the bypass sliding. The last drawing is $\Sigma \times \{\frac12\}$ in $C(-1)$ after attaching the bypass to $\Sigma \times \{0\}$.}
		\label{negative-punctured-torus-7}
		\end{center}
	 \end{figure}

	 (2) {\it $\Gamma_1$ contains $(m-2)$ arcs parallel to $\vects{0}{1}$, one arc parallel to $\vects{1}{1}$, and one arc parallel to $\vects{1}{0}$}: Take $\beta$ and $D_1$ as before, see Figure~\ref{negative-punctured-torus-8}. Perturb $D_1$ it so that it only intersects dividing curves in $\Gamma_0$. Since the total number of intersections is $m+4n-3$, which is larger than $2$, we can find a bypass for $\Sigma\times\{0\}$. It is easy to check that every possible bypass gives rise to a boundary-parallel dividing arc.
\end{proof}

	 \begin{figure}[htbp]
		\begin{center}
		\begin{overpic}[scale=1.2,tics=20]{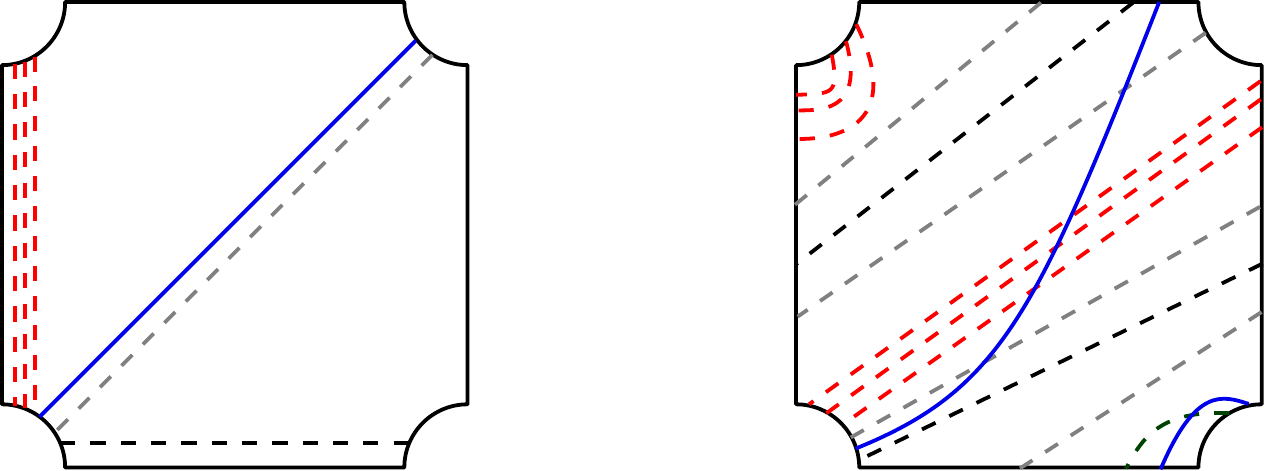}
		\put(60,-20){$\Sigma\times\{1\}$}
		\put(340,-20){$\Sigma\times\{0\}$}	
		\end{overpic}
		\vspace{0.7cm}
		\caption{$\Sigma \times \{0, 1\}$ in $C(-5)$. The blue lines are the intersections between the disc and $\Sigma \times \{0, 1\}$.}
		\label{negative-punctured-torus-8}
		\end{center}
	 \end{figure}

This second proposition deals with positive slopes $s$ that might occur in $\mathcal S(r)$, for $r \in \mathcal R$.

\begin{prop} \label{bypass-positive-slope}
	Suppose $s>0$ and $s \notin \{\frac{1}{n}, \frac{4n+1}{n} \mid n \in \mathbb{N}\}$. Then there exists an isotopic copy of $\Sigma$ in $C(s)$ that contains a boundary-parallel dividing arc.
\end{prop}

\begin{proof}
	Let $s=\frac mn$ for a pair of positive co-prime integers $m,n$. Then there are $m$ dividing arcs on $\Sigma$. As in Proposition~\ref{bypass-negative-slope}, we cut $C(s)$ along $\Sigma$ and round the edges to obtain $\Sigma\times[0,1]$ with a smooth convex boundary. Note that the $i$-th interval on $\Sigma\times\{1\}$ is connected to the $(i-2n-1)$-th interval (mod $2m$) on $\Sigma\times\{0\}$ via the dividing curves on $\bd\Sigma \times [0,1]$. We have two cases, by Proposition~\ref{normalizing-dividing-set}.
	
	(1) {\it $\Gamma_1$ contains $m$ arcs and one closed curve parallel to $\vects{1}{0}$}: Take $D_0$ as before, as shown in the left drawing of Figure~\ref{negative-punctured-torus-1}. Perturb the disc so that it only intersects dividing curves in $\Gamma_0$; there will necessarily be some twists on $\alpha\times\{0\}$ near the boundary (note that the direction of twisting will be opposite from that in the right drawing of Figure~\ref{negative-punctured-torus-1}).
	
	There is one intersection with the closed dividing curve, and on either side of this intersection point there are $(2n+1)$ and $(m+2n+1)$ intersection points with dividing arcs, respectively. The total number of intersections is $m + 4n + 2 > 2$, so there exists a bypass along the disc. The difference between the number of intersections on either side of the closed dividing curve is 
	\[ d = |m+(2n+1)-(2n+1)| = m. \]
	 
	Since $s \neq \frac1n$, we have $m>1$, and so the dividing curves in the disc cannot be nested as they are in Figure~\ref{dot-balancing}. Thus, there exists a bypass that does not straddle the closed dividing curve, and after attaching this bypass, we obtain an isotopic copy of $\Sigma$ containing a boundary-parallel dividing arc. 

	\begin{figure}[htbp]
		\begin{center}
		\begin{overpic}[scale=1.2,tics=20]{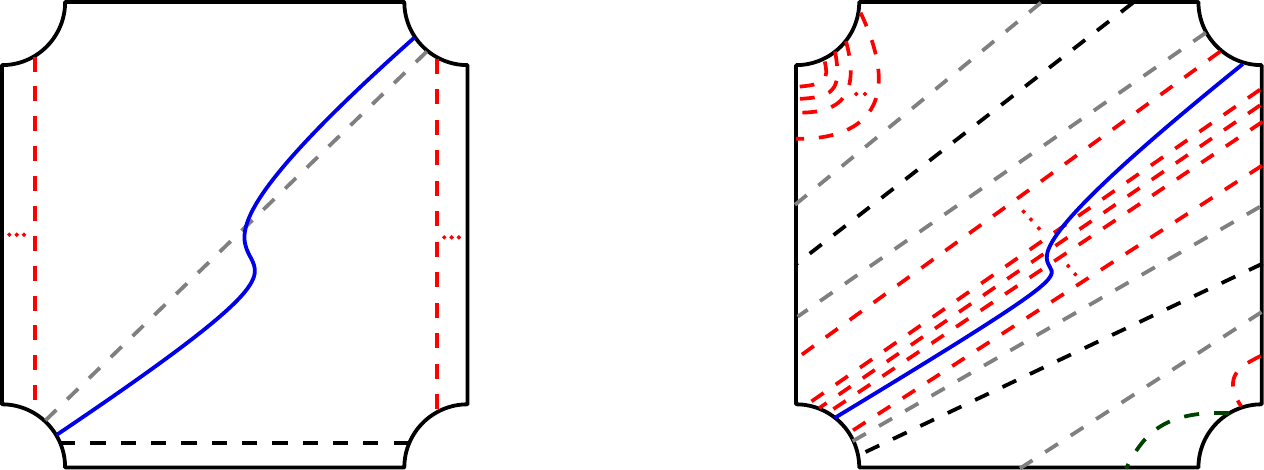}
		\put(60,-20){$\Sigma\times\{1\}$}
		\put(340,-20){$\Sigma\times\{0\}$}	
		\end{overpic}
		\vspace{0.7cm}
		\caption{$\Sigma \times \{0, 1\}$ in $C(9)$. The blue lines are the intersections between $D'_1$ and $\Sigma \times \{0, 1\}$.}
		\label{positive-punctured-torus-1}
		\end{center}
	 \end{figure}

	 \begin{figure}[htbp]
		\begin{center}
		\begin{overpic}[scale=1.2,tics=20]{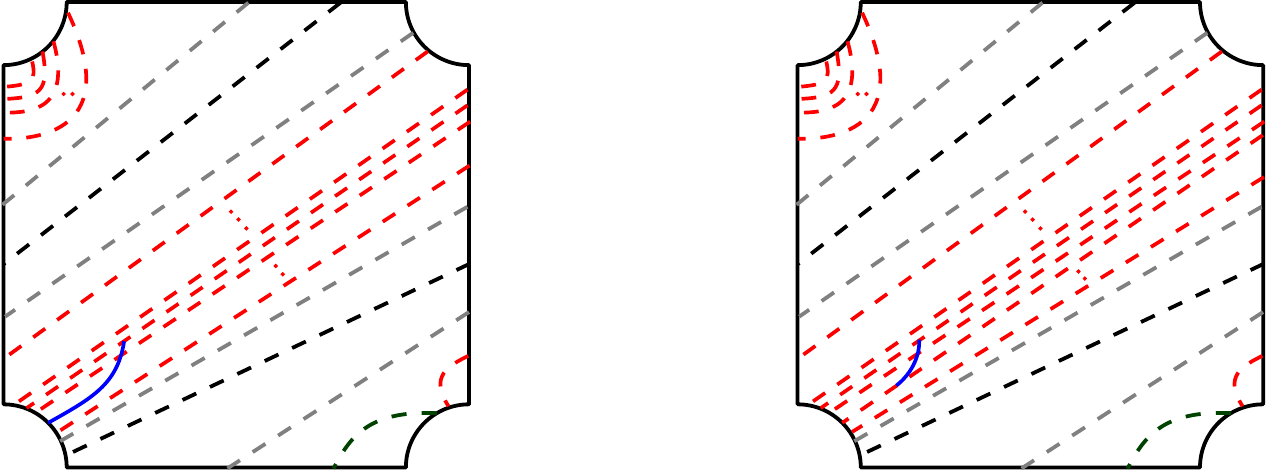}
		\put(60,-20){$\Sigma\times\{0\}$}
		\put(340,-20){$\Sigma\times\{0\}$}	
		\end{overpic}
		\vspace{0.7cm}
		\caption{$\Sigma \times \{0\}$ in $C(9)$. The blue line in the right drawing is an attaching arc of a bypass. The blue line in the left drawing is an attaching arc of a bypass after the bypass rotation.}
		\label{positive-punctured-torus-2}
		\end{center}
	 \end{figure}

	(2) {\it $\Gamma_1$ contains $(m-2)$ arcs parallel to $\vects{0}{1}$, one arc parallel to $\vects{1}{1}$ and one arc parallel to $\vects{1}{0}$}: Take $D_1$ as before such that the boundary does not intersect any dividing curves in $\Gamma_1$ or $\bd\Sigma \times[0,1]$, as shown in Figure~\ref{negative-punctured-torus-8}. It will only intersect dividing curves in $\Gamma_0$, and the total number of intersections is 
	\[ d=|(m-1)-2(2n+1)|=|m-4n-3|. \]
	If $d=0$, $D_1$ is an overtwisted disc. If $d>2$, we can find a bypass for $\Sigma\times\{0\}$ and it is easy to check this bypass gives rise to a boundary-parallel dividing arc. If $d=2$, this implies that $m = 4n+1$ or $m = 4n+5$. Suppose $m=4n+5$. Choose an arc $\gamma$ parallel to $\vects{1}{1}$ on $\Sigma\times\{1\}$ that intersects the arc parallel to $\vects{1}{1}$, see the left drawing of Figure~\ref{positive-punctured-torus-1}. Take a compressing disc $D_1' = \gamma\times[0,1]$ in $\Sigma\times[0,1]$ and perturb the disc so that its boundary does not intersect any dividing curve on $\bd\Sigma\times[0,1]$: then $D'_1$ will intersect three dividing curves in $\Gamma_0$, see the right drawing of Figure~\ref{positive-punctured-torus-1}. If there exists a bypass on $\Sigma\times\{0\}$ that intersects three dividing arcs, it will give rise to a boundary-parallel arc. If no such bypass is immediate, we can apply bypass rotation and sliding (Theorem~\ref{bypass-rotation} and Theorem~\ref{bypass-sliding}) to locate the bypass on $\Sigma\times\{0\}$, as in the proof of Proposition~\ref{bypass-negative-slope}; see Figure~\ref{positive-punctured-torus-2}. This bypass will intersect three dividing arcs, and will give rise to the desired boundary-parallel dividing arc.
\end{proof}

With the previous propositions under our belt, the only value of $s$ that might come up that we have yet to consider is $s = 0$. Since $\Sigma$ in $C(0)$ has no dividing arcs, the above techniques will not aide us in determining whether or not $C(0)$ thickens. Instead, we will see in the proof below that it is possible to entirely avoid considering the $s = 0$ case by making use of the $s = -\frac1n$ cases.

\begin{proof}[Proof of Lemmas~\ref{thickening} and ~\ref{complement count -1n}]
	Since $s \in \mathcal{S}(r)$, and $r \in \mathcal{R}$, we know that $s \notin (-4,-3] \cup (0,1] \cup (4,5]$. In case $s = 0$, then $r = -\frac1n$. We consider a knot $L$ in a tight contact structure on $M(-\frac1n)$ such that the complement of $L$ is $C(0)$. Consider the complement of a stabilization of $L$: since stabilizing is like attaching a bypass from the back with slope $-\frac1n$, the slope changes from $0$ counter-clockwise, and the complement of this stabilization is $C(-\frac1{n+1})$. We can consider this contact manifold instead of $C(0)$. (It will follow from the proof that the sign of the stabilization is irrelevant.)
	
	We first suppose that $s<0$. Then by Proposition~\ref{bypass-negative-slope}, we can find a bypass for $-\bd C(s)$ along an arc in $\bd\Sigma$. According to Theorem~\ref{bypass on torus}, we can increase $s$ by attaching this bypass to $-\bd C(s)$, and repeating this, we can thicken $C(s)$ until we obtain $C(-3)$ which does not thicken any further or we obtain $C(-\frac 1n)$, for some $n \in \N$. Assume we are in the second case, and let $\Gamma_\Sigma$ be a set of dividing curves on $\Sigma$ in $C(-\frac 1n)$. By Proposition~\ref{normalizing-dividing-set-2}, we only need to consider the following two cases.

	(1) {\it $\Gamma_\Sigma$ contains one boundary-parallel arc without any other dividing curves}: First, fix the signs of the regions of $\Sigma$ so that the sign of the bypass is negative. Observe that the relative Euler class evaluated on $\Sigma$ is $-2$. After cutting $C(-\frac 1n)$ along $\Sigma$ and rounding the edges to obtain $\Sigma\times[0,1]$ with a smooth convex boundary, we observe that the dividing set on the boundary is one closed curve parallel to $\bd\Sigma$. By choosing convex compressing discs for this handlebody whose Legendrian boundaries intersect the dividing set exactly twice, we see that there is a unique (potentially) tight contact structure on this handlebody, and thus a unique (potentially) tight contact structure on $C(-\frac 1n)$, which we denote by $\xi^-_{-n}$.
	
	It follows that any tight contact structure on $C(-\frac 1n)$ whose relative Euler class evaluated on $\Sigma$ is $-2$ is actually isotopic to $\xi^-_{-n}$, provided they have the same singular foliation on the boundary, which can be arranged by small perturbations. Hence the complement of a standard neighborhood of $L^-_{-1}$ in $(S^3,\xi_1^{OT})$ is isotopic to $\xi^-_{-1}$, by Theorem~\ref{legendrianapproximations} (we use notation from Section~\ref{sec:phi-lower-bound}) and thus that $\xi^-_{-1}$ is in fact tight. In the proof of Lemma~4.2 in \cite{Conway:admissible}, the first author showed that $(C(-1),\xi^-_{-1})$ in fact thickens to $C(\infty)$, and $C(-1)=T^- \cup C(\infty)$, where $T^-$ is a negative basic slice with dividing curve slopes $s_0=-1$ and $s_1=\infty$.  Decompose $T^-$ into $T_0^- \cup T_1^-$, where $T_0^-$ is $T^2\times[0,1]$ with dividing curve slopes $s_0=-1$ and $s_1=-\frac 1n$, and $T_1^-$ is $T^2\times[1,2]$ with dividing curve slopes $s_1=-\frac 1n$ and $s_2=\infty$. Note that each $T_i$ for $i=0,1$ is a stack of negative basic slices. Hence we obtain a tight contact structure on $C(-\frac 1n) = T_1^- \cup C(\infty)$. It is easy to check that the relative Euler class on $\Sigma$ is $-2$, so this contact structure is isotopic to $\xi^-_{-n}$. Therefore, $(C(-\frac 1n), \xi^-_{-n})$ also thickens to $C(\infty)$, for each $n$.
	
	Next, change the signs of the regions of $\Sigma$ so that the sign of the bypass is positive. In this case, the relative Euler class on $\Sigma$ is $2$. After re-running the above argument {\em mutatus mutandis}, we get a unique tight contact structure under these conditions, which we denote by $\xi^+_{-n}$.  Similarly, we build these tight contact structures $\left(C(-\frac 1n),\xi^+_{-n}\right)$ that thicken to $C(\infty)$.
	
	The above paragraphs show that these tight contact structures on $C(-\frac 1n)$ are determined by the sign of the bypass on $\Sigma$, which is in turn determined by the sign on the stacks of basic slices $T^\pm$, or equivalently the sign of $T_1^\pm$, and are thus independent of the tight contact structure on $C(\infty)$.

	\begin{figure}[htbp]
		\begin{center}
		\begin{overpic}[scale=1.2,tics=20]{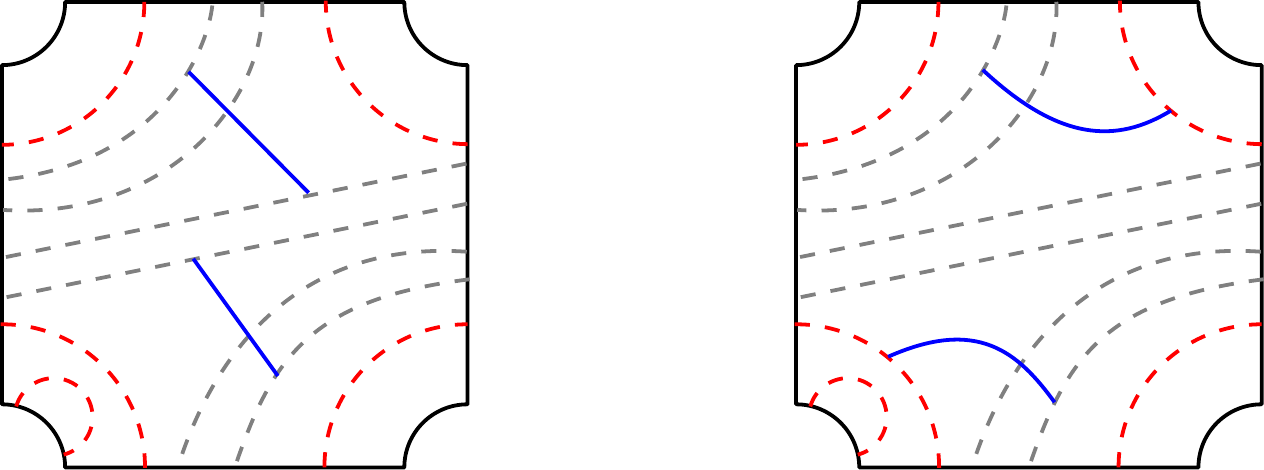}
		\put(60,-20){$\Sigma\times\{0\}$}
		\put(340,-20){$\Sigma\times\{0\}$}	
		\end{overpic}
		\vspace{0.7cm}
		\caption{The blue lines in the left picture are possible attaching arcs of a bypass on $\Sigma \times \{0\}$ along $c \times \{0\}$ (exactly one will exist, given a fixed dividing set on $c \times [0,1]$). The blue lines in the right picture are the attaching arcs after the bypass rotation.}
		\label{bypass-punctured-torus}
		\end{center}
	 \end{figure}

	 \begin{figure}[htbp]
		\begin{center}
		\begin{overpic}[scale=1.2,tics=20]{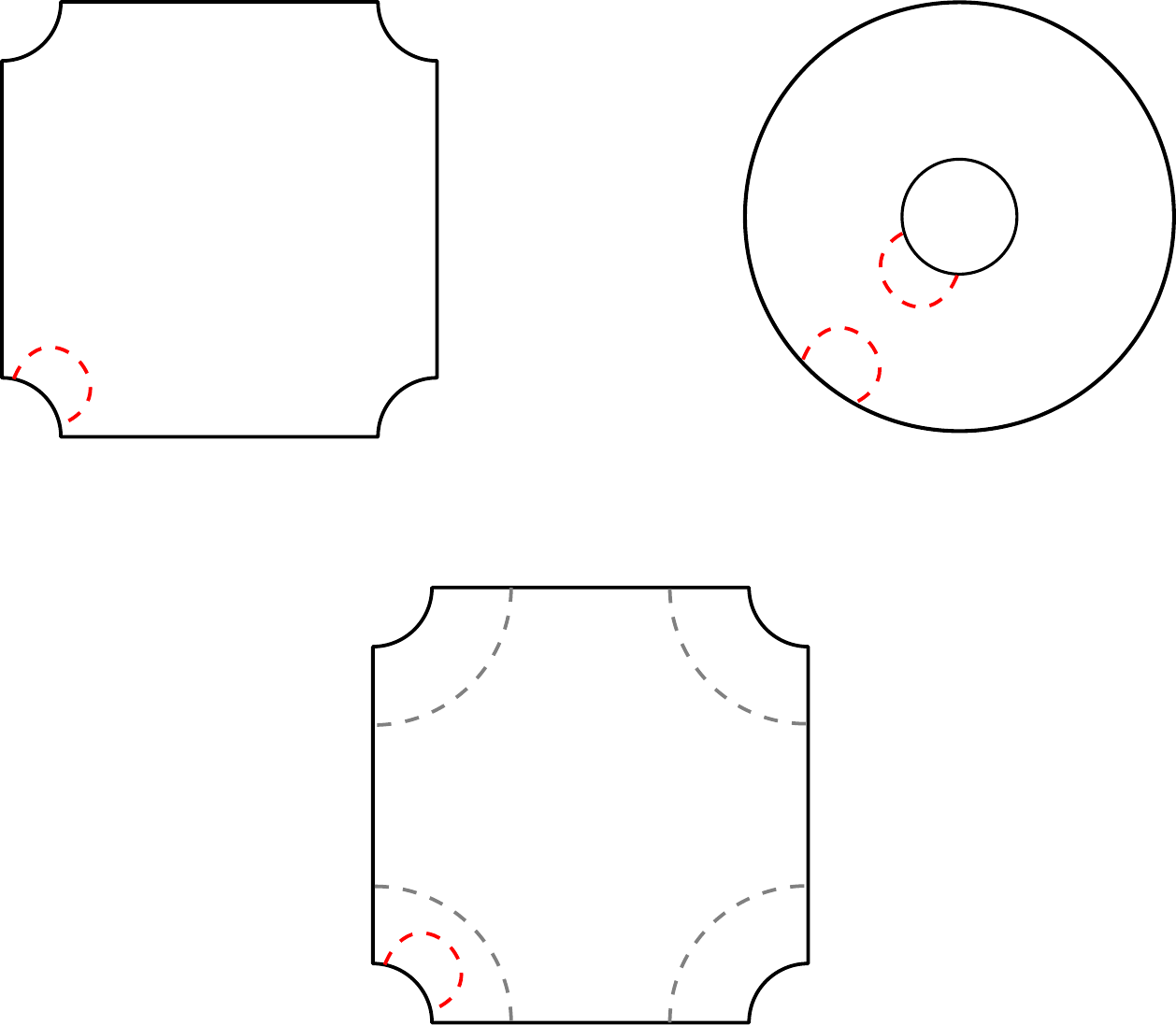}
		\put(70,200){$\Sigma$}
		\put(350,200){$A$}
		\put(215,-20){$\Sigma$}	
		\end{overpic}
		\vspace{0.7cm}
		\caption[caption]{Left: $\Sigma$ in $C(-\frac 1n)$. Right: annulus in $T^2\times[0,1]$.\\\null\hspace{2cm} Bottom: $\Sigma$ in $C(-\frac 1n)$ after gluing $T^2\times[0,1]$.}
		\label{half-Giroux-torsion}
		\end{center}
	 \end{figure}

	(2) {\it $\Gamma_\Sigma$ contains one arc and one closed curve parallel to $\vects{1}{0}$}: Observe that the relative Euler class evaluated on $\Sigma$ is $0$. In this case, we showed in the proof of Proposition~\ref{bypass-negative-slope} that there exists an isotopic copy of $\Sigma$ in $C(-\frac 1n)$ with one boundary-parallel dividing arc, one closed boundary-parallel dividing curve and perhaps two closed curves parallel to $\vects{1}{0}$, see the last drawings of Figure~\ref{negative-punctured-torus-5} and Figure~\ref{negative-punctured-torus-7}. We now show that if the two closed dividing curves parallel to $\vects{1}{0}$ are present, then they can be removed.
	
	Cut $C(-\frac 1n)$ along $\Sigma$ and round the edges to obtain $\Sigma\times[0,1]$ with a smooth convex boundary. Take an annulus $c\times[0,1]$, where $c$ is a closed curve on $\Sigma$ parallel to $\vects{0}{1}$ that intersects the dividing curves minimally: there exists a bypass on $\Sigma\times\{0\}$ by Theorem~\ref{imba}. Since $c$ intersects dividing curves on $\Sigma\times\{0\}$ at four points, there are two possible bypasses, but no matter which bypass we have, we can rotate the bypass (Theorem~\ref{bypass-rotation}) so that one end of a bypass intersects the boundary-parallel closed dividing curve. Now we can remove two dividing curves by attaching this bypass, see Figure~\ref{bypass-punctured-torus}.
	
	Now $\Gamma_\Sigma$ has one boundary-parallel arc and one boundary-parallel closed curve, which means that the dividing set of $\bd(\Sigma\times[0,1])$ consists of exactly three dividing curves parallel to $\bd\Sigma$. Fix the signs of the regions so that the bypass is negative. Contact structures on genus--$2$ handlebodies with this particular convex boundary were studied by Cofer \cite[Section~3]{Cofer} in the process of classifying contact structures on $\Sigma_2 \times I$; she determined that for fixed signs of the regions on the boundary of the handlebody (and a fixed singular foliation), there are exactly two tight contact structures up to isotopy, and both are universally tight. Thus, we can construct at most two tight contact structures on $C(-\frac1n)$ from the chosen signs of the regions.

	We claim that one of these contact structures must have a boundary-parallel half Giroux torsion layer. Indeed, consider the tight contact structure $(C(-\frac 1n),\xi^-_{-n})$ obtained in (1). If we glue a negative half Giroux torsion layer to the boundary, we obtain a tight contact structure on $C(-\frac 1n)$, by Theorem~\ref{gluing}, that contains a copy of $\Sigma$ whose dividing set is $\Gamma_\Sigma$, see Figure~\ref{half-Giroux-torsion}.

	By performing the same analysis with the opposite signs on the regions, we see that there are at most two tight contact structures on $C(-\frac 1n)$ without half Giroux torsion when the relative Euler class evaluated on $\Sigma$ is $0$, and that these contact structures are completely determined by the sign of the bypass on $\Sigma$.  Denote these contact structures by $(C(-\frac 1n),\xi'^{\pm}_{-n})$. We need to show that they thicken to $C(\infty)$ for $n > 1$ and are overtwisted for $n = 1$.
	
	Consider a contact structure $\xi_{-1}$ on $M(-1)$ obtained by Legendrian surgery on $L^{-}_0$ in $(S^3,\xi_1^{OT})$ (we use notation from Section~\ref{sec:obd}); this is tight by Theorem~\ref{surgeries on legendrianapproximations}. Let $L$ be a push-off of $L^-_0$, viewed as a knot in $(M(-1), \xi_{-1})$. Then it is not hard to see that $tb(L)=0$, $rot(L)=-1$, and the complement of a standard neighborhood of $L$ is the union of some tight contact structure on $C(\infty)$ with no half Giroux torsion and a negative basic slice (since the same is true for $L^-_{0}$).
	
	Let $L'$ be a positive stabilization of $L$. Then $tb(L')=-1$, $rot(L')=0$, and the complement of a standard neighborhood of $L'$ is $(C(-\frac 12),\xi')$ for some tight contact structure $\xi'$. Since $rot(L')=0$, the relative Euler class of $\xi'$ vanishes on $\Sigma$, so $\xi'$ is isotopic to one of $\xi'^{\pm}_{-2}$. If we thicken $C(-\frac12)$ back to $C(\infty)$, we would start by adding a positive basic slice, so we see that we constructed $\xi'^+_{-2}$. Similarly, repeating the above construction starting with $L^+_0$ (and switching all mentions of {\em negative} and {\em positive}) will construct $\xi'^-_{-2}$.
	
	Suppose $n > 2$. Then we can decompose $(C(-\frac 12), \xi'^\pm_{-2})$ into $T_0^{\pm} \cup T_1^{\pm} \cup C(\infty)$, where $T_0^{\pm} = T^2 \times [0,1]$ with dividing curve slopes $s_0=-\frac 12$ and $s_1=-\frac 1n$, and $T_1^{\pm} = T^2\times [1,2]$ with dividing curve slopes $s_1=-\frac 1n$ and $s_2=\infty$. Since $\xi'^\pm_{-2}$ restricted to $T_1^\pm \cup C(\infty)$ have the right relative Euler class evaluation and no half Giroux torsion, we conclude that these restrictions are in fact the $\xi'^\pm_{-n}$ that we are looking for, and they evidently thicken to $C(\infty)$ and are determined by the contact structure on $T$.
	
	For $n=1$, we decompose $C(-1)$ into $T^{\pm} \cup (C(-\frac 12),\xi')$, where $T^{\pm}$ is a basic slice with dividing curve slopes $s_0=-1$ and $s_1=-\frac 12$. However, this implies that $C(-1) = T^{\pm} \cup T' \cup C(\infty)$, for some $T' = T^2 \times I$ composed of two bypass layers of mixed sign. Since the dividing curve slopes on the boundary of $T^{\pm} \cup T'$ are connected by an edge in the Farey tessellation, and the signs of the bypasses in $T'$ are mixed, $T^{\pm} \cup T'$ is overtwisted, by Theorem~\ref{basic-slice}. Hence, $\xi'^\pm_{-1}$ are overtwisted, as required.
	
	Once we have thickened $C(s)$ to $C(\infty)$, we claim that it does not thicken further. If $C(\infty)$ thickens further, then as above, it thickens to $C(-3)$ and stops, or thickens all the way to $C(\infty)$. If the latter is true, then it contains half Giroux torsion, contrary to our hypothesis. If it thickens to $C(-3)$ and no further, then by Lemma~\ref{complement count -3}, the tight contact structure on $C(\infty)$ must be contactomorphic to the complement of a Legendrian figure-eight knot in $(S^3, \xi_{\rm std})$ glued to a $T^2 \times I$ such that the resulting dividing curves on the boundary are meridional. By the first author's study \cite{Conway:f8} of positive contact surgeries on the Legendrian figure-eight knot in tight $S^3$, any such contact structure on $C(\infty)$ is overtwisted. Thus, if $C(\infty)$ is tight, it cannot thicken.

	Next, consider the positive case $s>0$. Since $s \not\in (0,1] \cup (4, 5]$, we can find a bypass for $-\bd C(s)$ along an arc in $\bd\Sigma$ using Proposition~\ref{bypass-positive-slope}. As above, we can thicken $C(s)$ by attaching this bypass to $-\bd C(s)$, and by repeating this, we eventually obtain a tight contact structure on $C(\infty)$. As above, this $C(\infty)$ cannot thicken any further.
\end{proof}

\subsection{Tight contact structures on \texorpdfstring{$C(\infty)$}{C(infty)}} \label{sec:Cinfty} In this section, we study (possibly) tight contact structures on $C(\infty)$ without boundary-parallel half Giroux torsion and prove Lemma~\ref{complement count infinity}. As before, we will normalize the dividing curves on a copy of $\Sigma$ in $C(s)$, and after cutting $C(s)$ along $\Sigma$ and rounding the edges, arrive at a genus--$2$ handlebody. Denote the dividing set on $\Sigma$ by $\Gamma_\Sigma$. By Proposition~\ref{normalizing-dividing-set-2}, there are two possibilities for $\Gamma_\Sigma$.

(1) {\it $\Gamma_\Sigma$ contains a boundary-parallel arc}: In this case, there is a bypass for $-\bd C(s)$ along $\Sigma$ with slope $0$, so we can thicken $C(\infty)$ to $C(0)$ by attaching the bypass. We can find a convex torus in $C(\infty)\setminus C(0)$ with slope $-1$, and hence find $C(-1)$ inside $C(\infty)$. Then by Lemma~\ref{thickening}, we know that $C(-1)$ thickens to $C(\infty)$, which implies that $C(\infty)$ contains a boundary-parallel half Giroux torsion layer, contrary to our hypothesis.

\begin{figure}[htbp]
	\begin{center}
	\begin{overpic}[scale=1.2,tics=20]{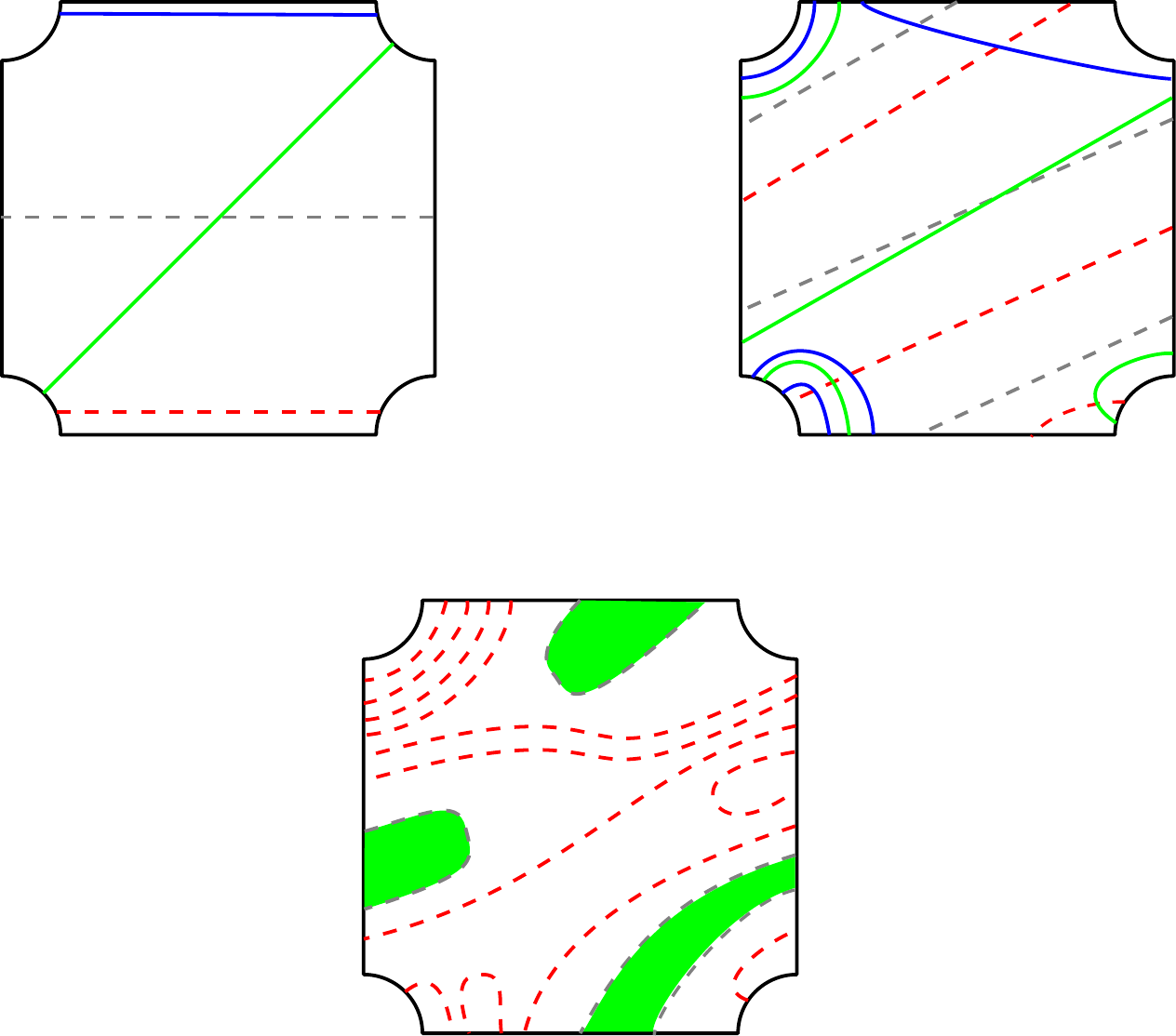}
	\put(60,205){$\Sigma\times\{1\}$}
	\put(340,205){$\Sigma\times\{0\}$}
	\put(200,-20){$\Sigma\times\{\frac12\}$}
	\put(60,370){$D_0$}
	\put(60,274){$D_1$}
	\put(323,278){$D_1$}
	\put(383,355){$D_0$}
	\end{overpic}
	\vspace{0.7cm}
	\caption{The top two drawings are $\Sigma \times \{0,1\}$ in $C(\infty)$. The dotted lines are dividing curves. The blue and green lines are the intersections of $D_0$ and $D_1$ and $\Sigma\times\{0,1\}$. The bottom drawing is $\Sigma \times \{\frac12\}$ in $C(\infty)$, after attaching bypasses along $D_0$ and $D_1$ that result in an overtwisted disc; the gray dotted lines give the contractible dividing curve.}
	\label{infty-punctured-torus-1} 
	\end{center}
 \end{figure}
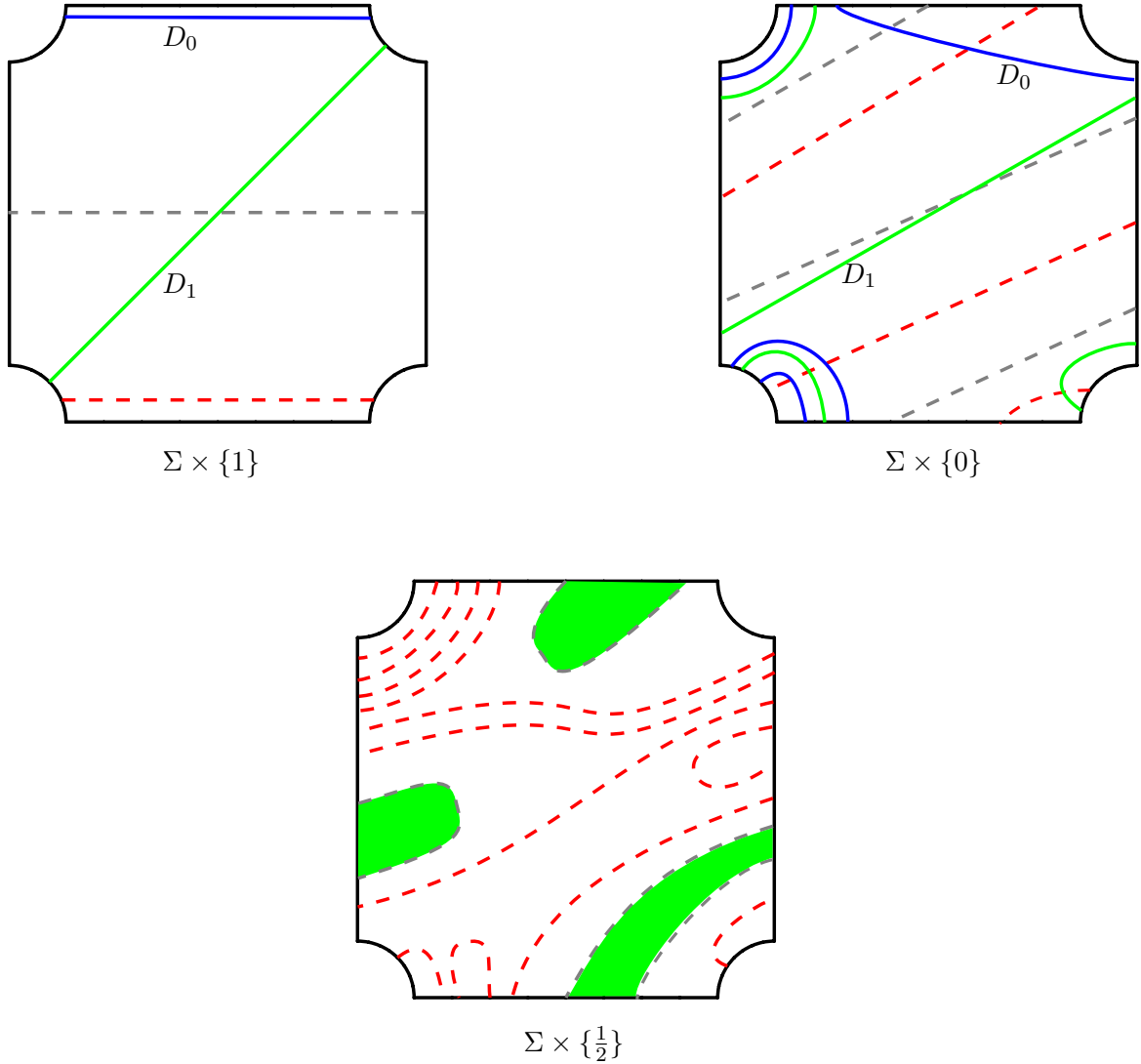

(2) {\it $\Gamma_\Sigma$ contains one arc and one closed curve both parallel to $\vects{1}{0}$}: Take compressing discs $D_0$ and $D_1$ for $\Sigma\times[0,1]$ (as in Section~\ref{sec:thickening}) by taking the product of an arc and $[0,1]$, where the subscript indicates the slope of the arc, as shown in the top drawings of Figure~\ref{infty-punctured-torus-1}. After a perturbation so that the discs intersect no dividing curves on $\bd\Sigma \times [0,1]$, the boundary of each disc intersects dividing curves at four points. Suppose there exists a bypass for $\Sigma\times\{0\}$ along $\bd D_1$ that straddles the closed dividing curve and a bypass for $\Sigma\times\{0\}$ along $\bd D_0$ that straddles the dividing arc. After attaching these bypasses, we obtain a contractible dividing curve, so the contact structure is overtwisted, see the bottom of Figure~\ref{infty-punctured-torus-1}.

Fix the signs of the regions of $\Sigma$. Since we have ruled out one possible combination of dividing curves on $D_i$, there are at most three tight contact structures on $C(\infty)$ up to isotopy that induce these signs of the regions. Now suppose there exist bypasses for $\Sigma\times\{0\}$ along $\bd D_0$ and $\bd D_1$ which both straddle the closed dividing curve. This setup corresponds to one of the three contact structures on $C(\infty)$ that we denote by $\xi$. We will show that $\xi$ is overtwisted.

\begin{lem} \label{infty overtwisted}
	$(C(\infty),\xi)$ is overtwisted.
\end{lem}

\begin{proof}
	In this proof, we use $\Gamma_t$ to denote the dividing set of a convex copy of $\Sigma_t = \Sigma \times \{t\} \subset \Sigma \times [0,1]$. (Recall that since we create $\Sigma_t$ by attaching bypasses to other boundary-parallel surfaces and then isotoping the surface to become disjoint, it follows that $\Sigma_t$ need not be a fiber in the natural fibration $\Sigma \times [0,1] \to [0,1]$.) Then we have $\Gamma_0$ and $\Gamma_1$ as pictured in Figure~\ref{infty-punctured-torus-1}. After attaching the bypass along $\bd D_0$ that straddles the closed dividing curve in $\Gamma_0$, we arrive at $\Sigma_{1/2}$, where $\Gamma_{1/2} = \Gamma_1$.

	Note that $\xi$ restricted to $\Sigma \times [\frac12, 1]$ is not $I$--invariant, since there exists a non-trivial bypass for $\Sigma_{1/2}$ along $D_1 \cap \Sigma_{1/2}$, and $I$--invariant structures only support trivial bypasses (see \cite[Proposition~4.10]{Giroux:bundles}).  Take a compressing disc $D$ for $\Sigma \times [\frac12, 1]$ as shown in Figure~\ref{infty-punctured-torus-2}. We claim that there exists a bypass for $\Sigma_1$ from the back along $D$ that straddles the closed dividing curve in $\Gamma_1$, which gives rise to a non-trivial bypass. To see this, note that a second compressing disc parallel to $\vects{1}{0}$ can be arranged to intersect only two dividing curves on $\partial(\Sigma \times [0,1])$, and thus the contact structure on $\Sigma \times [0,1]$ is determined by the dividing set on $D$. Since there are two possible dividing sets for $D$, one must correspond to the $I$--invariant structure, and one to our contact structure. Since the $I$--invariant contact structure cannot have a non-trivial bypass, it must be our contact structure that has the bypass. Now, after attaching the non-trivial bypass for $\Sigma_1$ from the back along $\partial D$, we obtain $\Sigma_{3/4}$, where $\Gamma_{3/4}$ is identical to $\Gamma_0$.

	\begin{figure}[htbp]
		\begin{center}
		\begin{overpic}[scale=1.2,tics=20]{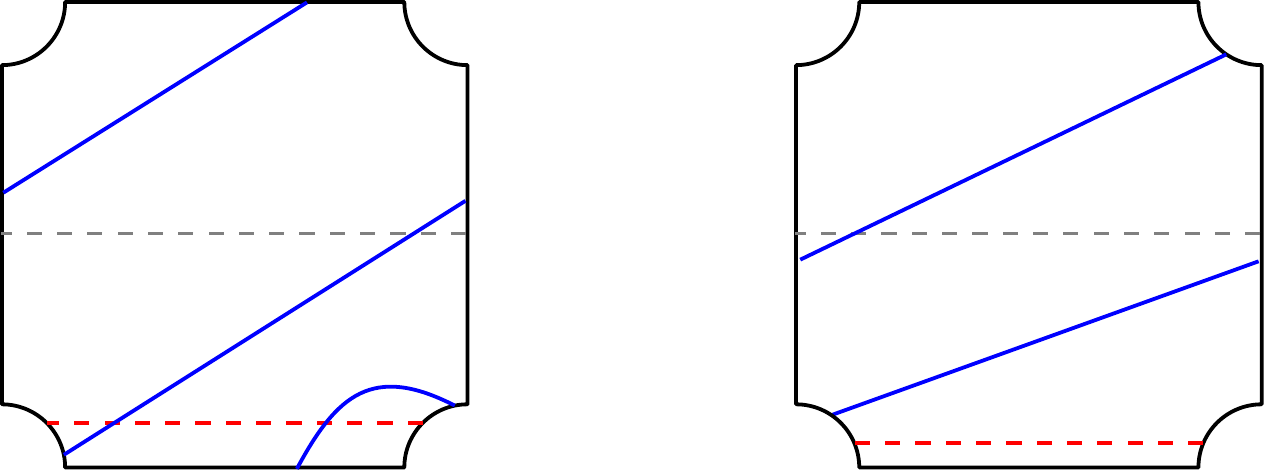}
		\put(60,-20){$\Sigma\times\{1\}$}
		\put(340,-20){$\Sigma\times\{\frac 12\}$}	
		\end{overpic}
		\vspace{0.7cm}
		\caption{The blue arcs are the intersections of the compressing disc $D$ and $\Sigma \times \{\frac12, 1\}$.}
		\label{infty-punctured-torus-2}
		\end{center}
	\end{figure}
	
	Next, let $\Sigma \times [1, \frac32]$ be $\phi^{-1}\left(\Sigma \times [0, \frac12]\right)$ and let $\Sigma \times [\frac32, 2]$ be $\phi^{-1}\left(\Sigma \times [\frac12, 1]\right)$. These naturally inherit a contact structure from $\xi$ by thinking of them as $\Sigma \times [0, \frac12]$ and $\Sigma \times [\frac12, 1]$ swung once around the fibration in $C(\infty)$, and so by gluing these pieces together we obtain a contact structure on $\Sigma \times [0, 2]$. We claim that this contact structure on $\Sigma \times [0,2]$ is overtwisted.
	
	First note that $\Gamma_1 = \Gamma_{1/2} = \Gamma_{7/4}$, the former equality from the first paragraph, and the latter since $\Gamma_{7/4} = \phi^{-1}(\Gamma_{3/4}) = \phi^{-1}(\Gamma_0) = \Gamma_1$. Now take a compressing disc for $\Sigma \times [\frac12, 1]$ as shown in Figure~\ref{infty-punctured-torus-3}. As above, since the contact structure on $\Sigma \times [\frac12, 1]$ is not $I$--invariant, we obtain a bypass along $\Sigma_{1/2}$; after attaching this bypass, we obtain $\Sigma_{t_0}$ for some $t_0 \in (\frac12,1)$, where $\Gamma_{t_0}$ contains dividing curves as in Figure~\ref{infty-punctured-torus-4}.
	
	Now take a compressing disc for $\Sigma \times [1, \frac74]$, as shown in Figure~\ref{infty-punctured-torus-2}. As above, there exists a non-trivial bypass for $\Sigma_{7/4}$ from the back along the compressing disc, and after attaching the bypass, we obtain $\Sigma_{t_1}$ for some $t_1 \in (1, \frac74)$, and $\Gamma_{t_1} = \Gamma_0$. Then Figure~\ref{infty-punctured-torus-4} shows the boundary of an overtwisted disc in $\Sigma \times [t_0, t_1]$.

	\begin{figure}[htbp]
		\begin{center}
		\begin{overpic}[scale=1.2,tics=20]{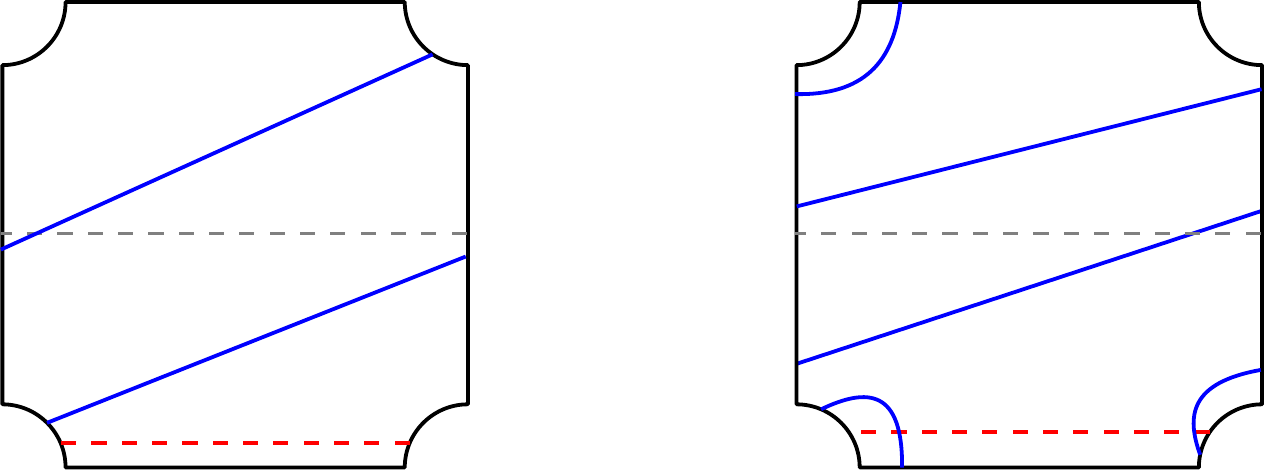}
		\put(60,-20){$\Sigma\times\{1\}$}
		\put(340,-20){$\Sigma\times\{\frac 12\}$}	
		\end{overpic}
		\vspace{0.7cm}
		\caption{The blue arcs are the intersections of the compressing disc and $\Sigma\times\{\frac 12,1\}$.}
		\label{infty-punctured-torus-3}
		\end{center}
	 \end{figure}
	
	 \begin{figure}[htbp]
		\begin{center}
		\begin{overpic}[scale=1.2,tics=20]{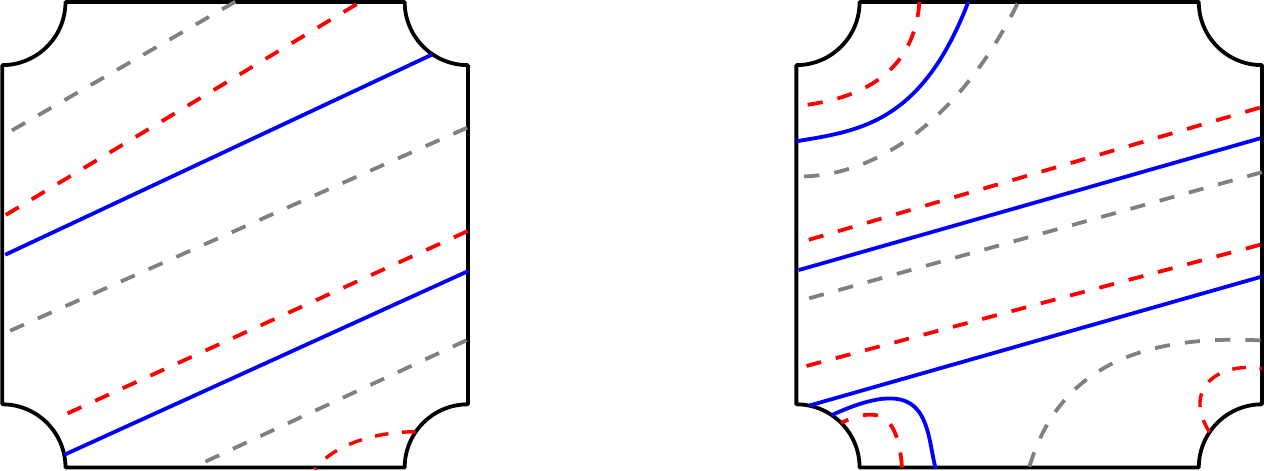}
		\put(60,-20){$\Sigma\times\{t_1\}$}
		\put(340,-20){$\Sigma\times\{t_0\}$}	
		\end{overpic}
		\vspace{0.7cm}
		\caption{The blue arcs are the intersections of an overtwisted disc and $\Sigma\times\{t_0,t_1\}$.}
		\label{infty-punctured-torus-4}
		\end{center}
	 \end{figure}

	 \begin{figure}[htbp]
		\begin{center}
		\begin{overpic}[scale=1.2,tics=20]{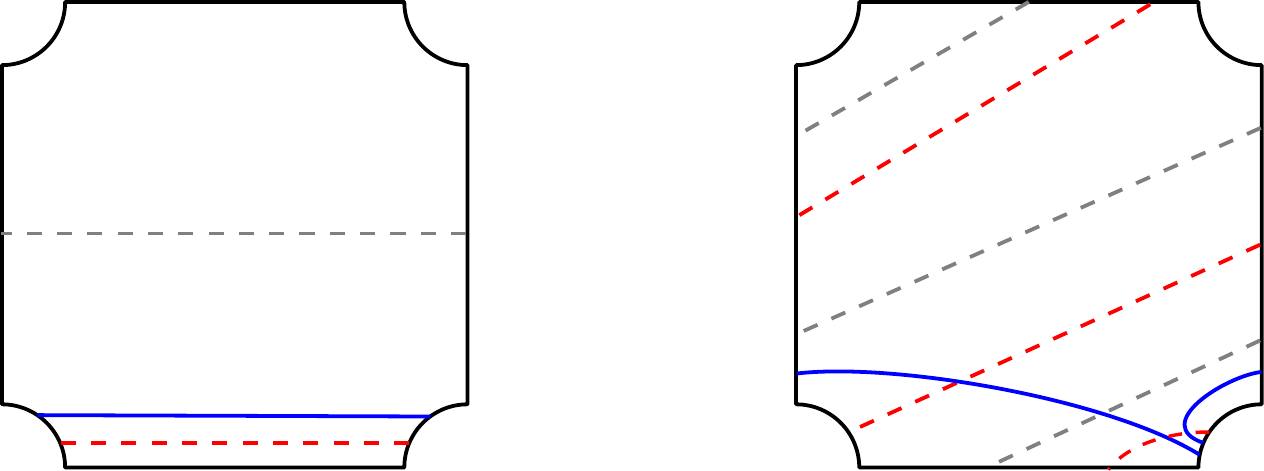}
		\put(60,-20){$\Sigma\times\{1\}$}
		\put(340,-20){$\Sigma\times\{\frac 34\}$}	
		\end{overpic}
		\vspace{0.7cm}
		\caption{The blue arcs are the intersections of the the compressing disc and $\Sigma\times\{\frac 34,1\}$.}
		\label{infty-punctured-torus-5}
		\end{center}
	 \end{figure}
	
	 \begin{figure}[htbp]
		\begin{center}
		\begin{overpic}[scale=1.2,tics=20]{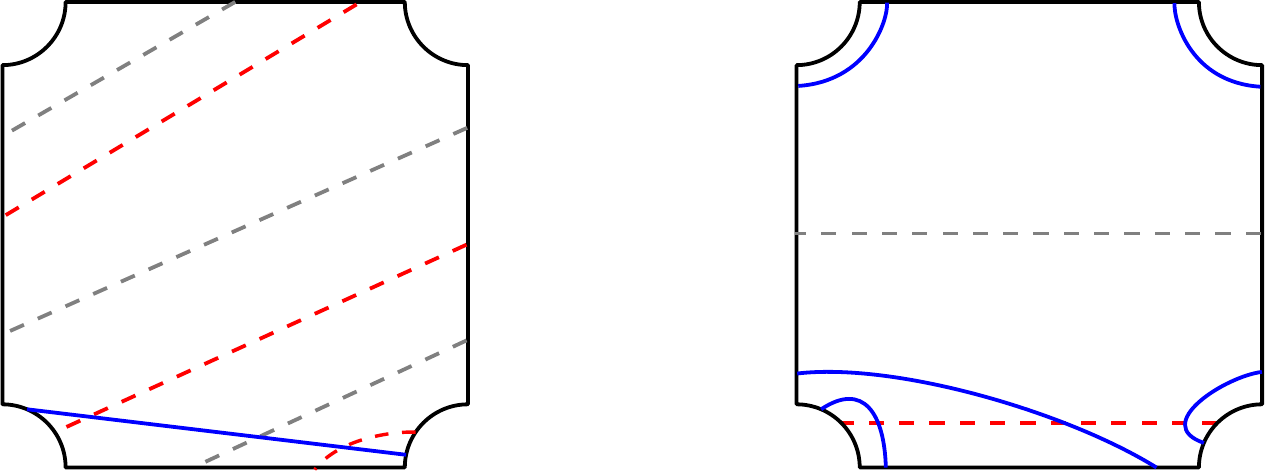}
		\put(60,-20){$\Sigma\times\{\frac 34\}$}
		\put(340,-20){$\Sigma\times\{\frac 12\}$}	
		\end{overpic}
		\vspace{0.7cm}
		\caption{The blue arcs are the intersection of the compressing disc and $\Sigma\times\{\frac 12,\frac 34\}$.}
		\label{infty-punctured-torus-6}
		\end{center}
	 \end{figure}

	 \begin{figure}[htbp]
		\begin{center}
		\begin{overpic}[scale=1.2,tics=20]{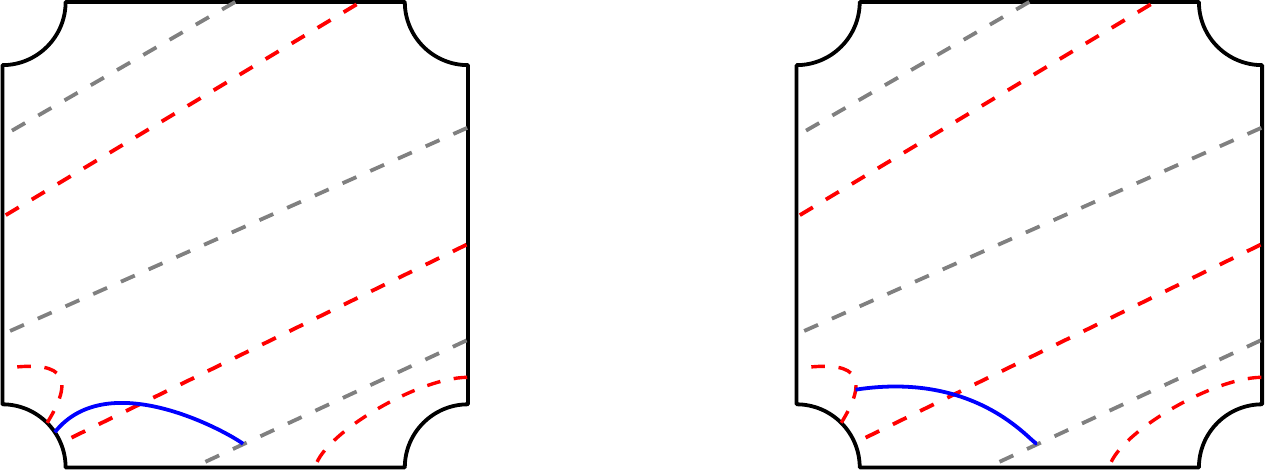}
		\put(60,-20){$\Sigma\times\{\frac 34\}$}
		\put(340,-20){$\Sigma\times\{\frac 34\}$}	
		\end{overpic}
		\vspace{0.7cm}
		\caption{The blue arc in the left picture is part of the attaching arc of a bypass from the back. The blue arc in the right picture is the attaching arc of a bypass after a bypass rotation.}
		\label{infty-punctured-torus-7}
		\end{center}
	 \end{figure}
	
	 \begin{figure}[htbp]
		\begin{center}
		\begin{overpic}[scale=1.2,tics=20]{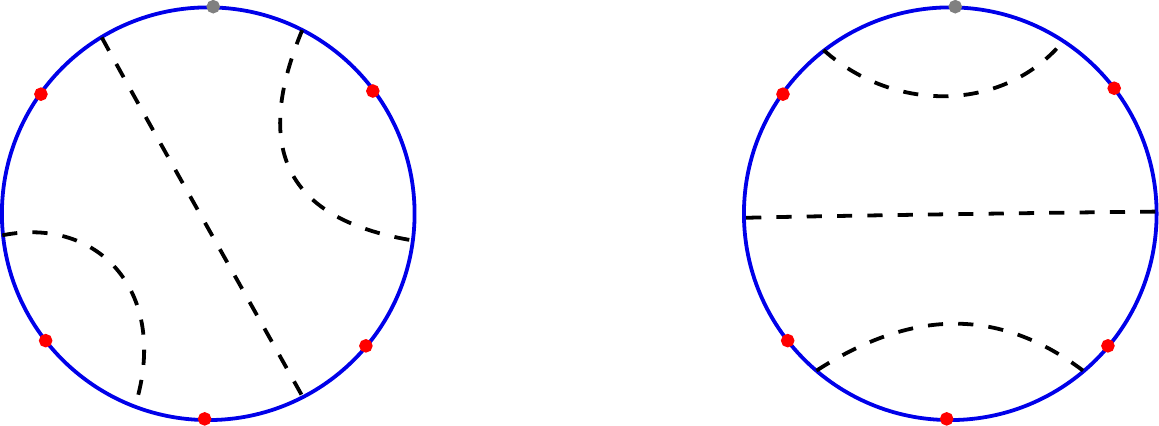}
		\end{overpic}
		\caption{Some possible dividing sets on the disc.}
		\label{bypass-disc}
		\end{center}
	 \end{figure}
	 
	We have just shown that a double cover of $(C(\infty), \xi)$ is overtwisted. We will now consider $\Sigma \times [0,1]$, and hence $(C(\infty), \xi)$ itself. We will first show that we can take $\Sigma_{3/4}$ and $\Sigma_{t_0}$ to be disjoint (recall that they are not fibers of $\Sigma \times [0,1] \to [0,1]$, but just boundary-parallel surfaces, so a priori they might intersect). Take a compressing disc for $\Sigma \times [\frac34,1]$ as shown in Figure~\ref{infty-punctured-torus-5}. If there exists a bypass for $\Sigma_{3/4}$ that straddles the dividing arc, then attaching this bypass will give us $\Gamma_{t_0}$, and hence we will have found a surface $\Sigma_{t_0}$ disjoint from $\Sigma_{3/4}$. If not, then there is a bypass that straddles the closed dividing curve.

	Now take a compressing disc for $\Sigma \times [\frac12, \frac34]$ as shown in Figure~\ref{infty-punctured-torus-6}. It is not hard to check that any configuration of dividing curves on the compressing disc will give rise to a bypass that either straddles the closed dividing curve on $\Gamma_{3/4}$ or one of the dividing arcs on $\Gamma_{3/4}$; see Figure~\ref{bypass-disc} for some examples.  If it straddles the closed dividing curve, then this bypass and the one on the other side (from the previous paragraph) glue together to form an overtwisted disc. Now assume there exists a bypass that straddles one of the arcs of $\Gamma_{3/4}$ (the bypass will not be entirely contained in either $\Sigma_{3/4}$ or $\Sigma_{1/2}$). After a bypass rotation (Theorem~\ref{bypass-rotation}), as in Figure~\ref{infty-punctured-torus-7}, we can attach a bypass to $\Sigma_{3/4}$ from the back that results in one boundary-parallel dividing arc and one boundary-parallel closed curve. Appealing again to Cofer's classification \cite{Cofer}, such a setup supports at most two tight contact structures, and ours is not the $I$--invariant one. It is straightforward to see that ours can be obtained by gluing a basic slice to $(C(-1),\xi'^\pm_{-1})$, and since we have already shown in the previous section that $\xi'^\pm_{-1}$ is overtwisted, $\xi$ is also overtwisted.
	
	Thus, we assume that $\Sigma_{t_0}$ can be taken to be disjoint from $\Sigma_{3/4}$, and hence can be taken to be in the interior of $\Sigma \times [\frac34, 1]$. Since $t_1 \in (1, \frac74)$, we know that $\Sigma_{t_1}$ is contained in the interior of $\Sigma \times [1, \frac74]$. Now by cutting open $C(\infty)$ along $\Sigma_{3/4}$ instead of along $\Sigma_0$, we see that $\Sigma \times [\frac34,\frac74]$ embeds inside $C(\infty)$. In particular, the overtwisted $\Sigma \times [t_0, t_1]$ embeds in $(C(\infty), \xi)$, showing that the latter is overtwisted.
\end{proof}

\begin{figure}[htbp]
	\begin{center}
	\begin{overpic}[scale=1.1,tics=20]{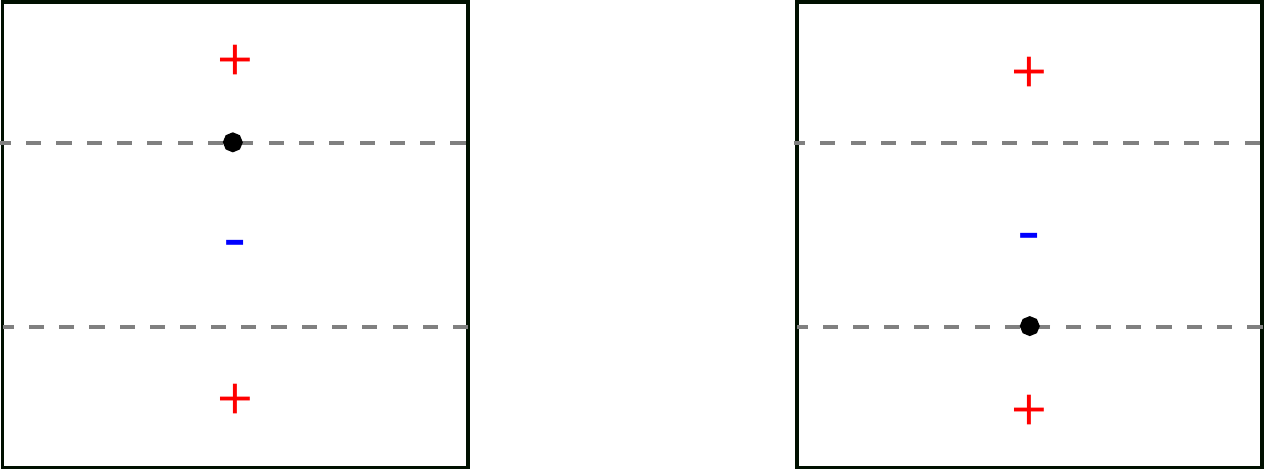}
	\put(78,94){$p$}
	\put(333,50){$p$}
	\end{overpic}
	\caption{Two possible choices of a point $p$ on a convex torus.}
	\label{points-on-torus}
	\end{center}
 \end{figure}

\begin{proof}[Proof of Lemma~\ref{complement count infinity}]
	By Lemma~\ref{infty overtwisted} and the discussion preceding it, there are at most four tight contact structures on $C(\infty)$ without boundary-parallel half Giroux torsion up to isotopy fixing a given singular foliation on the boundary: two from each choice of signs of the regions on the boundary. We will explicitly construct these contact structures from the torus bundle $M(0)$. 
	
	We first explicitly describe the two minimally twisting tight contact structures on $M(0)$ given by Honda's classification \cite[Table~2]{Honda:classification2}. Take a basic slice $T^2\times[0,1]$ with dividing curve slopes $s_0=\vects{2}{1}$ and $s_1=\vects{1}{0}$ and glue $T^2 \times \{1\}$ to $T^2 \times \{0\}$ via $\phi$. The two different signs of basic slice give rise to two different contact structures on $M(0)$, which we denote by $\xi^\pm$.

	Fix a tight contact structure on $M(0)$, and pick a point $p$ on a dividing curve of $T\times\{1\}$. Since $\phi$ is isotopic to a diffeomorphism of $T$ which fixes a neighborhood of $p$ (possibly also by isotoping the dividing curves), we can assume that the knot \[ L = \left(p\times[0,1]\right)\big/(p,1)\sim(p,0) \] is Legendrian. It is clear from the construction of $\xi^\pm$ (and it is explicit in Honda's table) that the contact planes twist through an angle smaller than $\pi$ as they traverse $L$. Thus, the contact framing of $L$ is identical to the product framing. Thus, the complement of the interior of a standard neighborhood of $L$ is a tight contact structure on $C(\infty)$.
	
	The signs of the regions on $\Sigma$ in the contact structure on $C(\infty)$ will be the same as the signs on the regions of $T \times \{1\}$, see Figure~\ref{points-on-torus}. But instead, we could have chosen $p$ to be on the other dividing curve, constructing a Legendrian knot $L'$; this would give us a contact structure on $C(\infty)$ with the opposite signs on the regions of $\Sigma$. However, it is straightforward to see that $L$ and $L'$ are isotopic Legendrian knots, which is what we desired to show.
	
	We know that we have indeed constructed all four of our contact structures on $C(\infty)$, since if the contact structures on $C(\infty)$ coming from $\xi^+$ and $\xi^-$ were isotopic, then they would produce isotopic contact structures on $M(0)$ after filling in the boundary. Since this does not happen, we know that the contact structures on $C(\infty)$ coming from $\xi^\pm$ are distinct. To show that these contact structures do not thicken, note that they do not have boundary-parallel Giroux torsion (as they arise from tight contact structures on closed $3$--manifolds), and hence do not thicken, as in the proof of Lemma~\ref{thickening}.
\end{proof}

\subsection{Tight contact structures on a solid torus} \label{sec:solid-torus} In this section, we prove Lemma~\ref{tight N}. The main ingredient of the proof is Theorem~\ref{solid torus}, which counts tight contact structures on a solid torus with $\infty$ meridional slope. Since the meridional slope $r$ of $N(\infty)$ and $N(-3)$ is not $\infty$, we first need to change the framing. Recall our convention that slopes $\frac pq$ on $T^2$ correspond to vectors $\vects{p}{q}$.

\begin{proof}[Proof of Lemma~\ref{tight N}]
	Recall that $N(\infty)$ is a solid torus with two dividing curves on its boundary, where the meridional slope is $r$ and the dividing curve slope is $\infty$. Let $r=\frac pq$, for a pair of co-prime integers $p, q$. We first assume that $\frac pq \in (0, 1]$.  Suppose $-\frac qp=[r_0,\ldots,r_n]$, where $r_0 \leq -1$ and $r_i \leq -2$, for $i=1,\ldots,n$. For some $p'$ and $q'$, we can write the following (see, for example, \cite{Rose}):
	\[ \matrixb{-r_0}{1}{-1}{0}\matrixb{-r_1}{1}{-1}{0}\cdots\matrixb{-r_n}{1}{-1}{0} = \matrixb{q}{q'}{-p}{-p'}. \]
	Then after changing coordinates by the matrix \[ \matrixb{q'}{-p'}{-q}{p},\] the solid torus $N(\infty)$ will have meridional slope $\infty$ and dividing curve slope $-\frac{q'}{q}$. The conditions on the $r_i$ imply that $-\frac{q'}{q} \in [-1,0)$, and so we can now invoke Theorem~\ref{solid torus} to conclude that $N(\infty)$ supports \[ \left|(r_n+1)\cdots(r_1+1)r_0\right| = \Phi(p/q)\] tight contact structures, since $-\frac{q}{q'}=[r_n,\ldots,r_0]$. This latter fact follows from taking the transpose of the matrix factorization above.

	If $p/q \not \in(0,1]$, then we first change coordinates via the matrix \[\matrixb{1}{k}{0}{1}\] for a unique $k \in \Z$, which fixes the dividing slope $\infty = \vects{1}{0}$ and takes the meridional slope to $\frac{p+kq}{q} \in (0, 1]$.  Following the previous steps, we arrive at the correct count when we recall that we defined $\Phi(p/q + k) = \Phi(p/q)$ for any integer $k$.

	Now consider $N(-3)$. After changing coordinates by the matrix \[ \matrixb{0}{1}{-1}{-3}, \] the solid torus $N(-3)$ will have meridional slope $-\frac{q}{p+3q}$ and dividing curve slope $\infty$. Hence there are $\Phi(-\frac{1}{3+r}) = \Psi(p/q)$ tight contact structures on $N(-3)$.
\end{proof}

\section{Lower Bound}
\label{sec:lowerbound}

In this section, we will construct enough distinct isotopy classes of contact structures on $M(r)$ to meet the upper bound found in Section~\ref{sec:upperbound}.  Our methods in fact construct tight contact structures on $M(r)$ for any rational $r$, but only for $r \in \mathcal R$ do we have an upper bound against which to compare.  We first deal with the case $r \geq 1$; we then construct those contact structures for $r < -3$ that are counted by $\Psi$; we finally construct those counted by $\Phi$ when $r < 0$, and distinguish them from those counted by $\Psi$.

\subsection{Tight contact structures when \texorpdfstring{$r \geq 1$}{r >= 1}} \label{sec:positive-surgery-lower-bound}

Consider the contact surgery diagram given in Figure~\ref{positive-surgery legendrian-diagram}.  Since all the contact surgery coefficients are negative, this describes a tight (in fact, Stein fillable) contact manifold, by \cite{Eliashberg:stein}. In this section, we will prove that all the tight contact structures on $M(r)$ are described by this contact surgery diagram, when $r \in \mathcal R$ is positive (and incidentally construct some tight contact structures on $M(r)$ for $r \in [4, 5)$).

When $r = 1$, we erase the knot $L$ from the surgery diagram; when $r > 1$, then $\frac{1}{1-r}$ is negative, so we can use the algorithm from Section~\ref{sec:surgery} to convert the surgery diagram into a new surgery diagram, where we perform Legendrian surgery on each component of a Legendrian link.  Note that there are choices (of stabilizations) to make when doing this conversion, as described in Section~\ref{sec:surgery}.

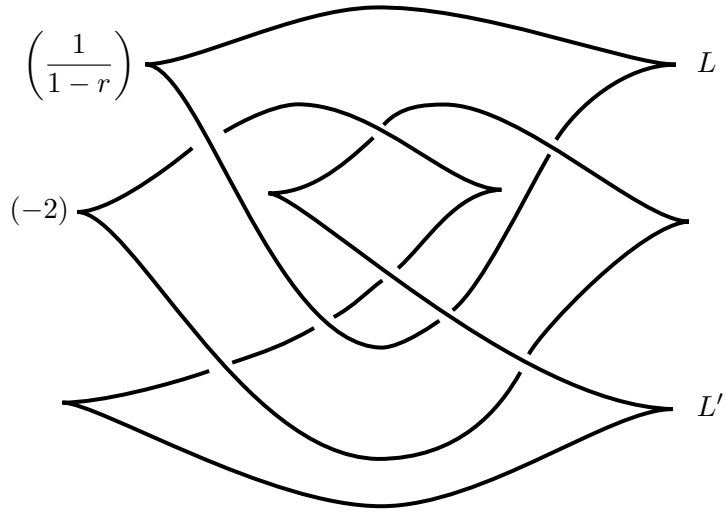
\begin{figure}[htbp]
\begin{center}
\begin{overpic}[scale=2,tics=20]{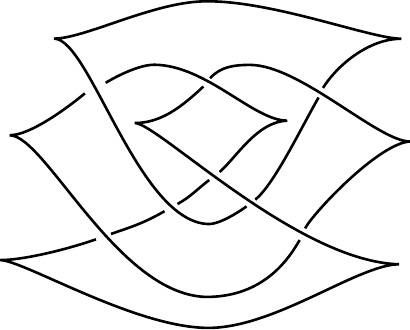}
\put(-15,165){$\left(\displaystyle\frac{1}{1-r}\right)$}
\put(-20,110){$(-2)$}
\put(240,165){$L$}
\put(240,35){$L'$}
\end{overpic}
\caption{A contact surgery diagram for a tight contact structure on $M(r)$, for $r \geq 1$. When $r = 1$, we erase $L$ from the diagram.}
\label{positive-surgery legendrian-diagram}
\end{center}
\end{figure}

\begin{prop}
For any set of stabilization choices, the contact surgery diagram in Figure~\ref{positive-surgery legendrian-diagram} gives a tight contact structure on $M(r)$, when $r \geq 1$.
\end{prop}
\begin{proof}
The resulting contact manifold is tight, as mentioned above. To see that the $3$--manifold is actually $M(r)$, we look at the corresponding smooth surgery diagram in Figure~\ref{positive-surgery smooth-diagram}. Since $tb(L) = -1$ and $tb(L') = 1$, we add these values to the contact surgery coefficient to get the smooth surgery coefficient.  After doing a right-handed Rolfsen twist on $L$ (see \cite{Rolfsen}), we arrive at the surgery diagram in Figure~\ref{positive-surgery rolfsen-twist}.  Blowing down $L'$, which is now a $-1$--framed unknot, gives us a smooth surgery diagram for $r$--surgery on the figure-eight knot, which is $M(r)$.
\end{proof}

\begin{figure}[htbp]
\begin{center}
\begin{overpic}[scale=1,tics=20]{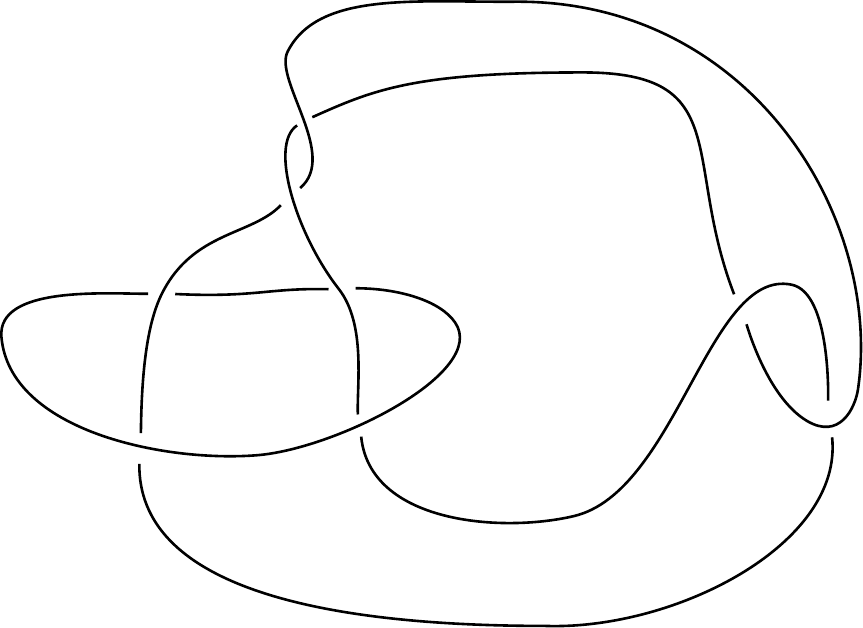}
\put(-23,65){$\displaystyle\frac{r}{1-r}$}
\put(65,160){$-1$}
\put(0,100){$L$}
\put(68,130){$L'$}
\end{overpic}
\caption{The smooth surgery diagram corresponding to the contact surgery diagram in Figure~\ref{positive-surgery legendrian-diagram}.}
\label{positive-surgery smooth-diagram}
\end{center}
\end{figure}

\begin{figure}[htbp]
\begin{center}
\begin{overpic}[scale=.8,tics=20]{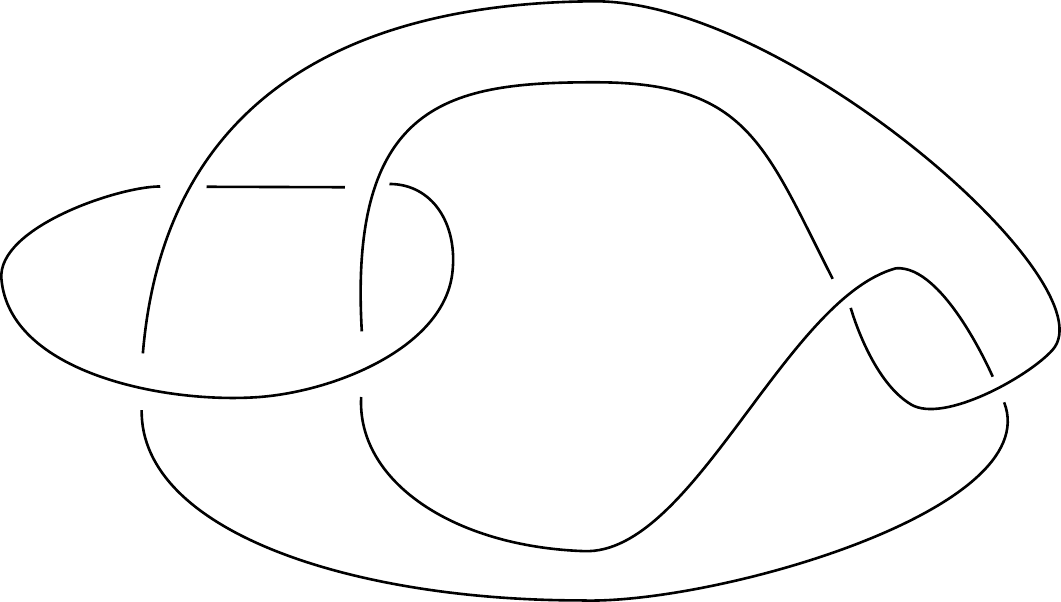}
\put(-10,65){$r$}
\put(100,142){$-1$}
\put(0,90){$L$}
\put(68,130){$L'$}
\end{overpic}
\caption{The result of a right-handed Rolfsen twist on $L$ from Figure~\ref{positive-surgery smooth-diagram}.}
\label{positive-surgery rolfsen-twist}
\end{center}
\end{figure}
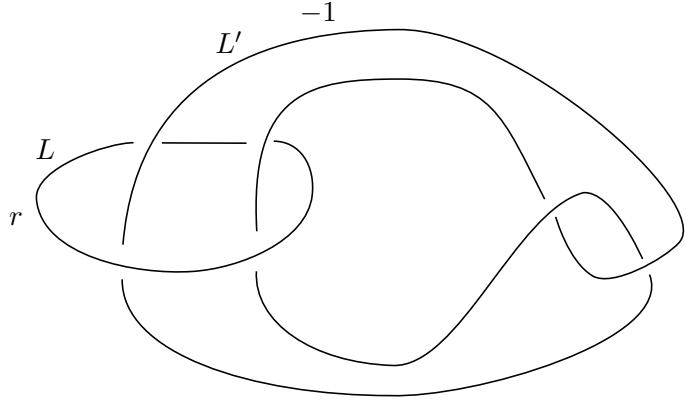

\begin{prop} \label{positive-surgery stabilization}
Different sets of stabilization choices give rise to non-isotopic tight contact structures on $M(r)$, distinguished by their Heegaard Floer contact invariants.
\end{prop}
\begin{proof}
Fix $r \geq 1$, and let $L_0, \ldots, L_n$ be the Legendrian knots involved in converting the contact $\left(\frac{1}{1-r}\right)$--surgery on $L$ into a sequence of Legendrian surgeries, as in Section~\ref{sec:surgery}.  Any two different stabilization choices lead to the same topological cobordism $W$ from $S^3$ to $M(r)$, comprised of $2$--handles attached to $L_0$, \ldots, $L_n$, and $L'$.

We claim that the different choices lead to Stein structures on the cobordism with non-isomorphic underlying Spin$^c$ structures.  Indeed, the first Chern class of the Stein structures will evaluate to $rot(L_i)$ and $rot(L')$ on the set of generators of $H_2(W; \Z)$ given by the cores of the 2--handles union Seifert surfaces for the attaching spheres, see \cite[Proposition~2.3]{Gompf}. Since different choices of stabilizations will lead to different values of $rot(L_i)$ and $rot(L')$, the Spin$^c$ structures underlying the Stein structures are not isomorphic.  By Theorem~\ref{steincobordism}, we conclude that the contact structures on $M(r)$ induced by the different Stein structures are not isotopic, and indeed have different Heegaard Floer contact invariants.
\end{proof}

\begin{prop} \label{positive-surgery count}
There are at least $2\Phi(r)$ tight contact structures on $M(r)$, for $r \geq 1$.
\end{prop}
\begin{proof}
We count the choices that we need to make in order to turn Figure~\ref{positive-surgery legendrian-diagram} into a sequence of Legendrian surgeries.  For $L'$, we need to stabilize once, and there are two ways of doing this.  It remains to show that there are $\Phi(r)$ ways of converting contact $(\frac{1}{1-r})$--surgery on $L$ into a sequence of Legendrian surgeries, then we will be done by Proposition~\ref{positive-surgery stabilization}.

First, if $r = 1$, then we erase $L$ from the surgery diagram, and thus our construction gives two distinct tight contact structures on $M(1)$. Since $\Phi(1) = 1$, we have constructed $2\Phi(1)$ tight contact structures on $M(1)$.

Now assume that $r > 1$.  By the algorithm in Section~\ref{sec:surgery}, in order to convert the surgery on $L$ into a sequence of Legendrian surgeries, we first need to calculate the negative continued fraction of $\frac{1}{1-r} = [r_0, \ldots, r_n]$. Then the total number of stabilization choices is $\left|r_0(r_1+1)\cdots(r_n+1)\right|$.

Let $s = r-1$.  If $s \in (0, 1]$, then the negative continued fraction of $-\frac{1}{s}$ is exactly what is involved in the definition of $\Phi(s)$. But since $r = s+1$, we have that $\Phi(r) = \Phi(s)$, and so we are done if $r \in (1, 2]$. So to finish the proof of the proposition, it is enough to show that the number of stabilization choices for contact $\left(-\frac{1}{s}\right)$--surgery is the same as the number of stabilization choices for contact $\left(-\frac{1}{s+1}\right)$--surgery, for $s > 0$.  This will follow from the fact that
\[ -1-\frac{1}{-1-1/s} = -\frac{1}{s+1}. \]
If the negative continued fraction for $-\frac{1}{s}$ is $[s_0, \ldots, s_n]$, then $-\frac{1}{s+1} = [s'_0, \ldots, s'_{n+1}]$, where $s'_0 = 1$, $s'_1 = s_0 -1$, and $s'_i = s_{i-1}$ for $i = 2, \ldots, n+1$.  It follows that
\[ \left|s'_0(s'_1+1)(s'_2+1)\cdots(s'_{n+1}+1)\right| = \left|s_0(s_1+1)\cdots(s_n+1)\right|. \]
\end{proof}

\subsection{Tight contact structures counted by \texorpdfstring{$\Psi$}{Psi}, when \texorpdfstring{$r < -3$}{r < -3}} \label{sec:psi-lower-bound}

There is a unique Legendrian figure-eight knot $L$ in $(S^3, \xi_{\rm{std}})$ up to isotopy with $tb(L) = -3$ and $rot(L) = 0$, by \cite[Theorem~5.3]{EH:knots}, depicted in Figure~\ref{legendrian-figure-eight}.  Thus, we can create contact structures on $M(r)$, for $r < -3$, by performing contact $(r+3)$--surgery on $L$.  This produces tight (in fact, Stein fillable) contact structures, since $r+3$ is negative when $r < -3$.

\begin{figure}[htbp]
\begin{center}
\begin{overpic}[scale=1,tics=20]{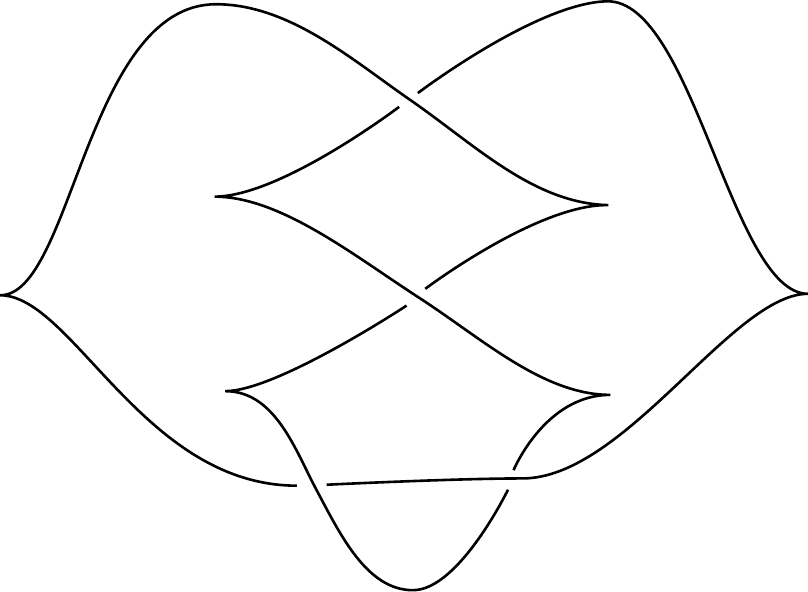}
\end{overpic}
\caption{The unique Legendrian figure-eight with $tb = -3$ and $rot = 0$ in $(S^3, \xi_{\rm{std}})$.}
\label{legendrian-figure-eight}
\end{center}
\end{figure}

As in the case $r \geq 1$, we convert the contact surgery on $L$ into a sequence of Legendrian surgeries via the algorithm in Section~\ref{sec:surgery}. As above, this potentially involves choices of the signs of stabilizations.  The following proposition is proved identically to Proposition~\ref{positive-surgery stabilization}.

\begin{prop}
Different sets of stabilization choices give rise to non-isotopic tight contact structures on $M(r)$, distinguished by their Heegaard Floer contact invariants. \qed
\end{prop}

It remains to verify that we have produced $\Psi(r)$ different isotopy classes of tight contact structures on $M(r)$, for $r < -3$.

\begin{prop} \label{negative-surgery psi-count}
Different choices of contact $(r+3)$--surgery on $L$ give rise to $\Psi(r)$ different isotopy classes of tight contact structures on $M(r)$, for $r < -3$.
\end{prop}
\begin{proof}
The proof here is as in Proposition~\ref{positive-surgery count}. The crux here is that to calculate $\Psi(r) = \Phi(-\frac{1}{r+3})$, we need to calculate the negative continued fraction of
\[ \frac{-1}{-1/(r+3)} = r+3, \]
which is the same calculation involved in counting the number of stabilization choices.
\end{proof}

\subsection{Tight contact structures counted by \texorpdfstring{$\Phi$}{Phi}, when \texorpdfstring{$r < 0$}{r < 0}} \label{sec:phi-lower-bound}

We start by considering the contact structure $\xi$ on $K$ supported by the figure-eight knot $K$.  Since the open book given by $K$ is a negative stabilization of the open book for $(S^3, \xi_{\rm{std}})$ with annular pages, we know that $\xi$ is overtwisted.  We will not need the calculation here, but it is not hard to work out that $d_3(\xi) = 1$, so that $\xi$ is isotopic to $\xi^{OT}_1$.

By Theorem~\ref{legendrianapproximations} and Theorem~\ref{surgeries on legendrianapproximations}, there exist Legendrian figure-eight knots $L^\pm_n$ in $(S^3, \xi^{OT}_1)$ such that Legendrian surgery on $L^\pm_n$ is tight and has non-vanishing Heegaard Floer contact invariant.  This latter follows from the fact that $\overline{K}$ also supports $\xi^{OT}_1$, and $c(\xi^{OT}_1) = 0$, since $\xi^{OT}_1$ is overtwisted, by \cite[Theorem~1.4]{OS:contact}.  In this section, we will show that all the tight contact structures counted by $\Phi$ on $M(r < 0)$ for $r \in \mathcal R$ arise by negative contact surgery on $L^\pm_n$.

Let $r \in [n, n+1)\cap\Q$, for some integer $n \leq -1$, and write $r = n + 1 - s$, where $s \in (0, 1]$.  Then contact $(-s)$--surgery on $L^\pm_{n+1}$ gives a contact structure on $M(r)$.  Recall that there are potentially different contact structures arising from different stabilization choices when converting the contact surgery to a sequence of Legendrian surgeries.

\begin{prop}
For any $s \in (0, 1]$ and $n \leq 0$, contact $(-s)$--surgery on $L^\pm_n$ gives a tight contact structure on $M(r)$, for any set of stabilization choices.
\end{prop}
\begin{proof}
Let $L$ be the knot on which we are doing surgery.  Via the algorithm in Section~\ref{sec:surgery}, we turn contact $(-s)$--surgery on $L$ into a sequence of Legendrian surgeries on $L_0, \ldots, L_m$. Since $-s \in [-1, 0)$, it follows that $L_0$ is isotopic to $L^\pm_n$.  Thus, if we do Legendrian surgery on $L_0$ first, then we arrive at a tight contact structure on $M(n-1)$ with non-vanishing Heegaard Floer contact invariant, by Theorem~\ref{surgeries on legendrianapproximations}.  The rest of the Legendrian surgeries will preserve this non-vanishing, by \cite[Theorem~4.2]{OS:contact}, and hence preserve tightness, by \cite[Theorem~1.4]{OS:contact}.
\end{proof}

Since we wish to construct new contact structures on $M(r)$ that we have not already constructed, we now prove the following.

\begin{prop} \label{negative-surgery phi-is-not-psi}
For $r < -3$, the tight contact structures on $M(r)$ coming from surgery on $L^\pm_n$ are distinct from those enumerated in Section~\ref{sec:psi-lower-bound}. In fact, the Heegaard Floer contact invariants of one set is disjoint from those of the other set.
\end{prop}
\begin{proof}
Both sets of contact structures on $M(r)$ arise via surgery on knots in $S^3$, but starting from different contact structures on $S^3$. However, in both cases, the same smooth cobordism $W$ from $S^3$ to $M(r)$ is built from the surgeries.  We will use this to argue that the sets of Heegaard Floer contact invariants of the two classes of contact structures are disjoint.

Recall that given a smooth cobordism $W$ from $S^3$ to $M(r)$, we can turn it around to get a smooth cobordism (which we call $\overline{W}$) from $-M(r)$ to $-S^3$.  From this, we can build a map $F_{\overline{W}} : \HFhat(-M(r)) \to \HFhat(-S^3)$, and if $W$ is built only of $2$--handles and $J$ is a Stein structure on $W$ giving a Stein cobordism from $(S^3, \xi)$ to $(M(r), \xi')$, then $F_{\overline{W}}(c(\xi')) = c(\xi)$, by \cite[Theorem~2.3]{LS:seifertsurgery}, where $c(\cdot)$ is the Heegaard Floer contact invariant.

Let $(M(r), \xi)$ be any tight contact structure constructed in Section~\ref{sec:psi-lower-bound}.  Since we built this contact manifold via Legendrian surgery on a link in $(S^3, \xi_{\rm{std}})$, we know that
\[ F_{\overline{W}}(c(\xi)) = c(\xi_{\rm{std}}) \neq 0 \in \HFhat(-S^3). \]
Now let $(M(r), \xi')$ be any tight contact structures constructed via Legendrian surgery on $L^\pm_n$. Since $L^\pm_n$ is a Legendrian knot in an overtwisted contact structure $\xi^{OT}_1$ on $S^3$, and we know that $c(\xi^{OT}_1) = 0$, then
\[ F_{\overline{W}}(c(\xi')) = c(\xi^{OT}_1) = 0 \in \HFhat(-S^3). \]
Thus $c(\xi') \neq c(\xi)$ for any choices of $\xi$ and $\xi'$, and so $\xi$ and $\xi'$ are not isotopic, for any choices of $\xi$ and $\xi'$.
\end{proof}

Note that we have completed the classification of tight contact structures on $M(n)$ when $n \neq -4$ is a negative integer. In Section~\ref{sec:upperbound}, we proved that for $n = -1, -2, -3$, then there is at most one tight contact structure on $M(n)$, whereas when $n < -4$, there are at most $|n|-2$ tight contact structures.  Since we constructed $\Psi(n) = |n|-3$ of them in Section~\ref{sec:psi-lower-bound}, and the result $(M(n), \xi_n)$ of Legendrian surgery on $L^-_{n+1}$ is not isotopic to any of those counted by $\Psi$, by Proposition~\ref{negative-surgery phi-is-not-psi}, we have a complete list of tight contact structures on $M(n)$ for negative integers $n \neq -4$.

Since the result of Legendrian surgery on $L^+_{n+1}$ is also not isotopic to any of those counted by $\Psi$, we can conclude that the result of Legendrian surgery on $L^+_{n+1}$ is isotopic to $(M(n), \xi_n)$.  Although we will not need it here, we claim that this is also true for $n = -4$.  Indeed, recall that the upper bound of one tight contact structure counted by $\Phi$ applies to any Legendrian surgery on a knot whose standard neighborhood thickens arbitrarily. Since $L^\pm_{-3}$ are such knots, by Theorem~\ref{legendrianapproximations}, the results of Legendrian surgery on $L^\pm_{-3}$ must also be isotopic. (In fact, this argument suffices for any $n$.)

If we need to do a non-integral contact surgery on $L^\pm_{n+1}$, then by following Section~\ref{sec:surgery}, we can turn it into a Legendrian surgery on $L_0$ and a further contact surgery on $L_1$, where both $L_0$ and $L_1$ are isotopic to $L^\pm_{n+1}$ (since our contact surgery coefficient is between $-1$ and $0$). The Legendrian knot $L_1$ corresponds to a rationally null-homologous Legendrian knot $L$ of order $|n|$ in $(M(n), \xi_n)$, which is the result of Legendrian surgery on $L_0$.  Using \cite[Lemma~6.4]{Conway:transverse}, we can work out its rational Thurston--Bennequin number and rational rotation number, and we get
\[ tb_{\Q}(L) = 1 \mbox{ and } rot_{\Q}(L) = \pm 1. \]
Let $L^\pm$ denote the two Legendrian knots that we get, distinguished by $rot_{\Q}(L^\pm) = \pm 1$.  We refer the reader to \cite{BE:rational} for more details on these invariants and just recall here that a stabilization of a rationally null-homologous knot affects $rot_{\Q}$ the same way as it would affect $rot$ of an integrally null-homologous knot, namely by adding or subtracting $1$ (depending on the sign of the stabilization).

\begin{prop} \label{negative-surgery phi-count}
Contact surgery on $L^\pm$ in $(M(n), \xi_n)$ creates at least $\Phi(r)$ isotopy classes of tight contact structures on $M(r)$, when $r < 0$, distinguished by their Heegaard Floer contact invariants.
\end{prop}
\begin{proof}
Let $r \in (n, n+1)$, let $s = n+1-r$, and let $L^\pm$ be the two Legendrian knots in $(M(n), \xi_n)$, as described above.  Smooth $r$--surgery on $L^\pm_{n+1}$ is equivalent to smooth $n$--surgery on $L^\pm_{n+1}$ followed by smooth $(n+\frac{2s-1}{s-1})$--surgery on a push-off of $L^\pm_{n+1}$ that links it $n+1$ times.  Converting this into contact surgeries, we see that the different choices of contact $(-s)$--surgery on $L^\pm_{n+1}$ give the same contact structures as the different choices of contact $(-\frac{s}{1-s})$--surgery on the knots $L^\pm \subset (M(n), \xi_n)$.

Different choices of this latter contact surgery, that is, different stabilization choices when converting it into a sequence of Legendrian surgeries on $L_1, \ldots, L_m$, build potentially different Stein structures on the same smooth cobordism $W$.  To distinguish the Stein structures, which allows us to apply Theorem~\ref{steincobordism}, we evaluate the first Chern class of the Stein structures on certain elements of $H_2(W; \Z)$.  The elements that we consider come from the union of a rational Seifert surface $\Sigma_i$ for $L_i$ glued to $|n|$ copies of the core $C_i$ of the $2$--handle attached to $L_i$.  In the Stein structure $J$ on $W$ induced by attaching Stein $2$--handles to each $L_i$, we calculate
\[ \big\langle c_1(J), \left[\Sigma_i \cup |n| \cdot C_i\right]\big\rangle = |n|\cdot rot_{\Q}(L_i), \] see \cite[Lemma~4.1]{LW:rationaltb}.
It follows from this and from Theorem~\ref{steincobordism} that we can get a lower bound on the number of distinct tight contact structures by counting the number of possible values of $rot_{\Q}$ of each $L_i$.

We perform this count just as we counted the number of stabilization choices in Proposition~\ref{positive-surgery count}, except that our count for $L_1$ is different, since we are able to choose either $L^+$ or $L^-$ to start off with. It is easy to see that this will only add $1$ to the number of possible values of $rot_{\Q}(L_1)$.  In order to more easily describe the count, we imagine that $L^\pm$ come from stabilizations of a Legendrian knot $L'$ that has $rot_{\Q}(L') = 0$, and we see that the number of contact $(-\frac{s}{1-s})$--surgeries on $L^\pm$ is equal to the number of contact $(-1-\frac{s}{1-s})$--surgeries there would be on $L'$.  Since
\[ -1-\frac{s}{1-s} = -\frac{1}{1-s}, \]
this count is equal to $\Phi(1-s)$, since $1-s \in (0, 1)$.  However, since $r \equiv 1-s\modp{1}$, by definition, $\Phi(r) = \Phi(1-s)$.  Thus, we have counted at least $\Phi(r)$ different sets of values for the $rot_{\Q}(L_i)$, and hence have constructed at least $\Phi(r)$ different isotopy classes of tight contact structures on $M(r)$, distinguished by their Heegaard Floer contact invariants.
\end{proof}

\begin{remark} By considering $L^\pm_n$ for any $n \in \Z$, the methods in this section construct at least $\Phi(r)$ tight contact structures on $M(r)$, for any $r \in \Q$. From Section~\ref{sec:upperbound}, we know that all the tight contact structures on $M(r \geq 1)$ arise via gluing solid tori to certain knot complements, and from Section~\ref{sec:positive-surgery-lower-bound}, we know that the upper bound is attained. It is not hard to identify these complements as the complements of $L^\pm_n$, and thus the construction in this section actually gives all $2\Phi(r)$ tight contact structures on $M(r \geq 1)$.
\end{remark}

\section{Fillability and Universal Tightness}
\label{sec:properties}

In this section, we will discuss the symplectic fillability (Theorem~\ref{fillable}) and universal tightness (Theorem~\ref{universally tight}) of the tight contact structures we have constructed on $M(r)$ for $r \in \mathcal R$.

\subsection{Symplectic Fillability}

\begin{proof}[Proof~of~Theorem~\ref{fillable}]
Since the contact manifolds in Sections~\ref{sec:positive-surgery-lower-bound}~and~\ref{sec:psi-lower-bound} are constructed via Legendrian surgery diagrams, they are Stein fillable, and hence also strongly fillable, by \cite{Eliashberg:stein}.  It remains to understand the contact manifolds constructed in Section~\ref{sec:phi-lower-bound}.  Additionally, when $r \in (n, n+1)$ is negative and non-integral, the contact structures on $M(r)$ are constructed via Legendrian surgery on a Legendrian link in $(M(n), \xi_n)$, the result of Legendrian surgery on $L^\pm_{n+1}$.  Since Legendrian surgery preserves Stein and strong fillability by \cite{Eliashberg:stein} and \cite{Weinstein}, respectively, it is enough to show that $(M(n), \xi_n)$ is Stein fillable for $n \in \{-9, -8, \ldots, -1\}$ and is strongly fillable for integers $n \leq -10$.

\begin{lem}
The contact manifolds $(M(n), \xi_n)$ are strongly fillable for any integer $n \leq -1$.
\end{lem}
\begin{proof}
We start by constructing abstract open books that support $(M(n), \xi_n)$, for $n \leq -1$.  For a positive integer $m$, let $\Sigma_m$ be a genus--$1$ surface with $m$ boundary components.  Then the open book for $S^3$ with binding the figure-eight knot is given by $(\Sigma_1, D_\alpha^{-1} D_\beta)$, where $\alpha$ and $\beta$ are as in Figure~\ref{negative-surgery-open-book}, and $D_\gamma$ denotes a positive Dehn twist about the simple closed curve $\gamma$. Then by Theorem~\ref{constructing open books}, the open book $(\Sigma_m, D_\alpha^{-1} D_\beta \Delta)$ supports the contact manifold $(M(-m), \xi_{-m})$, where $\Delta$ is the composition of positive Dehn twists about each boundary component of $\Sigma_m$.

\begin{figure}[htbp]
\begin{center}
\begin{overpic}[scale=2.5,tics=20]{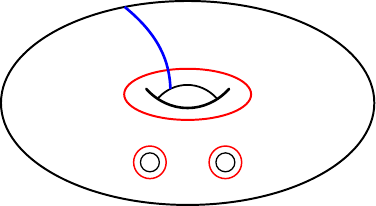}
\put(120,120){$\alpha$}
\put(80,80){$\beta$}
\put(126,26){\LARGE $\cdots$}
\end{overpic}
\caption{An abstract open book for ($M(-m), \xi_{-m}$). The red curves represent positive Dehn twists and the blue curve represents a negative Dehn twist.}
\label{negative-surgery-open-book}
\end{center}
\end{figure}

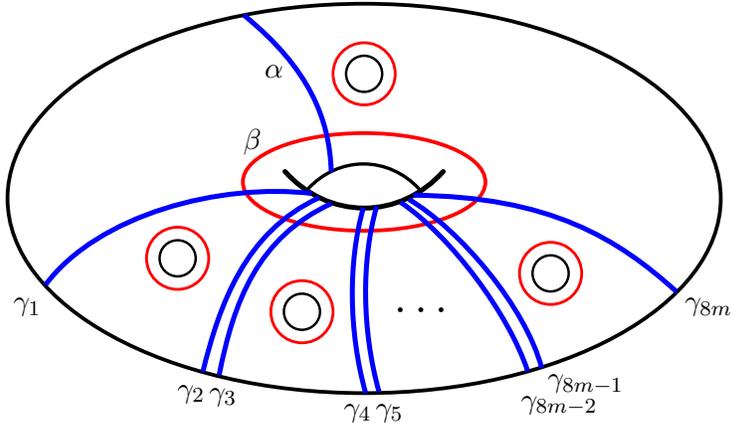
\begin{figure}[htbp]
\begin{center}
\begin{overpic}[scale=2.5,tics=20]{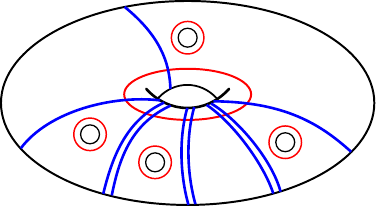}
\put(98,120){$\alpha$}
\put(90,93){$\beta$}
\put(3,32){$\gamma_1$}
\put(65,-1){$\gamma_2$}
\put(77,-2){$\gamma_3$}
\put(128,-8){$\gamma_4$}
\put(140,-8){$\gamma_5$}
\put(205,3){$\gamma_{8m-1}$}
\put(195,-5){$\gamma_{8m-2}$}
\put(257,32){$\gamma_{8m}$}
\put(147,28){\LARGE $\cdots$}
\end{overpic}
\vspace{1.5cm}
\caption{An abstract open book for ($M'_m, \xi'_m$) with page $\Sigma_{4m+1}$. Positive Dehn twists will be performed along red curves and negative Dehn twists along blue curves.}
\label{strongly-fillable-open-book}
\end{center}
\end{figure}

Consider now the surface and curves depicted in Figure~\ref{strongly-fillable-open-book}. For $m \geq 1$ an integer, we consider the abstract open book
\[ (\Sigma_{4m+1}, D_\alpha^{-1} D_\beta  D_{\gamma_1}^{-1} D_{\gamma_2}^{-1} \cdots  D_{\gamma_{8m-1}}^{-1} D_{\gamma_{8m}}^{-1}  \Delta). \]
The contact manifolds $(M'_m, \xi'_m)$ supported by these open books are known to be strongly fillable (and not Stein fillable), by work of Ghiggini \cite{Ghiggini:strongnotstein} (the open books appear in this form explicitly in \cite[Section~4.7.2]{VHM}).  Note that without the Dehn twists around the $\gamma_i$, the open book would be the same as the one given above that supports $(M(-4m-1), \xi_{-4m-1})$. We can add positive Dehn twists to the monodromy to cancel out the negative ones; this gives a new contact manifold that is built from the previous one by a sequence of Legendrian surgeries.  Thus, we can build a Stein cobordism from $(M'_m, \xi'_m)$ to $(M(-4m-1), \xi_{-4m-1})$.  Since strong fillability is preserved under Legendrian surgery by \cite{Weinstein}, we have shown that $(M(-4m-1), \xi_{-4m-1})$ is strongly fillable for any $m \geq 1$.

We briefly recall the {\em capping off} construction. Given an open book for a contact manifold $(M, \xi)$ with page $\Sigma_{m+1}$, where $m \geq 1$, one can arrive at an open book with page $\Sigma_m$ by gluing a disc onto one of the boundary components and extending the monodromy over the disc via the identity; this is called capping off.  By \cite[Theorem~2.1]{GS} (see also \cite[Theorem~5]{Wendl:nonexact}), there exists a symplectic cobordism from $(M, \xi)$ to $(M', \xi')$ (the contact manifold supported by the capped-off open book) that is strongly concave at $(M, \xi)$ and weakly convex at $(M', \xi')$.  Furthermore, by \cite{OO}, if $M'$ is a rational homology sphere, then we can perturb the symplectic structure such that it is strongly convex at $(M', \xi')$.

In our case, when $m \geq 1$, capping off any boundary component of $(\Sigma_{m+1}, D_{\alpha}^{-1}D_\beta\Delta)$ gives an open book that supports $(M(-m), \xi_{-m})$.  Since $M(m)$ is a rational homology sphere for all $m \geq 1$, we can glue the strongly symplectic cobordism from the previous paragraph to a strong filling of $(M(-m-1),\xi_{-m-1})$ to get a strong filling of $(M(-m), \xi_{-m})$.  Since we already know that $(M(-4m-1), \xi_{-4m-1})$ is strongly fillable, for all $m \geq 1$, we conclude that $(M(-m), \xi_{-m})$ is strongly fillable for all $m \geq 1$.
\end{proof}

\begin{lem}
The contact manifolds $(M(n), \xi_n)$ are Stein fillable, for any integer $-9 \leq n \leq -1$.
\end{lem}
\begin{proof}
Let $-9 \leq n \leq -1$ be an integer.  In \cite{KO}, Korkmaz and Ozbagci give explicit factorizations of $\Delta$ on $\Sigma_{-n}$ into a product of positive Dehn twists about non-separating curves.  Up to cyclic permutation of the Dehn twists, their factorizations take the form $\phi D_\alpha$, where $\phi$ is a product of positive Dehn twists.  Thus, the mapping class group $D_\alpha^{-1}D_\beta\Delta$ can be written as
\[ D_\alpha^{-1}D_\beta\Delta = \Delta D_\alpha^{-1}D_\beta = \phi D_\beta, \] which is the product of only positive Dehn twists around non-separating curves, and hence this describes a Stein filling of $(M(n), \xi_n)$.
\end{proof}

This concludes the proof of Theorem~\ref{fillable}.
\end{proof}

\subsection{Universal Tightness} 
\begin{proof}[Proof of Theorem~\ref{universally tight}]
	We can construct universally tight contact structures by perturbing taut foliations, by work of Eliashberg and Thurston \cite{ET}. Since Roberts \cite{Roberts} showed that $M(r)$ supports a taut foliation for any $r \in \Q$, it follows that $M(r)$ admits at least one universally tight contact structure for each $r$.
	
	Our approach to identifying the universally tight contact structure(s) will be by process of elimination.  First, note that for all $r \in \Q$, the surgery dual $K_r$ of the figure-eight knot in $M(r)$ --- which is the core of the surgery torus $N$, using notation from Section~\ref{sec:upperbound} ---  represents a non-torsion element in $\pi_1(M(r))$, and thus a neighborhood of $K_r$ lifts to the universal cover of $M(r)$. Thus, if $N$ carries a virtually overtwisted contact structure, then the overtwisted universal cover of $N$ embeds into the universal cover of $M(r)$, and hence the contact structure on $M(r)$ cannot be universally tight.
	
	By the classification work of Honda \cite[Proposition~5.1(2)]{Honda:classification1}, we know that there is a unique universally tight contact structure on a solid torus with integral meridional slope and dividing curve slope $\infty$, and exactly two universally tight contact structures when the meridional slope is non-integral. These are composed of basic slices of a consistent sign glued to a solid torus with a unique tight contact structure.
	
	Now, using the decomposition from Section~\ref{sec:upperbound} of tight contact structures on $M(r)$ for positive $r \in \mathcal R$ into tight contact structures on $C(\infty)$ glued to a solid torus, it follows that for positive $r \in \mathcal R$, there are at most two universally tight contact structures on $M(r)$ for $r \in \N$ and at most four for $r \not\in \N$. These correspond to the contact structures coming from all negative or all positive stabilizations on $L$ in Figure~\ref{positive-surgery legendrian-diagram}. Since a $180^\circ$ rotation of the surgery diagram in Figure~\ref{positive-surgery legendrian-diagram} switches signs of stabilizations on $L$ and on $L'$, this symmetry induces contactomorphisms of these contact structures in pairs. Thus for integral $r$, there are exactly two universally tight contact structures on $M(r)$, and for non-integral $r$, there are either two or four universally tight contact structures on $M(r)$.
	
	When $r < -3$, there are at most two universally tight contact structures counted by $\Psi$. These arise by choosing all negative or all positive stabilizations on contact $(r+3)$--surgery on the Legendrian knot in Figure~\ref{legendrian-figure-eight}, and by the argument above, these contact structures are contactomorphic.
	
	For any $r < 0$, there is at most one universally tight contact structure counted by $\Phi$ for integral $r$, and at most two for non-integral $r$.  These arise by choosing all negative or all positive stabilizations on contact surgery on $L^\pm_{n+1}$ (see Section~\ref{sec:phi-lower-bound}), which are knots in the overtwisted contact structure $(S^3, \xi^{OT}_1)$. We claim that there is an involution of this overtwisted contact manifold that switches $L^+_{n+1}$ and $L^-_{n+1}$ and switches the signs of the stabilization choices of surgery. Indeed, $(S^3, \xi^{OT}_1)$ can be described by contact $(+1)$--surgery on each component of the standard Legendrian Hopf link. This surgery diagram admits a $180^\circ$ rotational symmetry, inducing an involution on the manifold after surgery that evidently switches signs of stabilizations of knots in the complement of the Hopf link, as claimed.
	
	The theorem will now follow from the preceding discussion and Proposition~\ref{Psi-virtually-OT}.
\end{proof}

	\begin{prop} \label{Psi-virtually-OT}
		All contact structures on $M(r)$ for $r \in \mathcal{R}$ counted by $\Psi$ are virtually overtwisted.
	\end{prop}
	\begin{proof}
	Let $r = -\frac pq$ for a pair of co-prime positive integers $p, q$. Note that $M(r)$ has a $p$-fold cover $\widetilde{M} = M'_{-1/q}(K')$, where $M'$ is the $p$-fold branched cover of $S^3$ branched over the figure-eight knot $K$, and $K'$ is the preimage of $K$. Consider the contact structure on $\widetilde{M}$ coming from pulling back a contact structure on $M(r)$ that is one counted by $\Psi$.  This cover contains a $p$-fold cyclic cover of $C(s)$, for $s \in (r, -4)$. We will show below that this cyclic covering of $C(s)$ is overtwisted for any $p \geq 4$, and hence that the contact structure on $\widetilde M$ is overtwisted.
	
		\begin{figure}[htbp]
			\begin{center}
			\vspace{0.5cm}
			\begin{overpic}[scale=0.85,tics=20]{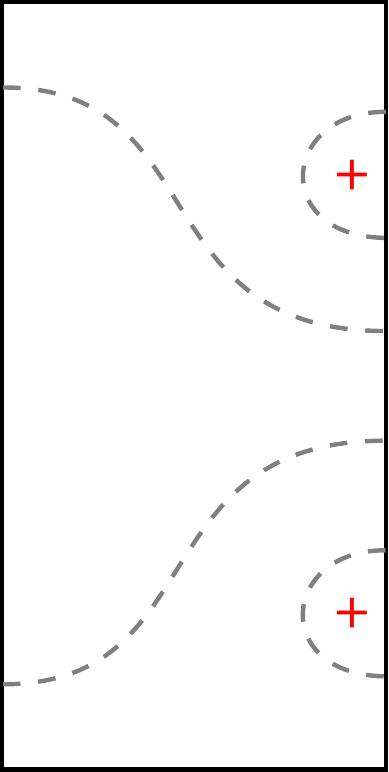}
			\put(-10,-12){\footnotesize $s=-1$}
			\put(80,-12){\footnotesize $s=-\frac 34$}
			\end{overpic}
			\vspace{0.7cm}
			\caption{Possible dividing curves on an annulus in $\widetilde{C_4}(-1)\setminus \widetilde{C_4}(-\frac 34)$.}
			\label{annulus-1}
			\end{center}
		\end{figure}
	
	Specifically, consider a contact structure on $C(s)$ for $s < -4$ that thickens to $-3$ and no further. We will show that for each integer $p \geq 4$, the $p$-fold cyclic covering of $C(s)$ is overtwisted.  It is sufficient to consider when $p$ is $4$ or an odd prime, since any composite-order cover is a composition of covers, and the required property will persist under the composition.
	
	We construct covers by cutting $C(s)$ along $\Sigma$ to get $\Sigma \times [0,1]$, letting $\Sigma \times [i, i+1]$ be $\phi^{-i}\left(\Sigma \times[0,1]\right)$, and gluing appropriately. Since $C(s)$ has two dividing curves of slope $s$ on the boundary, the $p$-fold cover will have dividing curves of slope $s/p$ on the boundary (if $p$ divides $s$, then there will be more than two dividing curves of slope $s/p$). Therefore, we denote the $p$-fold cover of $C(s)$ by $\widetilde{C_p}(\frac sp)$.

		{\bf Case $\bm{p = 4}$.} $C(s)$ thickens to $C(-4)$, which thickens to $C(-3)$. Therefore, the $4$-fold cover $\widetilde{C_4}(\frac s4)$ thickens to $\widetilde{C_4}(-1)$, which thickens to $\widetilde{C_4}(-\frac34)$. If the contact structure on the cover is tight, then the $T^2 \times I$ involved in the latter thickening would be a basic slice (since the dividing curve slopes are $-1$ and $-3/4$), so a vertical annulus would have bypasses only of one sign. By a relative Euler class computation, there must be two bypasses for $\bd \widetilde{C_4}(-\frac34)$; see Figure~\ref{annulus-1} for an example.
		
		There are three positions on the annulus where the bypasses can be located, but we claim that there are isotopic annuli in the basic slice realizing any of the combinations of bypass locations.  Indeed, given one dividing curve setup, after pushing over one of the bypasses, the other potential bypasses are trivial bypasses, so they always exist, by \cite[Lemma~1.8]{Honda:gluing}.
		
		Now we glue these annuli to copies of $\Sigma$ in $\widetilde{C_4}(-\frac34)$, where $\Sigma$ has dividing curves as in the left figure of Figure~\ref{negative-punctured-torus-1}.  The dividing set of $\Sigma$ union one of the annuli will contain a contractible dividing curve, and hence the contact structure on $\widetilde{C_4}(\frac s4)$ is overtwisted.
		
		\begin{figure}[htbp]
			\begin{center}
			\begin{overpic}[scale=0.85,tics=20]{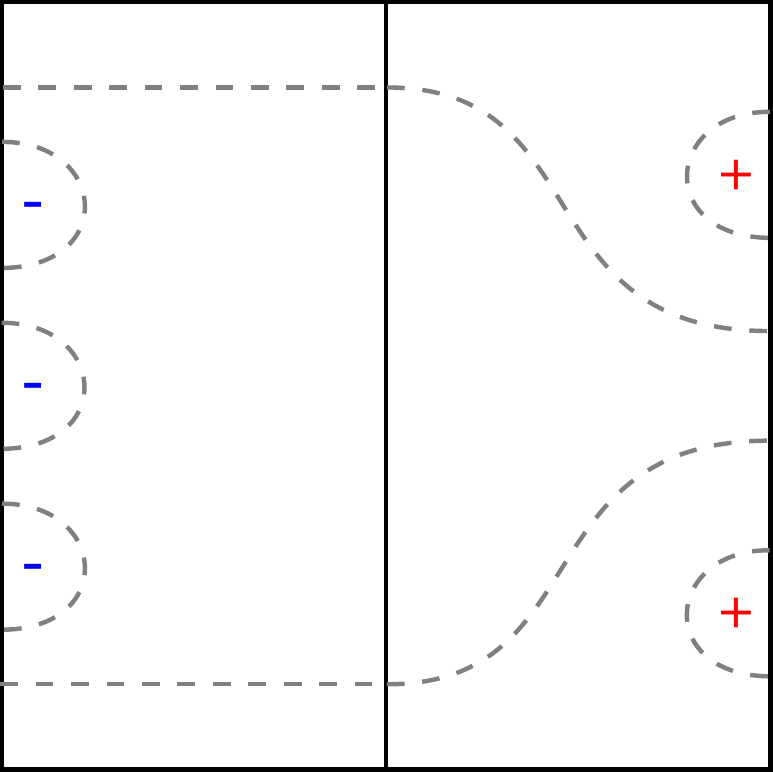}
			\put(-10,-12){\footnotesize $s=-\frac 43$}
			\put(80,-12){\footnotesize $s=-1$}
			\put(175,-12){\footnotesize $s=-1$}
			\end{overpic}
			\vspace{0.7cm}
			\caption{An annulus in $T^2\times[0,1]$ contained in $\widetilde{C_3}(\frac s3)$.}
			\label{annulus-2}
			\end{center}
		\end{figure}
		
		\begin{figure}[htbp]
			\begin{center}
			\begin{overpic}[scale=0.85,tics=20]{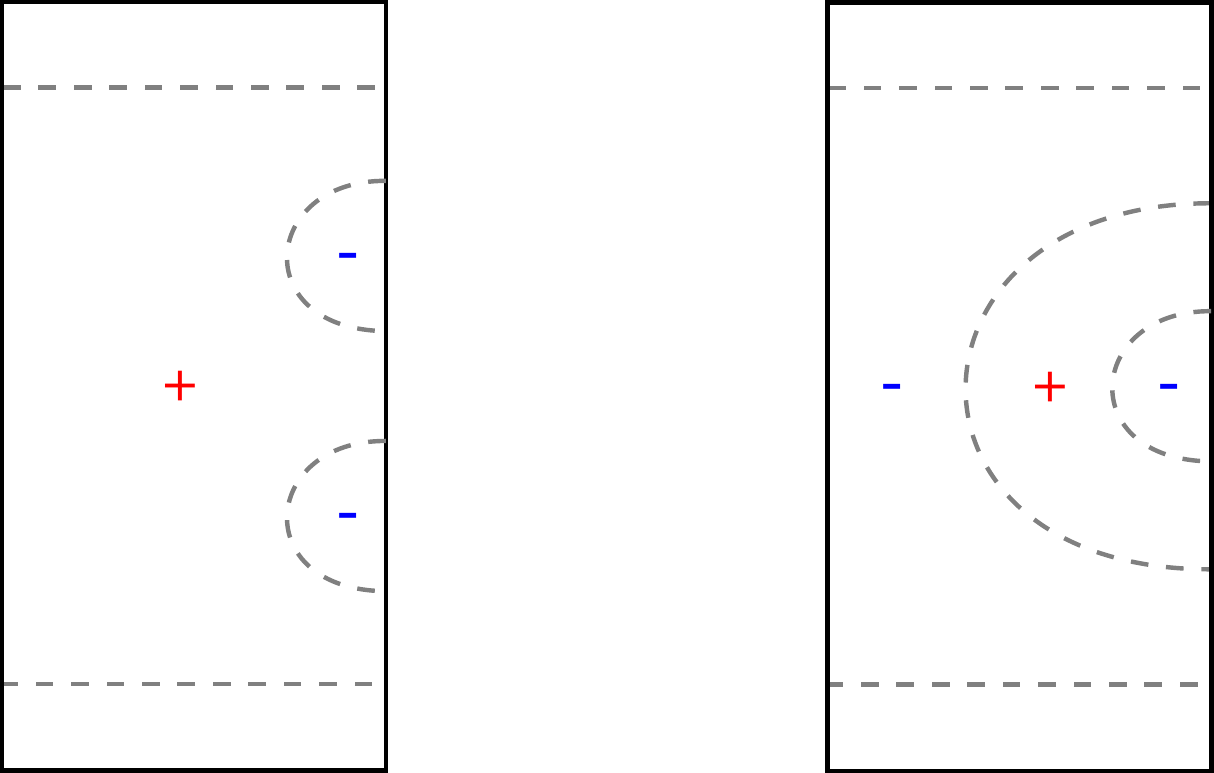}
			\put(-10,-12){\footnotesize $s=-1$}
			\put(80,-12){\footnotesize $s=-1$}
			\put(190,-12){\footnotesize $s=-1$}
			\put(280,-12){\footnotesize $s=-1$}
			\end{overpic}
			\vspace{0.7cm}
			\caption{Some possible annuli in bypass layers which reduce the number of dividing curves.}
			\label{annulus-3}
			\end{center}
		\end{figure}

		\begin{figure}[htbp]
			\begin{center}
			\begin{overpic}[scale=0.85,tics=20]{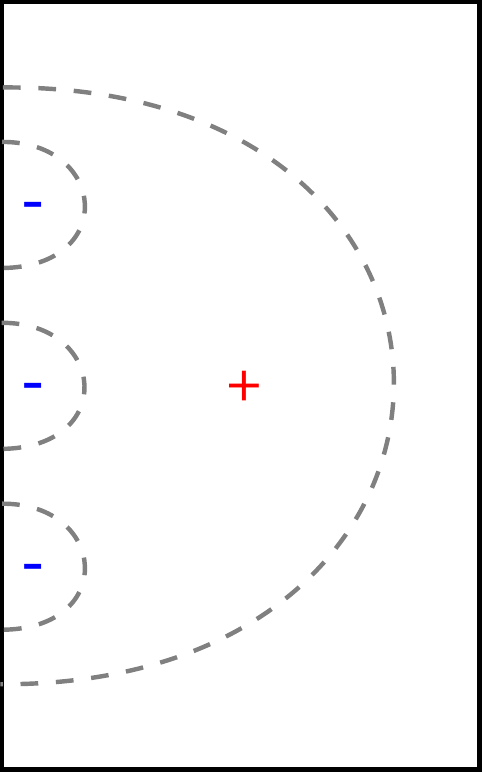}
			\put(-10,-12){\footnotesize $s=-\frac 43$}
			\put(105,-12){\footnotesize $s=0$}
			\end{overpic}
			\vspace{0.7cm}
			\caption{The annulus $A$ in $\widetilde{C_3}(-\frac 43)$.}
			\label{annulus-4}
			\end{center}
		\end{figure}

		{\bf Case $\bm{p=3}$.} We will show that the $3$-fold cover of $C(s)$ is virtually overtwisted, and that any further cyclic cover is overtwisted. Since $C(-4)$ thickens to $C(-3)$, which has two dividing curves of slope $-3$ on the boundary, $\widetilde{C_3}(-\frac 43)$ thickens to $\widetilde{C_3}(-1)$, which has six dividing curves of slope $-1$ on its boundary. This $\widetilde{C_3}(-1)$ contains a copy of $\Sigma$ with dividing set as shown in the left drawing of Figure~\ref{negative-punctured-torus-1}, and we can reduce the number of dividing curves on the boundary by attaching bypasses. More precisely, $\widetilde{C_3}(-\frac43) \setminus \widetilde{C_3}(-1)$ is a copy of $T^2\times[0,1]$ with dividing curve slopes $s_0=-\frac 43$ and $s_1=-1$ where $T^2\times\{1\}$ has six dividing curves. By \cite[Theorem~2.2]{Honda:classification1}, this can be factored into $T^2\times[0,\frac 12]$ with $s_0=-\frac 43$ and $s_{1/2}=-1$, and $T^2\times[\frac 12, 1]$ with $s_{1/2}=-1$ and $s_1=-1$ where $T^2\times\{\frac 12\}$ has two dividing curves; see Figure~\ref{annulus-2} for an example. There are several possible dividing sets on the annulus in $T^2\times[\frac 12, 1]$ (two are shown in Figure~\ref{annulus-3}), but the relative Euler class has only three possible evaluations: $0$ and $\pm2$.
		
		If the contact structure in the cover is tight, then $T^2\times[0, \frac 12]$ would be a basic slice with dividing curve slopes $-\frac43$ and $-1$, and the only possible values of the relative Euler class evaluted on the annulus are $\pm3$: this can be seen either by inspecting possible dividing curves or by changing coordinates to a standard basic slice and computing there.  If the relative Euler class evaluated to $0$ on the annulus in $T^2\times[\frac 12, 1]$, then since $T^2\times[0,1]$ covers a basic slice, the relative Euler class evaluates on the entire annulus to $\pm1$, and hence would evaluate on the annulus in $T^2 \times [0, \frac12]$ to $\pm1$, which is a contradiction.
		
		So suppose the relative Euler class evaluates to $2$ on the annulus in $T^2\times[\frac 12, 1]$ (the same argument will also work for the $-2$ case). If we glue the annulus in $T^2\times[\frac 12, 1]$ to $\Sigma$ in $\widetilde{C_3}(-1)$, the relative Euler class will evaluate to $2$ on the union, so the surface contains a positive bypass for $T^2 \times \{\frac12\}$. Now since the relative Euler class evaluates to $\pm 1$ on the annulus in $T^2\times[0,1]$, it must evaluate to $-1$ or $-3$ on the annulus in $T^2 \times [0, \frac12]$. Since we know from above that it cannot be $-1$, it must be $-3$. By considering the possible dividing sets on a vertical annulus in this $T^2 \times [0, \frac12]$, it is not hard to see that the annulus contains three negative bypasses for $T^2\times\{0\}$. After gluing $T^2\times[0,1]$ to $\widetilde{C_3}(-1)$ and gluing the annulus to $\Sigma$, we either obtain a contractible dividing curve (and hence an overtwisted disc) or an annulus $A$ with dividing set as shown in Figure~\ref{annulus-4}.
		
		\begin{figure}[htbp]
			\begin{center}
			\begin{overpic}[scale=0.8,tics=20]{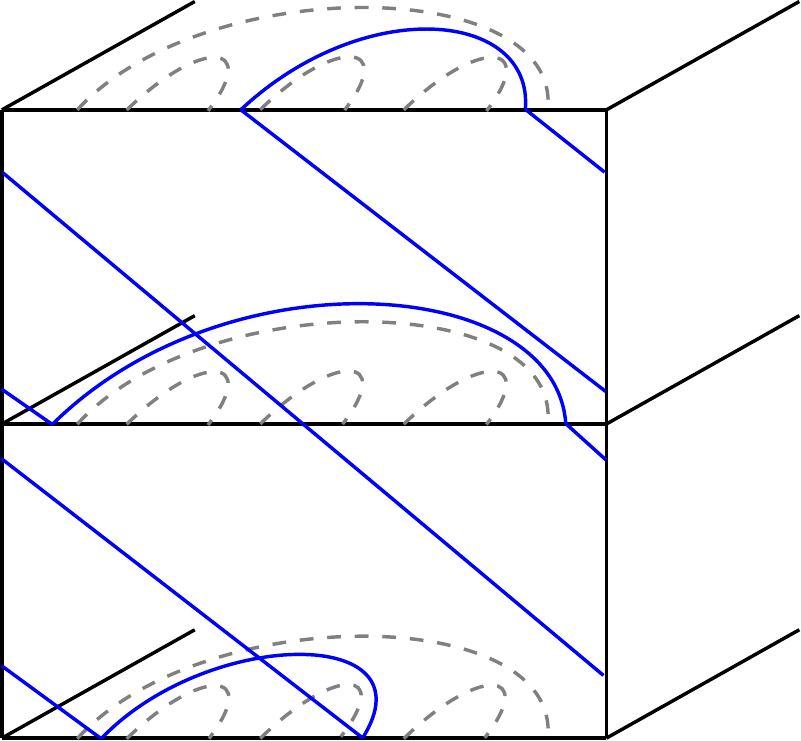}
			\end{overpic}
			\vspace{0.4cm}
			\caption{After realising the blue curve as Legendrian, it bounds an overtwisted disc in $A \times [0,2] \subset \Sigma \times [0,2]$. The dotted lines are dividing curves on three copies of $A$; the unshown dividing curves on $\partial\Sigma \times [0,2]$ are parallel to the diagonal blue arcs.}
			\label{overtwisted-cover}
			\end{center}
		\end{figure}
		
		In the latter case, we will show that any further cyclic cover of $\widetilde{C_3}(-\frac43)$, and hence of $\widetilde{C_3}(\frac s3)$, is overtwisted.  To this end, consider $A \times [0,2] \subset \Sigma \times [0,2]$: this manifold embeds in all further cyclic covers of $\widetilde{C_3}(-\frac 43)$.  As shown in Figure~\ref{overtwisted-cover}, we can find a curve in $A \times [0,2]$ that bounds a disc and does not intersect any dividing curves on $A \times [0,2]$. After Legendrian realizing this curve, it will bound an overtwisted disc, and thus $\Sigma \times [0,2]$ is overtwisted (see \cite[Proposition~5.1]{Honda:classification1} for similar). Thus, any cyclic covering of $\widetilde{C_3}(-\frac 43)$ is overtwisted.

		{\bf Case $\bm{p=5}$.} The $5$-fold cover $\widetilde{C_5}(\frac s5)$ contains a boundary-parallel $T^2\times[0,1]$ with dividing curve slopes $s_0=-\frac 45$ and $s_1=-\frac 35$. Since $-\frac 45 < -\frac 23 < -\frac 35$, we can factor $T^2\times[0,1]$ into $T^2\times[0,\frac 12]$ with $s_0=-\frac 45$ and $s_{1/2}=-\frac 23$, and $T^2\times[\frac 12,1]$ with $s_{1/2}=-\frac 23$ and $s_1 = -\frac 35$. If the contact structure on the cover is tight, then $T^2 \times [\frac 12, 1]$ would be a basic slice, and the relative Euler class would evaluate to $\pm 1$ on the annulus in $T^2 \times [\frac12, 1]$. There are several possible dividing sets for the annulus, but as in the $p=4$ case, we can realize an annulus that has (at least) two bypasses for $T^2 \times \{1\}$, and locate them where desired. After gluing $T^2\times[0,1]$ to $\widetilde{C_5}(-\frac 35)$ and gluing the annulus to $\Sigma$, we obtain a contractible dividing curve, and hence $\widetilde{C_5}(\frac s5)$ is overtwisted.

		{\bf Case $\bm{p \geq 7}$.} The $p$-fold cyclic cover $\widetilde{C_p}(\frac sp)$ contains a boundary-parallel $T^2\times[0,1]$ with dividing curve slopes $s_0=-\frac 4p$ and $s_1=-\frac 3p$. We factor this into $T^2\times[0,\frac 12]$ with $s_0=-\frac 4p$ and $s_{1/2}=-\frac 1k$, and $T^2\times[\frac 12,1]$ with $s_{1/2}=-\frac 1k$ and $s_1 = -\frac 3p$, for a positive integer $k$ such that $-\frac 4p < -\frac 1k < -\frac 3p$.
		
		Consider a vertical annulus in $T^2 \times [0,1]$. As in the case $p = 4$, there are two bypasses for $T^2 \times \{1\}$, and similarly, we can find an annulus with the bypasses located such that after gluing $T^2\times[0,1]$ to $\widetilde{C_p}(-\frac 3p)$ and gluing the annulus to $\Sigma$, we obtain a contractible dividing curve. Thus, $\widetilde{C_p}(\frac sp)$ is overtwisted.
	\end{proof}

\bibliography{references}{}
\bibliographystyle{plain}
\end{document}